\documentclass[11pt, oneside,reqno]{article}
\usepackage[top=2.54cm, bottom=3cm, right=3cm, left=3cm]{geometry}
\usepackage[utf8]{inputenc}
\usepackage[english]{babel}
\usepackage{amsmath, mathtools, amsthm, amssymb}
\usepackage{dsfont}
\usepackage{changepage}
\usepackage{caption, subcaption}
\usepackage{hyperref}
\usepackage{float}
\usepackage{titling}
\usepackage{cite}
\usepackage[titletoc,toc,title]{appendix}

\DeclareMathOperator*{\argmin}{argmin}
\DeclareMathOperator*{\argmax}{argmax}
\DeclareMathOperator\arctanh{arctanh}

\theoremstyle{plain}
\newtheorem{thm}{Theorem}[section] 

\newtheorem{corollary}{Corollary}[thm]
\newtheorem{lemma}[thm]{Lemma}
\newtheorem{prop}[thm]{Proposition}

\theoremstyle{definition}
\newtheorem{defn}[thm]{Definition} 

\makeatletter
\def\thmheadassumption#1#2#3{%
	\thmname{#1}\thmnumber{\@ifnotempty{#1}{ }\@upn{#2}}%
	\thmnote{ {\the\thm@notefont[#3]}}}
\makeatother

\newtheoremstyle{assumption}
{}
{}
{\itshape}
{}
{\bfseries}
{.}
{ }
{\thmheadassumption{#1}{(A#2)}{#3}}

\theoremstyle{assumption}
\newtheorem{assumption}{}

\newcommand{\assumptionref}[1]{(A\ref{#1})}

\theoremstyle{remark}
\newtheorem*{rmk}{Remark}
\newtheorem*{notation}{Notation}

\newcommand{\R}{\mathbb{R}}
\newcommand{\N}{\mathbb{N}}

\newcommand{\Prob}{\mathbb{P}}
\newcommand{\E}{\mathbf{E}}
\newcommand{\dx}{\, \mathrm{d}x}

\newcommand{\ds}{\, \mathrm{d}s}
\newcommand{\dt}{\, \mathrm{d}t}
\newcommand{\di}{\, \mathrm{d}}
\newcommand{\norm}[1]{\left\lVert #1 \right\rVert}
\newcommand{\probspace}{(\Omega, \mathcal{F}, \mathbb{P})}

\newcommand{\expect}[1]{\E \left[ #1 \right]}

\newcommand{\hilbert}{\mathcal H}

\makeatletter
\newcommand{\subalign}[1]{%
	\vcenter{%
		\Let@ \restore@math@cr \default@tag
		\baselineskip\fontdimen10 \scriptfont\tw@
		\advance\baselineskip\fontdimen12 \scriptfont\tw@
		\lineskip\thr@@\fontdimen8 \scriptfont\thr@@
		\lineskiplimit\lineskip
		\ialign{\hfil$\m@th\scriptstyle##$&$\m@th\scriptstyle{}##$\hfil\crcr
			#1\crcr
		}%
	}%
}
\makeatother

\begin{document}
	
	\title{Exponential contractions and robustness for approximate Wonham filters}
	
	\author{Samuel N. Cohen\thanks{Mathematical Institute, University of Oxford and Alan Turing Institute, {\tt cohens@maths.ox.ac.uk}.}
		\and
		Eliana Fausti\thanks{Mathematical Institute, University of Oxford, {\tt fausti@maths.ox.ac.uk}.}}
	
	\date{\today}
	
	\maketitle
	
	\begin{abstract}
		We consider the problem of estimating the state of a continuous-time Markov chain from noisy observations.
		We show that the corresponding optimal filter is strictly contracting pathwise, when considered in the Hilbert projective space, and give explicit deterministic and pathwise rates of convergence. Using this, we provide alternative proofs of the robustness of optimal filters, improving on known error estimates, and derive rigorous and computable error bounds for approximate filters.
	\end{abstract}
	
	\begin{adjustwidth}{0.6cm}{0.6cm}
		\begin{flushleft}
			\textbf{MSC}: 93E11, 62M05, 60J55. 
			
			\textbf{Keywords}: Nonlinear filtering, stability, model robustness, approximate filters, error bounds, Hilbert projective metric, local times.
		\end{flushleft}
	\end{adjustwidth}

	\section{Introduction}
	Estimating a random hidden process from incomplete, noisy observation is a common problem arising in engineering, signal processing, finance and many other applications. The general setting consists of a \textit{signal (or state) process} $X$ evolving in time (typically taken to be Markov), which cannot be measured directly, but needs to be estimated using the information given by an \textit{observation process} $Y$, whose dynamics depend on $X$. Computing and analyzing the optimal solution to this problem is the main objective of the theory of stochastic filtering.
	
	Stochastic filtering is a classical topic in stochastic analysis, and optimal filters have been derived in various contexts, e.g. in continuous or discrete time, with finite or infinite state-space, and so on. The setting of linear underlying dynamics, giving rise to the famous Kalman--Bucy filter \cite{kal60, kalbucy61}, was the first to be considered in continuous time, and is now well understood. Nonlinear filtering, on the other hand, presents challenges from both the theoretical and practical perspective (we refer to Bain and Crisan \cite{bain09} for an exposition of nonlinear filtering).
	
	The optimal nonlinear filter is the solution to a \textit{nonlinear stochastic} (depending on the context, \textit{partial}) \textit{differential equation} called the \textit{Kushner--Stratonovich} equation. In almost all practical applications, however, it cannot be computed directly:~for example, the model for $X$ and $Y$, which the filtering equations explicitly depend on, might have misspecified parameters, or be completely unknown. Moreover, even when the true model is available, solving the filtering equations numerically can be intractable, due to the high (in many cases, infinite) dimensional and non-local nature of the problem. More often than not, approximate filters, rather than the optimal filter, are employed. This begs the questions of whether or not these approximations are reliable, and how we can quantify their error with respect to the optimal filter. These questions are central in the study of the \textit{robustness} of the nonlinear filter.
	
	In this paper, we will focus on the case of finite state-space nonlinear filtering in continuous time. The optimal filter in this case is the solution to an SDE, and it is sometimes referred to as the \textit{Wonham filter} \cite{won65}. Significant progress on the robustness of the Wonham filter was made by Chigansky and Van Handel \cite{chigansky07} and Van Handel \cite{vanhandel_phd}. In these works, one considers the $L^1$-error between the true Wonham filter and the Wonham filter with misspecified model parameters. Following an approach that relies on computing bounds for the derivatives of the filter, they prove that the error stays finite over an infinite time horizon, and vanishes as the misspecified parameters are sent to the true ones. The method in \cite{chigansky07} could potentially be used to compute error bounds for some more general approximate filters, and not only those given by misspecification in the underlying model. However, the estimates in \cite{chigansky07} are not tight enough to provide useful quantitative bounds (see \cite[Remark~2.8]{chigansky07} and \cite[Remark~3.3.8]{vanhandel_phd}) so these results remain primarily of qualitative interest.
	
	The robustness question is invariably linked to the issue of \textit{stability} of the nonlinear filter. Compared to robustness, stability is only concerned with the error due to misspecification of the initial conditions of the filtering equations. If the error between the true filter and the `wrongly initialized' filter decays to zero as time passes, then the filter `forgets' the initial error and is called \textit{asymptotically stable}. This has consequences for approximate filtering: intuitively, in a discretized-time setting, if the nonlinear filter is stable, using an approximate filter is essentially the same as using the optimal filter, but introducing an  approximation error at each time step. If all the approximation errors are bounded, stability ensures that they are also `forgotten' as time goes on, so that the total error stays bounded, and we recover robustness-type estimates.
	
	In discrete-time, finite state-space nonlinear filtering, this is indeed how robustness estimates have been derived (see Budhiraja and Kushner \cite{Budhiraja98}, Le Gland and Mevel \cite{legland00_hmm} and Le Gland and Oudjane \cite{legland04}). The main difference with the continuous-time setting is that the stability estimates available in the literature for the Wonham filter are not strong enough to directly apply this kind of methodology. Indeed, to pull off this argument in continuous time, one would need exponential (or similar) contraction of the stability error. The first goal of this paper will be to establish such a contraction result for the Wonham filter. The second objective is to use our stability estimates to provide computable error bounds for approximate filters.
	
	\subsection{Discussion of known results}
	
	Filtering stability has been an  active field of study since the 1990's. A key paper in the literature is \cite{ocone96}, in which Ocone and Pardoux establish a relationship between the stability of the Kalman filter and detectability/stabilizability of the signal-observation linear control system. Their arguments for stability in the nonlinear setting, however, rely on a result by Kunita \cite{kunita71}, which was later found to contain a mistake (see Baxendale, Chigansky and Liptser \cite[Section 2]{bax-chiga-lip04} for a detailed explanation and a counterexample, and Budhiraja \cite{Budhiraja03} for an analysis of its relevance in the context of nonlinear filtering stability). The gap in Kunita's proof was addressed by Van Handel, who established the necessary conditions for the stability of the nonlinear filter in different settings (for ergodic signals in discrete and continuous time in \cite{vanhandel09_AP}, non-ergodic signal with compact state-space in \cite{vanhandel09_PTRF}, and with Polish state-space in \cite{vanhandel09_AAP}). More recently, in Kim, Mehta and Meyn \cite{kim2021conditional} and Kim and Mehta \cite{kim2021dual}, stability of the Wonham filter is shown to be equivalent to stabilizability of a dual control problem, in an extension of \cite{ocone96} to the nonlinear case. We refer the interested reader to Chigansky \cite{chigansky06survey} for an extensive review of nonlinear filtering stability results (in discrete time with finite state-space) and to \cite[Part~3]{oxford_handbook} for a broad collection of survey papers. Of particular relevance to our setting, Chigansky, Liptser and Van Handel \cite{chigansky_handbook} gives an accessible introduction to the stability results of \cite{vanhandel09_PTRF,vanhandel09_AP,vanhandel09_AAP}.
	
	While the above results guarantee stability of the filter in the strongest possible generality (and under the weakest possible assumptions), their qualitative nature makes them unsuitable for understanding general approximation errors. On the other hand, if one is willing to impose relatively strong ergodicity assumptions on the signal process, there are explicit decay rates available in the literature, at least for the particular case of the Wonham filter. Delyon and Zeitouni  \cite{delyon1991} introduced the study of the top Lyapunov exponent for the Wonham filter, and proved that it is negative under certain conditions on the model parameters. This method was expanded by Atar and Zeitouni \cite{atar-zeitouni-lyapunov97, atar-zeitouni97}, who, under a fairly strong mixing assumption for the signal, compute an explicit exponential decay rate for the stability error. Applying the techniques of \cite{atar-zeitouni97},  Baxendale, Chigansky and Lipster weakened the ergodicity assumptions slightly by proving a.s.~negativity of the decay rate if all the states of $X$ communicate \cite[Theorem~4.1]{bax-chiga-lip04} (although we lose an explicit rate). Finally, by working with the \textit{smoother process} (as described in e.g.~Liptser and Shiryayev \cite[Theorem~9.5]{lipster77}), they provide an explicit exponential rate of decay for a mixing signal in terms of its ergodic distribution \cite[Theorem~4.2]{bax-chiga-lip04}, and a \textit{non-asymptotic} exponential bound for the stability error \cite[Theorem~4.3]{bax-chiga-lip04}, with the same decay rate as \cite{atar-zeitouni-lyapunov97, atar-zeitouni97}.
	
	As far as we are aware, the bound in \cite[Theorem~4.3]{bax-chiga-lip04} is the only non-asymptotic bound available in the literature for the stability error of the Wonham filter in continuous time. The prefactor to the exponential decay term is proportional to the dimension of the Wonham SDE and the Radon--Nikodym derivatives of the true and the `wrong' initial distribution, and it is far too large for the bound to be useful from a quantitative point of view. Van Handel improves it significantly (although the result still remains far from a contraction), and the best estimate for the prefactor is found by combining \cite[Proposition~3.5]{chigansky07} and \cite[Corollary~2.3.2]{vanhandel_phd}. This stability result is central in the robustness analysis for the Wonham filter carried out in \cite{chigansky07}. On the other hand, the robustness results for the nonlinear filter in discrete time \cite{Budhiraja98,legland00_hmm, legland04} that we mentioned previously build on the work on stability by Atar and Zeitouni (in \cite{atar-zeitouni-lyapunov97, atar-zeitouni97} the analysis is carried out for both discrete and continuous time settings).
	
	The fundamental contribution of \cite{atar-zeitouni-lyapunov97, atar-zeitouni97} is to introduce the use of the Hilbert projective distance (see \cite[Eq.~9]{atar-zeitouni97}, or \eqref{eq:hilbert_metric} below) as a metric on the space of probability measures to carry out stability estimates for the nonlinear filter. A key advantage of using the Hilbert metric is that positive linear operators contract under this distance:~this is a result by Birkhoff (see \cite{birkhoff57} or \cite[Chapter~XVI]{birkhoff_lattice}). The work of Seneta on the product of positive linear operators (see \cite{seneta81}, or \cite[Chapter~3]{seneta_matrices}), which encompasses the analysis of the ergodicity of discrete-time Markov chains, is particularly illuminating for understanding how powerful a tool the Hilbert distance can be when used in the right context. Recalling that the generator of a discrete-time Markov chain is a stochastic matrix, Birkhoff's and Seneta's works make the stability results for discrete-time nonlinear filtering intuitively straightforward.
	
	Atar and Zeitouni provide asymptotic rates for the decay of the stability error of the filter, for both the discrete and continuous time case. Building on these ideas, and on Seneta's work, Le Gland and Mevel \cite{legland00_matrices, legland00_hmm}, and then Le Gland and Oudjane \cite{legland04} proved non-asymptotic and non-logarithmic stability bounds for the discrete time setting, conditional on a strong mixing assumption for the signal process. In \cite{legland04}, they are also able to tackle the issue of robustness of the nonlinear filter (in discrete time) and in particular they study the global error of interacting particle approximations to the filtering process. Our results in this paper follow roughly along the same lines, although in the continuous time setting. Moreover, our approach is fundamentally different from that in \cite{atar-zeitouni-lyapunov97, atar-zeitouni97, legland00_hmm, legland00_matrices, legland04};~the only common aspect is the use of the Hilbert metric in the stability analysis.
	
	\subsection{Main contributions and organization of the paper}
	Our first main contribution is an exponential contraction estimate for the stability error of the continuous-time, finite state-space nonlinear filter, in Hilbert distance (see Theorem~\ref{thm:contraction}). In fact, our statement is stronger, as we can prove contractivity of the hyperbolic tangent of the Hilbert distance, which directly implies the former. Both of these are, to the best of our knowledge, new results, which improve significantly on the quantitative estimates for the error available in the literature. We also present an alternative way to study the stability error of the continuous time filter in Hilbert distance, which does not rely on Atar and Zeitouni's arguments. Instead, inspired by Amari \cite{amari16}, we will introduce a change of coordinates from the probability simplex to $\R^n$, and study the evolution of the Wonham filter in the new coordinate system. As we will see, our arguments will present some similarities with the proof of \cite[Theorem~4.3]{bax-chiga-lip04}, despite a different approach.
	
	Our second main contribution is a robustness-type estimate for the Wonham filter (see Theorem~\ref{thm:expected_hilbert_bounds}). Compared to \cite{chigansky07}, we state our error bounds for a general approximate filter, and in terms of the Hilbert distance. Since the Hilbert distance is stronger than the $L_1$-norm, which is used in \cite{chigansky07}, the error bounds we provide are tighter (although still not optimal, as we will discuss in Section~\ref{sec:optimality_rate} and Section~\ref{sec:numerics_approx}). We also believe our proof methodology to be interesting in its own right, being far simpler than the arguments in \cite{chigansky07}:~it relies only on standard stochastic analysis tools, while in \cite{chigansky07} the authors need Malliavin calculus to deal with anticipative stochastic integrals.
	
	Finally, our findings in Theorem~\ref{thm:contraction} suggest that the hyperbolic tangent of the Hilbert distance (instead of simply the Hilbert distance) might be the optimal metric for studying the error of approximate filters. In the particular case when the approximate filter is chosen so that the stochastic term of the Wonham SDE is matched exactly, this yields tighter, pathwise bounds for the error, which we prove in Theorem~\ref{thm:pathwise_decay_approx_error}.
		
	The paper is organized as follows:~in the next section we set-up the filtering equations and define the Hilbert norm, before stating our assumptions and main results. Section \ref{sec:contraction} is dedicated to the proof of Theorem~\ref{thm:contraction} and some discussion of our stability results. Section \ref{sec:robustness} is split in two parts:~in the first half we recover, in a way, Chigansky and Van Handel's results on robustness with respect to misspecified model parameters (see Theorem~\ref{thm:robustness_misspecified_model_params} and compare with \cite[Theorem~1.1]{chigansky07});~in the second half we prove the error bounds for a general approximate filter given in Theorem~\ref{thm:expected_hilbert_bounds} as well as the pathwise bounds of Theorem~\ref{thm:pathwise_decay_approx_error}, and we also present and discuss some numerical experiments.
	
	\section{Filtering set-up and main results}\label{sec:filtering_setup}
	
	Let $\probspace$ be a probability space with a filtration $\{ \mathcal{F}_{t}, t \geq 0 \}$ satisfying the usual conditions. Consider an $\{\mathcal{F}_{t}\}$-adapted continuous-time, time-homogeneous Markov chain $X = (X_t)_{t \ge 0}$ with finite state-space $\mathbb{S} = \{ a_0, \dots, a_n \}$, and associated transition intensity matrix $Q = (q_{ij}) \in \R^{(n+1)\times(n+1)}$.  We let $\mathcal{M}^+(\mathbb{S})$ and $\mathcal{P}(\mathbb{S})$ denote respectively the non-negative measures and the probability measures on  $\mathbb{S}$. Let the initial distribution of $X$ be given by $\mu^i = \Prob (X_0 = a_i)$.
	
	Recall that the $Q$-matrix is defined as the matrix of transition rates such that its entries for each row sum to 0, and its off-diagonal entries are non-negative, i.e.~$\sum_{j} q_{ij} = 0$ for all $i$, and $q_{ij} \geq 0$ for all $i,j \le n+1, \: i \neq j$, and
	\begin{equation*}
	M_{t}^{\varphi}=\varphi\left(X_{t}\right)-\varphi\left(X_{0}\right)-\int_{0}^{t} Q \varphi\left(X_{s}\right) \mathrm{d} s, \quad t \geq 0
	\end{equation*}
	is an $\{\mathcal{F}_{t}\}$-adapted, right-continuous martingale for all bounded functions $\varphi: \mathbb{S} \rightarrow \R$.
 
	Let $h=(h_{i})_{i=1}^{d} : \mathbb{S} \rightarrow \mathbb{R}^{d}$ be a bounded function and $\sigma \neq 0$. Suppose $W$ is a standard $\{\mathcal{F}_{t}\}$-adapted $d$-dimensional Brownian motion independent of $X$, and let $Y = (Y_{t})_{t\ge 0}$ be the process satisfying the SDE
	\begin{equation}\label{eq:observation}
	Y_{t}=Y_{0}+\int_{0}^{t} h\left(X_{s}\right) \mathrm{d} s+\sigma W_{t}.
	\end{equation}
	Let $\{\mathcal{Y}_{t}\}_{t\ge0}$ be the (completed) natural filtration generated by the observation process $Y$. This describes the information available from observing $Y$ in the time-interval $[0,t]$.
	
	By common practice, we identify the state-space $\mathbb{S}$ with the standard basis $\{e_0, \dots, e_n \}$ for $\R^{n+1}$. Denote by $\pi_{t} = \expect{X_t | \mathcal{Y}_t}$ the conditional expectation of $X$ given  $\mathcal{Y}_{t}$. In other words, by abuse of notation, $\pi_t^i = \Prob (X_t = a_i | \mathcal{Y}_t)$.
	
	The process $\pi_{t}$ satisfies the Wonham form of the \textit{Kushner--Stratonovich  equation} (see e.g. \cite[Eq.~3.53]{bain09}):
	\begin{equation}\label{eq:wonhamNorm_d_dim}
	\di \pi_{t} = Q^{\top} \pi_{t} \dt + \frac{1}{\sigma^2}\sum_{k=1}^{d} \left( H^{k} - \pi_{t}^{\top}h_{k} \, \mathbb{I}_{n+1} \right) \pi_{t} \left( \di Y^{k}_{t} - \pi_{t}^{\top}h_{k} \dt \right), \quad \pi_0 = \mu,
	\end{equation}
	where, for $k=1, \ldots, d$, $H^{k}=\operatorname{diag}\left(h_{k}(a_i)\right)$ is an $(n+1)\times(n+1)$-dimensional diagonal matrix and $\mathbb{I}_{n+1}$ is the identity matrix. Note that \eqref{eq:wonhamNorm_d_dim} is initialized at $\mu = \textrm{law}(X_0)=\expect{X_0}$.
	
	The probabilities $\pi_t$ for $t \ge 0$ are $(n+1)$-dimensional (column) vectors, so \eqref{eq:wonhamNorm_d_dim} is a $(n+1)$-dimensional nonlinear SDE. In fact, since the components $\pi_t^i$ must sum to 1 for all $t \ge 0$, the SDE \eqref{eq:wonhamNorm_d_dim} describes a flow on the $n$-dimensional probability simplex $\mathcal{S}^n$, where
	\begin{equation*}
	\mathcal{S}^n = \bigg\{ x \in \R^{n+1} \, : \, \sum_i x_i = 1, \, x_i \ge 0       \bigg\}.
	\end{equation*}
	We write $\mathring{\mathcal{S}}^n$ for the interior of the simplex, that is, $x\in \mathring{\mathcal{S}}^n$ if $x\in \mathcal{S}^n$ and $x^i>0$ for all $i$.
	
	We conclude our set-up by introducing our choice of metric. Given two non-negative measures $\mu, \nu \in \mathcal{M}^+(\mathbb{S})$ expressed as non-negative vectors in $\R^{n+1}$, the Hilbert projective distance $\hilbert$ is defined by
	\begin{equation}\label{eq:hilbert_metric}
	\hilbert(\mu,\nu) = \left\{ \begin{array}{ll}\vspace{2pt}
	\log \left( \frac{\max_{j : \nu^j > 0} \frac{\mu^j}{\nu^j}}{\min_{i : \nu^i > 0} \frac{\mu^i}{\nu^i}} \right),
	& \mu \sim \nu, \\
	\infty, & \mu \nsim \nu.
	\end{array} \right.
	\end{equation}
	The Hilbert distance is a pseudo-metric for $\mathcal{M}^+(\mathbb{S})$:~it is non-negative, symmetric, satisfies the triangle inequality, and $\hilbert(\mu, \nu) = 0$ if and only if $\nu = c \mu$ for some constant $c \in \R^+$. It is a metric on the probability simplex $\mathcal{S}^n$. We refer to \cite[Chapter~XVI]{birkhoff_lattice} and \cite[Chapter~3]{seneta_matrices} for further discussion. The following property relating the Hilbert metric to the Euclidean metric on $\mathbb{R}^{n+1}$ will prove useful.

	\begin{lemma}\label{lemma:hilbertsemicts}
 For any $(\mu,\nu)\in \mathcal{S}^n \times \mathcal{S}^n$, it holds that
  \begin{itemize}
  
  \item[(i)] $\liminf_{m \to \infty}\mathcal{H}(\mu_m, \nu_m)\geq \mathcal{H}(\mu,\nu)$, for all sequences $(\mu_m, \nu_m)\to (\mu,\nu)$ converging in the Euclidean metric; 
  \item[(ii)] there exists a sequence $(\mu_m, \nu_m)\in \mathring{\mathcal{S}}^n \times \mathring{\mathcal{S}}^n$ such that $\lim_{m \to \infty}\mathcal{H}(\mu_m,\nu_m)= \mathcal{H}(\mu,\nu)$,  and $(\mu_m, \nu_m)\to (\mu, \nu)$ in the Euclidean metric;
  \item[(iii)] if $\mu,\nu\in \mathring{\mathcal{S}}^n$, then $\lim_{m \to \infty}\mathcal{H}(\mu_m, \nu_m)= \mathcal{H}(\mu,\nu)$, for all sequences $(\mu_m, \nu_m)\to (\mu,\nu)$ converging in the Euclidean metric.
  \end{itemize}
	\end{lemma}
	\begin{proof}
		If $(\mu,\nu)\in \mathring{\mathcal{S}}^n \times \mathring{\mathcal{S}}^n$, the result is immediate from the definition of the Hilbert metric and continuity of division and logarithms, establishing \textit{(iii)}. 
		Consider now a pair of sequences  $\{\mu_m\}, \{\nu_m\} \in \mathring{\mathcal{S}}^n$ convergent in the Euclidean metric, with respective limits $\mu, \nu \in \mathcal{S}^n$. Suppose first that $\mu\nsim \nu$, then it is easy to verify that either $\max_i\{\mu_m^i/\nu_m^i\}\to \infty$ or $\min_i\{\mu_m^i/\nu_m^i\}\to 0$, hence $\mathcal{H}(\mu_m, \nu_m)\to \infty = \mathcal{H}(\mu,\nu)$. 

		Suppose instead that $\mu\sim \nu$, let $J=\{j:\mu^j=\nu^j=0\}.$ If
		\[\limsup_{m \to \infty} \max_{j\in J} \{\mu^j_m/\nu^j_m\}\leq \max_{i : \nu^i > 0} \{\mu^i/\nu^i\}
		\quad \text{and} \quad
		\liminf_{m \to \infty} \min_{j\in J} \{\mu^j_m/\nu^j_m\}\geq \min_{i : \nu^i > 0} \{\mu^i/\nu^i\},\]
		then a direct calculation shows $\mathcal{H}(\mu_m,\nu_m)\to \mathcal{H}(\mu,\nu)$. Since we can always choose $\{\mu_m \}$ and $\{\nu_m \}$ satisfying the two above inequalities, this proves \textit{(ii)}. 
  
		If a given sequence $(\mu_m, \nu_m)$ does not satisfy the inequalities above, then take any subsequence, still indexed by $m$, such that $\mathcal{H}(\mu_m, \nu_m)$ converges in $[0,\infty]$, and such that at least one of the above inequalities is violated for every term in the subsequence; in particular, suppose that for some $\epsilon>0$, for all $m$,
		\[\max_{j\in J} \{\mu^j_m/\nu^j_m\}> \max_{i:\nu^i>0} \{\mu^i/\nu^i\}+\epsilon.\]
		Then 
		\[\lim_{m \to \infty}\mathcal{H}(\mu_m,\nu_m)\geq \log\bigg(\frac{\max_{i:\nu^i>0}\{\mu^i/\nu^i\}+\epsilon}{\min_{i:\nu^i>0}\{\mu^i/\nu^i\}}\bigg)>\mathcal{H}(\mu,\nu).\]
		A similar argument holds for any subsequence with $\min_{j\in J}\{\mu^j_m/\nu^j_m\}<\min_{i:\nu^i >0}\{\mu^i/\nu^i\}-\epsilon$.
		Therefore, we conclude 		$\liminf_{m \to \infty} \mathcal{H}(\mu_m, \nu_m) \geq \mathcal{H}(\mu,\nu)$, which is \textit{(i)}.
	\end{proof}

	\subsection{Main results}\label{sec:main_results}
	
	Throughout the paper, we make the following assumptions on the nonlinear filtering system described in Section~\ref{sec:filtering_setup}.
	\begin{assumption}
		$X$ is a time-homogeneous continuous-time Markov chain on $n+1$ states.
	\end{assumption}
	\begin{assumption}
		$h = (h_i)_{i=1}^d$ is bounded for all $i$.
	\end{assumption}
	\begin{assumption}
		$\sigma = 1$ and $d=1$ in \eqref{eq:observation} and \eqref{eq:wonhamNorm_d_dim}.
	\end{assumption}
	
	The final assumption only serves the purpose of simplifying notation -- all our results are easily extendable to the case of multi-dimensional $Y$ and invertible $\sigma \in \R^{d \times d}$. Similarly, we could easily allow for time-dependence in $\sigma$ and $h$, as long as the first is bounded away from zero, and the second stays bounded for all $t$, and for time inhomogeneity in the Markov chain dynamics of $X$.
	\begin{notation}
		Given that we take the observations $Y$ to be one-dimensional, the sensor function $h : \, \mathbb{S} \rightarrow \R $ can be seen as a vector $h \in \R^{n+1}$ with entries $h^i = h(a_i)$ for $i = 0, \dots, n$. From now on we will employ this notation. We also denote by $H = \operatorname{diag}(h)$ the diagonal matrix with entries $(H)_{ii} = h^i$. In general, we will always denote the components of vectors (or vector-valued processes) with superscripts. We denote by $\mathbf{N}$ the set of natural numbers $\{0, \dots, n \}$. Sometimes we will write $\di A_t \le \di \tilde A_t$ for two Lebesgue--Stieltjes measures $A_t$ and $\tilde A_t$ on $[0, \infty)$, by which we mean $\int_s^t \di A_r \le \int_s^t \di \tilde A_r$ for all $0 \le s < t < \infty$.
	\end{notation}
	
	For reference, we rewrite here equation \eqref{eq:wonhamNorm_d_dim} for the Wonham filter given the above assumptions and notation
	\begin{equation}\label{eq:wonhamNorm}
	\di \pi_{t} = Q^{\top} \pi_{t} \dt + \left( H - \pi_{t}^{\top}h \, \mathbb{I}_{n+1} \right) \pi_{t} \left( \di Y_{t} - \pi_{t}^{\top}h \dt \right), \quad \pi_0 = \mu.
	\end{equation}
	
	We will first consider the long time behaviour of the error between $\pi_t$ and $\tilde\pi_t$, where $\tilde\pi_t$ is the filter initialized with the `wrong' initial data $\tilde\pi_0 = \nu \neq \mu$ but the same dynamics as $\pi$. The evolution equation for $\tilde\pi_t$ is given by
	\begin{equation}\label{eq:wonhamWrong}
	\di \tilde\pi_{t} = Q^{\top} \tilde\pi_{t} \dt + \left( H - \tilde\pi_{t}^{\top}h \, \mathbb{I}_{n+1} \right) \tilde\pi_{t} \left( \di Y_{t} - \tilde\pi_{t}^{\top}h \dt \right), \quad \tilde\pi_0 = \nu.
	\end{equation}
	
	Our key result is the following pathwise estimate on the stability of the filter.
	
	\begin{thm}[Contraction rate of $\mathcal{H}(\pi_t, \tilde \pi_t)$]\label{thm:contraction}
		Let $\pi_t$ be the solution to \eqref{eq:wonhamNorm} and $\tilde \pi_t$ the solution to \eqref{eq:wonhamWrong}.
		Suppose $q_{ij} > 0$ for all $ i \neq j$. Then for all $t < \infty$,
		\begin{equation*}
		\tanh \bigg( \frac{\hilbert ( \pi_t, \tilde \pi_t )}{4} \bigg) \le \tanh \bigg( \frac{\hilbert (\mu, \nu )}{4} \bigg) e^{- \lambda t},
		\end{equation*}
		where $\lambda = 2 \min_{i \neq j} \sqrt{q_{ij} q_{ji}}$. In particular,
		\begin{equation*}
		\hilbert ( \pi_t, \tilde \pi_t ) \le \hilbert (\mu, \nu ) e^{- \lambda t}.
		\end{equation*}
	\end{thm}
	
	Unsurprisingly, our contraction rate is the same as the asymptotic rate in \cite{atar-zeitouni-lyapunov97}, and the non-asymptotic rate in \cite[Theorem~4.3]{bax-chiga-lip04}, and shares the issue of only being (strictly) positive if all the off-diagonal entries of $Q$ are (strictly) positive. This is a very strong mixing assumption on $X$;~however, it seems necessary to be able to compute an explicit contraction rate, and in fact a similar assumption is made in \cite{legland04} in the discrete-time setting (see \cite[Definition~3.2]{legland04}).
	
	Given these stability estimates, we are able to proceed to the next challenge of understanding the error of approximate filters. Consider a general approximate filter of the form
	\begin{equation}\label{eq:wonhamApprox}
	\di \tilde \pi_t  = \tilde f_t \dt + \tilde g_t \di Y_t, \quad \tilde \pi_0 = \nu,
	\end{equation}
	where $\tilde f_t, \tilde g_t$ are $ \R^{n+1}$-valued $\{\mathcal{Y}_t\}$-predictable process and $\tilde \pi_t \in \mathring{\mathcal{S}}^n$ for all $t$, and we refer to Section~\ref{sec:approximate_filter} for the necessary assumptions on $\tilde f_t$, $\tilde g_t$ and $\tilde \pi_t$.
	\begin{thm}[Bounds for the expected Hilbert error] \label{thm:expected_hilbert_bounds}
		Let $\pi_t$ be the solution to \eqref{eq:wonhamNorm} and $\tilde \pi_t$ the solution to \eqref{eq:wonhamApprox}. Suppose $\mu, \nu \in \mathring{\mathcal{S}}^n$ and $q_{ij} > 0$ for all $i \neq j$. Assuming sufficient integrability in \eqref{eq:wonhamApprox} (see Assumption~\assumptionref{assmp:approx_filter}), for all $t < \infty$, we have that
		\begin{align*}
			\expect{\hilbert (\pi_t, \tilde \pi_t)}
			&\le \hilbert (\mu, \nu) e^{-\lambda t}
			+ \int_0^t e^{-\lambda(t-s)} \expect{\max_{i,k}  \Big\{ \mathcal{E}_s^{1,i}  - \mathcal{E}_s^{1,k}  - \frac{1}{2} \big(\mathcal{E}_s^{2,i} - \mathcal{E}_s^{2,k}  \big) \Big\} } \ds \nonumber \\
			&\quad + \max_j |h^j|  \int_0^t e^{-\lambda(t-s)} \expect{ \max_{i,k} \big\{ \mathcal{E}_s^{3,i}  - \mathcal{E}_s^{3,k}  \big\} } \ds \nonumber \\
			&\quad + \frac{1}{4}\sum_{(i,k)} \sum_{(j,l) \neq (i,k)} \expect{\int_0^t e^{-\lambda(t-s)} \di L^{0}_s(\Delta_{ik}(\cdot) - \Delta_{jl}(\cdot))},
		\end{align*}
		where $\lambda = 2\min_{i \neq k}\sqrt{q_{ik}q_{ki}}$ is the deterministic contraction rate from Theorem~\ref{thm:contraction}. For $j \in \mathbf{N}$ the error terms are given by
		\begin{equation}\label{eq:error_terms_approx}
			\mathcal{E}_t^{1, j}  = \Bigg( \sum_{m=0}^{n} q_{mj} \frac{\tilde \pi_t^m}{\tilde \pi_t^j} \Bigg) - \frac{\tilde f_t^j}{\tilde \pi_t^j}, \quad
			\mathcal{E}_t^{2, j}  = (h^j)^2 - \frac{(\tilde g^j_t)^2}{(\tilde \pi_t^j)^2}, \quad
			\mathcal{E}_t^{3, j}  = h^j - \frac{\tilde g^j_t}{\tilde \pi_t^j}, \quad
		\end{equation}
		and the processes $(\Delta_{ik}(t))_{t \ge 0}$ for $(i,k) \in \mathbf{N} \times \mathbf{N}$ are defined as $\Delta_{ik}(t) = \log \frac{\pi_t^i}{\pi_t^k} - \log \frac{\tilde \pi_t^i}{\tilde \pi_t^k}$. The set $\mathcal{I}_t = \{ (i,k) \, : \, \Delta_{ik}(t) = \hilbert (\pi_t, \tilde \pi_t)\}$ is the argmax of these processes for all $ t < \infty$, and $L_t^0(\Delta_{ik}(\cdot) - \Delta_{jl}(\cdot))$ denotes the local time at $0$ of the difference process $(\Delta_{ik} - \Delta_{jl})$ for all $(i,k)$, $(j,l) \in \mathbf{N} \times \mathbf{N}$.
	\end{thm}
Assuming there is no error in the stochastic terms when comparing \eqref{eq:wonhamApprox} and \eqref{eq:wonhamNorm}, a stronger result is possible.
	\begin{thm}[Pathwise decay rate for the Hilbert error]\label{thm:pathwise_decay_approx_error}
		Under the same assumptions as in Theorem~\ref{thm:expected_hilbert_bounds}, suppose that the error terms $\mathcal{E}_t^{3,i}$ defined in \eqref{eq:error_terms_approx} vanish for all $i \in \mathbf{N}$ and all $t \ge 0$, and $\tilde \pi_t$ is observable. Let $u_t \in (0,1)$ be the unique solution to the ODE with random coefficients given by
	\begin{equation}\label{eq:thm_approx_pathwise_ode}
		\frac{\di u_t}{\dt} = - \tilde \lambda^{\star} (t, u_t) u_t + \max_{i,k}  \big\{ \mathcal{E}_s^{1,i}  - \mathcal{E}_s^{1,k} \big\} (1-u_t^2),
		\qquad u_0 = \tanh \bigg( \frac{\hilbert( \mu, \nu)}{4}   \bigg),
	\end{equation}
	where
	\begin{equation*}
	\tilde \lambda^{\star}(t, u_t)
	= \min_{i \neq k} \bigg\{\bigg( q_{ik} \frac{\tilde \pi_t^i}{\tilde \pi_t^k}
	+ \hspace{-5pt} \sum_{\substack{ j \neq i,k, \\ j \notin \mathcal{\tilde J}^i_k(t, u_t)}} \hspace{-5pt} q_{j k} \frac{\tilde \pi_t^j}{\tilde \pi_t^k}\bigg) \frac{1 + u_t}{1-u_t}
	+ \bigg( q_{ki} \frac{\tilde \pi_t^k}{\tilde \pi_t^i}
	+ \hspace{-5pt} \sum_{\substack{ j \neq i,k, \\ j \in \mathcal{\tilde J}^i_k(t, u_t)}} \hspace{-5pt} q_{j i} \frac{\tilde \pi_t^j}{\tilde \pi_t^i}\bigg) \frac{1 - u_t}{1+u_t} \bigg\},
	\end{equation*}
	and $
	\mathcal{\tilde J}^i_k(t, u_t) := \Big\{ j \in \mathbf{N} \, : \, \frac{q_{j k}}{\tilde \pi_t^k} \ge \frac{q_{j i}}{\tilde \pi_t^i} \Big( \frac{1-u_t}{1+u_t} \Big)^2 \Big\}$.
	Then for all $t < \infty$,
	\begin{equation*}
	\tanh \bigg( \frac{\hilbert (\pi_t, \tilde \pi_t)}{4} \bigg) \le u_t.
	\end{equation*}
	
	In particular, $\tilde 
	\lambda^{\star}(t, u_t) \ge \tilde \lambda^{\star}_t$, where
	\begin{align}
		\tilde \lambda^{\star}_t 
		&:= 2\min_{i \neq k}
		\left\{ \min_{S \subseteq \mathbf{N}} \sqrt{  q_{ik}q_{k i}
		+  \hspace{-7pt} \sum_{j \in S,\, j \neq i,k}  \hspace{-7pt} q_{ik} q_{j i} \frac{\tilde \pi_t^j}{\tilde \pi_t^k}
		+ \hspace{-7pt} \sum_{l \notin S, \, l \neq i,k} \hspace{-7pt} q_{ki} q_{l k} \frac{\tilde \pi_t^l}{\tilde \pi_t^i}
		+  \hspace{-3pt}   \sum_{\subalign{\: j &\in S, \\ j &\neq i,k}} \sum_{\subalign{\: l &\notin S, \\ l &\neq i,k}}  \hspace{-3pt} q_{ji} q_{l k} \frac{\tilde \pi_t^j \tilde \pi_t^l}{\tilde \pi_t^i \tilde \pi_t^k} } \right\}, \nonumber \\
		&\ge 2\min_{i \neq k} \left\{ \sqrt{ q_{i k} q_{k i}
		+ \sum_{j \neq i, k} \min \bigg\{ \frac{ q_{j i} q_{i k}}{\tilde \pi_t^k}, \frac{q_{j k} q_{k i}}{\tilde \pi_t^i} \bigg\}  \tilde \pi_t^j } \right\} =: \tilde \lambda_t, \label{eq:thm_approx_pathwise_rate}
	\end{align}
	which gives that for all $t < \infty$, we have the two bounds
	\begin{equation}
		\tanh \bigg( \frac{\hilbert (\pi_t, \tilde \pi_t)}{4} \bigg) \le \tanh \bigg( \frac{\hilbert (\mu, \nu)}{4} \bigg) e^{- \int_0^t \tilde \lambda^{\star}_s \ds}
		+ \frac{1}{4}\int_0^t e^{- \int_s^t \tilde \lambda^{\star}_r \di r} \max_{i,k}  \big\{ \mathcal{E}_s^{1,i}  - \mathcal{E}_s^{1,k} \big\} \ds, \label{eq:tanh_pathwise_approx_bound}
	\end{equation}
	and
	\begin{equation}
		\hilbert (\pi_t, \tilde \pi_t) \le \hilbert (\mu, \nu) e^{- \int_0^t \tilde \lambda^{\star}_s \ds}
		+ \int_0^t e^{- \int_s^t \tilde \lambda^{\star}_r \di r} \max_{i,k}  \big\{ \mathcal{E}_s^{1,i}  - \mathcal{E}_s^{1,k} \big\} \ds. \label{eq:hilbert_pathwise_approx_bound}
	\end{equation}
	\end{thm}

	\begin{rmk}
		Theorem~\ref{thm:pathwise_decay_approx_error} suggests an approach to constructing approximate filters with relatively small error. If $h$ is known, and we can choose $\tilde g_t = H \tilde \pi_t$ such that the error terms $\mathcal{E}^{2,i}$ and $\mathcal{E}^{3,i}$ vanish for all $i \in \mathbf{N}$, then the errors due to the stochastic term vanish, and the local time terms as well. From a numerical perspective, this is equivalent to killing the infinitesimal errors of order $\sqrt{\dt}$;~this is natural when looking for an approximate solution to the Wonham SDE.
	\end{rmk}

	\begin{rmk}
        As in the proof of Theorem \ref{thm:contraction}, by using Lemma \ref{lemma:hilbertsemicts}, it is possible to lift the assumption that $\mu, \nu \in \mathring{\mathcal{S}}^n$ in Theorem \ref{thm:pathwise_decay_approx_error}, provided one takes sufficient care in constructing the solution to the ODE \eqref{eq:thm_approx_pathwise_ode}.
    \end{rmk}
	The following corollary provides some exchangeability between $\pi_t$ and $\tilde \pi_t$ when computing the decay rates.

	\begin{corollary}\label{cor:exchanging_pi_with_pi_tilde}
		Assume the Wonham filter $\pi_t$ is observable. Theorem~\ref{thm:pathwise_decay_approx_error} holds equivalently if one substitutes $\lambda^{\star}(t, u_t)$ for $\tilde \lambda^{\star}(t,u_t)$ in \eqref{eq:thm_approx_pathwise_ode} and $\lambda^{\star}_t$ for $\tilde \lambda^{\star}_t$ in \eqref{eq:tanh_pathwise_approx_bound} and \eqref{eq:hilbert_pathwise_approx_bound}, where
		\begin{equation*}
			\lambda^{\star} (t, u_t)
			= \min_{i \neq k} \bigg\{\bigg( q_{ik} \frac{ \pi_t^i}{ \pi_t^k}
			+ \hspace{-5pt} \sum_{\substack{ j \neq i,k, \\ j \in \mathcal{J}^i_k(t, u_t)}} \hspace{-5pt} q_{j k} \frac{ \pi_t^j}{ \pi_t^k}\bigg) \frac{1 - u_t}{1+u_t}
			+ \bigg( q_{ki} \frac{ \pi_t^k}{ \pi_t^i}
			+ \hspace{-5pt} \sum_{\substack{ j \neq i,k, \\ j \in \mathcal{J}^i_k(t, u_t)}} \hspace{-5pt} q_{j i} \frac{ \pi_t^j}{ \pi_t^i}\bigg) \frac{1 + u_t}{1-u_t} \bigg\},
		\end{equation*}
			and $\mathcal{J}^i_k(t, u_t) := \Big\{ j \in \mathbf{N} \, : \, \frac{q_{j k}}{ \pi_t^k} \le \frac{q_{j i}}{\pi_t^i} \Big( \frac{1+u_t}{1-u_t} \Big)^2 \Big\}$,
		and $\lambda^{\star}_t$ is defined equivalently to \eqref{eq:thm_approx_pathwise_rate} with $\pi_t$ in place of $\tilde \pi_t$. 
	\end{corollary}

	\section{Contraction rates in the Hilbert projective metric}\label{sec:contraction}
	
	The aim of this section is to prove Theorem \ref{thm:contraction} and a few more results related to the stability of the nonlinear filter with respect to its initial conditions.
	
	We start by introducing a family of coordinate transformations from the interior of the probability simplex $\mathring{\mathcal{S}}^n$ to $\R^n$ that map a discrete probability distribution to its natural parameters. We derive the evolution equation for the Wonham filter in these new parametrizations, and then consider the difference between the natural parameters of the Wonham filter initialized at $\mu = \text{law}(X_0)$ and those of the Wonham filter `wrongly' initialized at $\nu \neq \mu$. By relating the $\ell_{\infty}$ norm of the difference, maximized over parametrizations, to the Hilbert projective metric, we are able to compute explicitly an exponential contraction rate in the Hilbert metric for the Wonham filter. Up until the proof of Theorem~\ref{thm:contraction}, we will regularly make the extra assumption that $\pi_0 = \mu$ and $\tilde \pi_0 =\nu$ belong to the interior of the simplex.
	
	\begin{rmk}
		For the entirety of this section, $(\pi_t)_{t \ge0}$ represents the Wonham filter initialized at $\pi_0 = \mu$, and $(\tilde \pi_t)_{t \ge0}$ the Wonham filter initialized at $\tilde \pi_0 = \nu$.
	\end{rmk}

	\subsection{Coordinate transformations}\label{sec:change_of_coords}
	
	The coordinate transformation we will consider here sends a probability distribution to what, in statistics, are called the \textit{natural (or canonical) parameters}. Natural parameters are the usual choice of parametrization for an exponential family of distributions, which have probability densities that can be written in general form as
	\begin{equation}\label{eq:exponential_family}
	p(x, \theta) = \exp \{ \theta \cdot c(x) + k(x) - \psi(\theta) \},
	\end{equation}
	where $\theta \in \R^n$ is the $n$-dimensional vector of natural parameters, $c(x)$ is the vector of \textit{sufficient statistics} of the distribution, and its $n$ components are linearly independent, $k(x)$ is a function of $x$ and $\psi(\theta)$ is the log partition function. A change of measure from $\dx$ to $\di \upsilon (x) = \exp \{k(x)\} \dx$ allows us to ignore $k(x)$, as long as $p(x, \theta)$ is understood as a density with respect to the measure $\di \upsilon(x)$ instead. We will assume $k(x)=0$ for simplicity.
	
	Our choice of studying the filtering equations in the coordinate system $\theta$ of natural parameters is motivated by Amari's theory of information geometry \cite{amari85, amari16}. If $\mu,\nu\in \mathring{\mathcal{S}}^n$, then the filtering process $\pi_t$ lives in $\mathring{\mathcal{S}}^n$, that is, the interior of the probability simplex. In the language of information geometry, $\mathring{\mathcal{S}}^n$ is an $n$-dimensional statistical manifold, with $\theta$ (and its dual affine, the \textit{expectation parameter} $\eta$) as a global chart. The Riemannian metric for $\mathring{\mathcal{S}}^n$ is the Fisher Information, which infinitesimally agrees with the KL-divergence. In this paper, we will not make use of the differential geometrical structures for $\mathring{\mathcal{S}}^n$ developed by Amari, as our choice of norm for the distance between probability vectors is the Hilbert norm, which does not allow for a smooth geometry. However, it will still be convenient to work in the global coordinate system given by the $\theta$ parametrization. 
	
	Consider a probability vector $p \in \mathring{\mathcal{S}}^n \subset \R^{n+1}$. Note that $\mathring{\mathcal{S}}^n$ is an $n$-dimensional exponential family, and we can write a discrete distribution $p \in \mathring{\mathcal{S}}^n$ in the form \eqref{eq:exponential_family} by fixing $k \in \mathbf{N}$ and choosing $c^i(x) = \delta_{a_i}(x)$ (for $a_i \in \mathbb{S}$). Choosing $k \in \mathbf{N}$, we define the diffeomorphism $\theta_k : \mathring{\mathcal{S}}^n \rightarrow \R^n $ that maps $p \mapsto \theta_k$ as follows:
	\begin{equation}\label{eq:theta-trans}
	\theta_k^i = \log \frac{p^i}{p^k} , \quad \forall i \in \mathbf{N}.
	\end{equation}
	We remark that $\theta^k_k = 0$ could be ignored as an entry of the vector $\theta_k$ (and it can be `skipped'), so that indeed $\theta_k \in \R^{k-1} \times \{ 0 \} \times \R^{n-k} \cong \R^n$.
	
	The inverse map $\theta_k^{-1}$ is given by
	\begin{equation}\label{eq:theta-inv}
	p^i = \frac{\exp{\theta_k^i}}{1 + \sum_{j \neq k} \exp{\theta_k^j}}, \quad \forall i \in \mathbf{N}.
	\end{equation}
	
	We now would like to apply the coordinate transformation \eqref{eq:theta-trans} to $(\pi_t)_{t \ge 0}$ and $(\tilde \pi_t)_{t \ge 0}$ and derive evolution equations for the parameters $\theta_k(\pi_t)$ and $\theta_k(\tilde \pi_t)$. For all $k \in \mathbf{N}$, for notational simplicity define
	\begin{equation*}
	\theta_k(t) := \theta_k(\pi_t),\quad \tilde \theta_k(t) : = \theta_k(\tilde \pi_t),
	\end{equation*}
	so that, component-wise, we have
	\begin{equation*}
	\theta^i_k(t) := \log \frac{\pi^i_t}{\pi^k_t},\quad \tilde \theta^i_k(t) : = \log \frac{\tilde \pi^i_t}{\tilde \pi^k_t}, \quad \forall (i, k) \in \mathbf{N} \times \mathbf{N}.
	\end{equation*}
	
	The following lemma guarantees that these processes are almost surely well-defined for all $t < \infty$. For its proof we refer to \cite{chigansky07}. 
	\begin{lemma}[Lemma~2.1 in \cite{chigansky07}]\label{lemma:non_divergence_to_boundary}
		Denote by $\pi_{s,t}(\mu)$ the solution at time $t \ge 0$ to \eqref{eq:wonhamNorm} initialized at time $s \le t$ with $\pi_s = \mu$. Then
		\begin{equation*}
		\Prob \big( \pi_{s,t}(\mu) \in \mathring{\mathcal{S}}^n \textrm{ for all } \mu \in \mathring{\mathcal{S}}^n \textrm{ and all } 0 \le s \le t < \infty  \big) = 1.
		\end{equation*}
	\end{lemma}
 	\begin{corollary}\label{cor:finiteness_of_theta_and_hilbert}
		Assume $\mu, \nu \in \mathring{\mathcal{S}}^n$. We have that, almost surely,
		\begin{equation*}
		\hilbert (\pi_t, \tilde \pi_t), \, |\theta_k^i(t)|, \, |\tilde \theta_k^i(t)| < \infty, \quad \forall \, (i,k) \in \mathbf{N} \times \mathbf{N},
		\end{equation*}
		for all times $0 \le t < \infty$.
	\end{corollary}
	The proof of Lemma \ref{lemma:non_divergence_to_boundary} also directly yields the following alternative result.
 \begin{lemma}\label{lemma:non_divergence_to_boundary2}
 Denote by $\pi_{t}(\mu)$ the solution at time $t \ge 0$ to \eqref{eq:wonhamNorm} initialized at time $0 \le t$ with $\pi_0 = \mu$. Suppose $q_{ij}>0$ for all $i\neq j$. Then
   \begin{equation*}
	   \Prob \big( \pi_{t}(\mu) \in \mathring{\mathcal{S}}^n \textrm{ for all } \mu \in {\mathcal{S}}^n \textrm{ and all } 0 < t < \infty  \big) = 1.
	   \end{equation*}  
    \end{lemma}

	We now proceed to study the dynamics of the natural parameters $\theta_k^i(t)$ and $\tilde \theta_k^i(t)$. For all pairs of indices $(i, k) \in \mathbf{N} \times \mathbf{N}$, define the difference process
	\begin{equation}\label{eq:delta_ik_def}
	\Delta_{ik}(t) : = (\theta_k^i(t) - \tilde \theta_k^i(t))_{t \ge 0},
	\end{equation}
	where $\Delta_{ii} = 0$ for all $i \in \mathbf{N}$.
	
	We start with the following proposition.
	
	\begin{prop}\label{prop:evol_theta}
		Assume $\mu, \nu \in \mathring{\mathcal{S}}^n$. For all $0 \le t < \infty$ and all pairs of indices $(i, k) \in \mathbf{N} \times \mathbf{N}$, the process $\Delta_{ik}(t)$ is $C^1$ in time and has the dynamics
		\begin{align}\label{eq:evol_diff_theta}
		\frac{\di}{\dt} \Delta_{ik}(t) &= 
		-\sum_{\substack{j = 0 \\ j\neq k}}^{n} q_{jk} \big( e^{\theta_k^j} - e^{\tilde \theta_k^j} \big)
		+ \sum_{\substack{j = 0 \\ j\neq i}}^{n} q_{ji} \big( e^{\theta_i^j} - e^{\tilde \theta_i^j} \big), \nonumber \\
		\Delta_{ik}(0) &= \log \frac{\mu^i}{\mu^k} - \log \frac{\nu^i}{\nu^k}.
		\end{align}
	\end{prop}
	
	\begin{proof}
		For $i = k$ the process $\Delta_{kk}$ is identically 0, so the statement holds trivially. Assume $i \neq k$.
		Consider $\theta_k^i(t) = \log (\pi^i_t / \pi^k_t)$ for $i \neq k$. We apply It\^o's formula and obtain that, for any choice of $k \in \mathbf{N}$, and $i \neq k$, we have
		\begin{align} \label{eq:evol_theta_ik}
		\di \log \frac{\pi^i}{\pi^k} (t) &= -\sum_{\substack{j = 0 \\ j\neq k}}^{n} q_{jk} \frac{\pi^j_t}{\pi^k_t} \dt
		+ \sum_{\substack{j = 0 \\ j\neq i}}^{n} q_{ji} \frac{\pi^j_t}{\pi^i_t} \dt + (q_{ii}-q_{kk}) \dt + (h^i - h^k) \di B_t \nonumber \\
		&\quad + \frac{1}{2} \Big( (h^k)^2 - (h^i)^2 + 2(h^i - h^k) \pi_t^{\top} h \Big) \dt, \nonumber \\
		\log  \frac{\pi^i}{\pi^k} (0) &= \log  \frac{\mu^i}{\mu^k},
		\end{align}
		where for readability we have introduced the innovation process $B_t = Y_t - \int_0^t \pi_s^{\top} h \ds$, which is a $\{\mathcal{Y}_{t}\}$-adapted Brownian motion (see e.g.~\cite[Proposition~2.30]{bain09}).
		Similarly,
		\begin{align*}
		\di \log \frac{\tilde\pi^i}{\tilde\pi^k} (t) &= -\sum_{\substack{j = 0 \\ j\neq k}}^{n} q_{jk} \frac{\tilde\pi^j_t}{\tilde\pi^k_t} \dt
		+ \sum_{\substack{j = 0 \\ j\neq i}}^{n} q_{ji} \frac{\tilde\pi^j_t}{\tilde\pi^i_t} \dt + (q_{ii}-q_{kk}) \dt + (h^i - h^k) \di B_t \nonumber \\
		&\quad + \frac{1}{2} \Big( (h^k)^2 - (h^i)^2 + 2(h^i - h^k) \tilde\pi_t^{\top} h \Big) \dt + (h^i - h^k) \big( \pi_t^{\top} h - \tilde\pi_t^{\top} h \big) \dt, \nonumber \\
		\log  \frac{\tilde\pi^i}{\tilde\pi^k}  (0) &= \log \frac{\nu^i}{\nu^k}.
		\end{align*}
		Subtracting the two equations, we see that the difference has absolutely continuous dynamics
		\begin{align}\label{eq:evol_diff_theta_fractions}
		\di \bigg( \log  \frac{\pi^i}{\pi^k}(t)  - \log \frac{\tilde\pi^i}{\tilde\pi^k}(t) \bigg) &= 
		-\sum_{\substack{j = 0 \\ j\neq k}}^{n} q_{jk} \bigg( \frac{\pi^j_t}{\pi^k_t} - \frac{\tilde\pi^j_t}{\tilde\pi^k_t} \bigg) \dt
		+ \sum_{\substack{j = 0 \\ j\neq i}}^{n} q_{ji} \bigg( \frac{\pi^j_t}{\pi^i_t} - \frac{\tilde\pi^j_t}{\tilde\pi^i_t} \bigg) \dt, \nonumber \\
		\log  \frac{\pi^i}{\pi^k}(0)  - \log \frac{\tilde\pi^i}{\tilde\pi^k} (0) &= \log \frac{\mu^i}{\mu^k} - \log \frac{\nu^i}{\nu^k}.
		\end{align}
		Noting that the right-hand side of the above equation is continuous in time (since $\pi_t$ and $\tilde \pi_t$ are both continuous), we have that the derivative of $\Delta_{ik}$ exists and is continuous for every $t \ge 0$, and \eqref{eq:evol_diff_theta} follows.
	\end{proof}
	
	\subsection{The Hilbert error}
	Comparing \eqref{eq:delta_ik_def} with \eqref{eq:hilbert_metric}, we now observe that the Hilbert norm can be expressed through the maximal process
	\begin{equation*}
	\Delta_{\infty}(t) := \max_{k \in \mathbf{N}} \norm{\theta_k(t) - \tilde \theta_k(t)}_{\ell^\infty} = \max_{(i,k) \in \mathbf{N} \times \mathbf{N}} \Delta_{ik}(t).
	\end{equation*}
	This can be seen easily by observing that
	\begin{align} \label{eq:Delta_equal_H}
	\Delta_{\infty}(t)
	&= \max_{(i,k) \in \mathbf{N} \times \mathbf{N}} \bigg( \log  \frac{\pi^i}{\pi^k}(t)  - \log \frac{\tilde\pi^i}{\tilde\pi^k}(t) \bigg)
	= \max_{(i,k) \in \mathbf{N} \times \mathbf{N}} \bigg( \log  \frac{\pi^i}{\tilde \pi^i}(t)  - \log \frac{\pi^k}{\tilde\pi^k}(t) \bigg) \nonumber \\
	&= \max_{i \in \mathbf{N}} \log  \frac{\pi^i}{\tilde \pi^i}(t)  - \min_{k \in \mathbf{N}} \log \frac{\pi^k}{\tilde\pi^k}(t) = \mathcal{H}(\pi_t, \tilde\pi_t),
	\end{align}
	where the last equality follows by monotonicity of log.
	
	We want to study the evolution in time of the stochastic process $\Delta_{\infty}(t)$. We here adapt some arguments from \cite{bax-chiga-lip04}, since it turns out that our difference processes $\Delta_{ik}$ of Proposition~\ref{prop:evol_theta} have dynamics somewhat similar to the equations of the smoother process considered in \cite[Section~5.2, Eq.~5.6 \& Eq.~5.7]{bax-chiga-lip04}.
	
	We will need the following lemma in what follows. 
	\begin{lemma}[Theorem A.6.3 in Dupuis and Ellis \cite{dupuis_ellis}]\label{lemma:dupuis}
		Let $g \, : \, [0,1] \rightarrow \R$ be an absolutely continuous function. Then for every real number $r \in \R$, the set $\{t \, : \, g(t) = r, \dot{g}(t) \neq 0 \}$ has Lebesgue measure $0$.
	\end{lemma}
	
	In what follows, when we say that an adapted stochastic process $Z(t,\omega)$ is absolutely continuous or has absolutely continuous paths (a.s.), we mean not only that it can be written as $\di Z(t,\omega)=g(t,\omega)\dt$ with $g\in L^1([0,t])$, for all $t>0$ (a.s.), but also that the weak derivative $g(t,\omega)$ is jointly measurable and adapted to the underlying filtration.
	The next lemma confirms that this is the case for the process $\Delta_{\infty}(t, \omega)$.
	
	\begin{lemma}\label{lemma:delta_inf_measurable}
		Assume $\mu, \nu \in \mathring{\mathcal{S}}^n$. The stochastic process $(t,\omega)\mapsto\Delta_{\infty}(t, \omega)$ has absolutely continuous paths (in particular, it is predictable).
	\end{lemma}
	\begin{proof}
		Fix an arbitrary $k \in \mathbf{N}$. Start by considering the processes $\Delta^{\star}_{i,k}(t) = \Delta_{0k} \vee \Delta_{1k} \vee \dots \vee \Delta_{ik}$ for $i \in \mathbf{N}$. We proceed by induction to prove absolute continuity of $\Delta^{\star}_{n,k}(t) = \max_{i \in \mathbf{N}} \Delta_{ik}$. Trivially, $\Delta^{\star}_{0,k}(t) = \Delta_{0k}(t)$ is absolutely continuous, since it is either constant 0 by definition (if $k=0$), or is absolutely continuous by Proposition~\ref{prop:evol_theta} (if $k \neq 0$). Consider the case $i = 1$, with $\Delta^{\star}_{1,k}(t) = \Delta_{0k}(t) \vee \Delta_{1k}(t)$. Recall that $a \vee b = \frac{1}{2} ( a + b + |a-b| )$. Then
		\begin{equation*}
		\Delta^{\star}_{1,k}(t) = \frac{1}{2} \big( \Delta_{0k}(t) + \Delta_{1k}(t) + |\Delta_{0k}(t) - \Delta_{1k}(t) |   \big).
		\end{equation*}
		By Proposition~\ref{prop:evol_theta} we have that $\Delta_{0k}(t)$ and $\Delta_{1k}(t)$ are $C^1$ in time, and $\mathcal{F}_t$-measurable in $\omega$. By the chain rule for weakly differentiable functions, if $F(t)$ is absolutely continuous with weak derivative $f(t)$, then
		\begin{equation}\label{eq:absolute_val_abs_cont}
		\di | F(t) | = \mathrm{sign}(F(t)) f(t) \dt.
		\end{equation}	
		Thus we have that $|\Delta_{0k}(t) - \Delta_{1k}(t) |$ is absolutely continuous in time (for each $\omega$), and it is clear from the form of \eqref{eq:absolute_val_abs_cont} that the weak derivative is jointly measurable in $(t, \omega)$ and $\mathcal{F}_t$-adapted. Hence the same is true for $\Delta^{\star}_{1,k}(t)$.
		
		Now noting that $\Delta^{\star}_{i,k}(t) = \Delta^{\star}_{i-1,k}(t) \vee \Delta_{ik}(t)$ for all $2 \le i \le n$, as before we can write
		\begin{equation*}
		\Delta^{\star}_{i,k}(t) = \frac{1}{2} \big( \Delta^{\star}_{i-1,k}(t) + \Delta_{ik}(t) + |\Delta^{\star}_{i-1,k}(t) - \Delta_{ik}(t) |   \big),
		\end{equation*}
		and by induction it follows that $\Delta^{\star}_{n,k}(t) = \max_{i \in \mathbf{N}} \Delta_{ik}$ has absolutely continuous paths.
		
		Since the argument above is independent of our choice of $k$, we have that $\Delta^{\star}_{n,k}(t)$ is absolutely continuous for all $k \in \mathbf{N}$. Now all we have to do is take the maximum of $\Delta^{\star}_{n,k}(t)$ over all $k \in \mathbf{N}$ and prove it is also absolutely continuous. Consider the processes $\Delta^{\star}_k(t) = \Delta^{\star}_{n,0} \vee \Delta^{\star}_{n,1} \vee \dots \vee \Delta^{\star}_{n,k}$ for $k \in \mathbf{N}$. Proceeding by induction exactly as above, by exploiting the absolute continuity of the processes $\Delta^{\star}_{n,k}$, we finally obtain that the process $\Delta^{\star}_n(t) = \max_{k \in \mathbf{N}} \Delta^{\star}_{n,k}$ is measurable in $(t, \omega)$ and absolutely continuous in time. Noting that $\Delta^{\star}_n(t) = \Delta_{\infty}(t)$, we are done.
	\end{proof}
	
	\begin{lemma}\label{lemma:measurable_selection}
		Assume $\mu, \nu \in \mathring{\mathcal{S}}^n$. There exists a $\{\mathcal{Y}_t\}$-predictable selection of indices $(t, \omega) \mapsto (i^{\star}(t, \omega), k^{\star}(t, \omega))$ such that
		\begin{equation*}
		\Delta_{\infty}(t, \omega) = \Delta_{i^{\star}(t, \omega) k^{\star}(t, \omega)}(t, \omega) \quad \text{for all $t, \omega$}.
		\end{equation*}
		Moreover, the dynamics of $\Delta_{\infty}(t, \omega)$ are given by
		\begin{align}\label{eq:evol_delta_infty}
		\di \Delta_{\infty}(t) &= 
		\sum_{i \in \mathbf{N}} \sum_{k \in \mathbf{N}} {\bf{1}}_{ \{(i^{\star}, k^{\star})(t) = (i,k) \} } \frac{\di}{\dt} \Delta_{ik}(t) \dt, \nonumber \\
		\Delta_{\infty}(0) &= \log \frac{\mu^{i^{\star}(0)}}{\mu^{k^{\star}(0)}} - \log \frac{\nu^{i^{\star}(0)}}{\nu^{k^{\star}(0)}}.
		\end{align}
	\end{lemma}
	\begin{proof}
		Consider the measurable space $(M, \mathcal{M})$, where $M = ([0,\infty) \times \Omega)$ and $\mathcal{M}$ is the $\{\mathcal{Y}_t\}$-predictable $\sigma$-algebra. Let $U = \mathbf{N} \times \mathbf{N}$ endowed with the discrete topology. Consider the function $f: M \times U \rightarrow \R$ such that $f((t, \omega), (i,k)) = \Delta_{ik} (t, \omega)$. Note that $z(\cdot, (i,k)) = \Delta_{ik} (\cdot)$ is $\mathcal{M}$-measurable for all $(i,k) \in U$ by Proposition~\ref{prop:evol_theta}. Moreover, $z((t,\omega), \cdot) = \Delta_{\cdot} (t,\omega)$ is continuous as a function $U \rightarrow \R$ (because it is defined on the discrete space $U = \mathbf{N}\times \mathbf{N}$). The function $\Delta_{\infty} : M \rightarrow \R$ is $\mathcal{M}$-measurable by Lemma \ref{lemma:delta_inf_measurable}. Since $\Delta_{\infty} = \max_{(i,k) \in \mathbf{N} \times \mathbf{N}} \Delta_{ik}$, we must have that the image of $\Delta_{\infty}$ is contained in the image of $f$. In other words, we have
		\begin{equation*}
		\Delta_{\infty} (t,\omega) \in f((t,\omega), U) \quad \forall(t,\omega) \in M.
		\end{equation*}
		Then by Filippov's implicit function lemma (see e.g. \cite[Theorem~A.10.2]{cohen15}) there exists an $\mathcal{M}$-measurable (i.e.~$\{\mathcal{Y}_t\}$-predictable) map $u: M \rightarrow U$ that maps $(t,\omega) \mapsto (i^{\star}(t,\omega), k^{\star}(t,\omega))$ such that
		\begin{equation*}
		\Delta_{\infty}(t, \omega) = f\big((t, \omega), u(t,\omega)\big) = f\big((t, \omega), (i^{\star}(t,\omega), k^{\star}(t,\omega))\big) =  \Delta_{i^{\star}(t, \omega) k^{\star}(t, \omega)}(t, \omega).
		\end{equation*}
		To prove the second part of the Lemma, recall that by Lemma \ref{lemma:delta_inf_measurable} we have that $\Delta_{\infty}(t, \omega)$ is absolutely continuous. Then $\di \Delta_{\infty}(t) = g(t) \dt$ for some density $g(t)$ such that $\int_0^t |g(s)|\ds < \infty$ a.s. for each $t \ge 0$. Since $\sum_{i \in \mathbf{N}} \sum_{k \in \mathbf{N}} {\bf{1}}_{ \{(i^{\star}, k^{\star})(t) = (i,k) \} } = 1$, we can write
		\begin{equation*}
		\Delta_{\infty}(t) = \Delta_{\infty}(0) + \int_0^t \sum_{i \in \mathbf{N}} \sum_{k \in \mathbf{N}} {\bf{1}}_{ \{(i^{\star}, k^{\star})(t) = (i,k) \} } g(s) \ds.
		\end{equation*}
		So, if we can show that for any $(i,k) \in \mathbf{N} \times \mathbf{N}$ and any $t > 0$ we have
		\begin{equation}\label{eq:lemma_eq_max_diff}
		\int_0^t {\bf{1}}_{ \{(i^{\star}, k^{\star})(t) = (i,k) \} } \Big| g(s) - \frac{\di}{\dt} \Delta_{ik}(s) \Big| \ds = 0 \quad \mathrm{a.s.},
		\end{equation}
		we are done. Rewriting the left-hand side of the above, we have
		\begin{align*}
		0 &\le \int_0^t {\bf{1}}_{ \{(i^{\star}, k^{\star})(t) = (i,k) \} } \Big| g(s) - \frac{\di}{\dt} \Delta_{ik}(s) \Big| \ds \\
		&\le
		\int_0^t {\bf{1}}_{ \{ \Delta_{\infty}(s) - \Delta_{ik}(s) = 0 \} } \Big| g(s) - \frac{\di}{\dt} \Delta_{ik}(s) \Big| \ds \\
		&= \int_0^t {\bf{1}}_{ \big\{ \Delta_{\infty}(s) - \Delta_{ik}(s) = 0, \: g(s) - \frac{\di}{\dt} \Delta_{ik}(s) \neq 0 \big\} } \Big| g(s) - \frac{\di}{\dt} \Delta_{ik}(s) \Big| \ds,
		\end{align*}
		and since the set $\big\{ s \, : \, \Delta_{\infty}(s) - \Delta_{ik}(s) = 0, \: g(s) - \frac{\di}{\dt} \Delta_{ik}(s) \neq 0 \big\}$ has measure $0$ by Lemma \ref{lemma:dupuis}, we see \eqref{eq:lemma_eq_max_diff} holds and the proof is complete.
	\end{proof}
	
	Finally, \eqref{eq:Delta_equal_H} gives us the chance to spell out the following lemmata, which will be useful later.
	\begin{lemma}\label{lemma:argmax}
		Assume $\mu, \nu \in \mathring{\mathcal{S}}^n$. For all $t < \infty$, the indices $i^{\star}(t,\omega)$ and $k^{\star}(t,\omega)$ respectively maximize and minimize the quantity $\frac{\pi_t^j}{\tilde \pi_t^j}$ over $j \in \mathbf{N}$. Moreover, we have that $ \frac{\pi_t^{i^{\star}}}{\tilde \pi_t^{i^{\star}}}  = :M_t \ge 1$ and $ \frac{\pi_t^{k^{\star}}}{\tilde \pi_t^{k^{\star}}}  = :\frac{1}{m_t} \le 1$ for all $t < \infty$.
	\end{lemma}
	\begin{proof}
		Fix $(t, \omega) \in [0, \infty) \times \Omega$. Recalling Lemma \ref{lemma:measurable_selection}, we see from the definition of $\Delta_{ik}$ and \eqref{eq:Delta_equal_H} that
		\begin{align*}
		\log \frac{\pi_t^{i^{\star}(t, \omega)}}{ \tilde \pi_t^{i^{\star}(t, \omega)}}(\omega)  - \log  \frac{\pi_t^{k^{\star}(t, \omega)}}{\tilde\pi_t^{k^{\star}(t, \omega)}}(\omega)
		&= \Delta_{i^{\star}(t, \omega) k^{\star}(t, \omega)}(t, \omega)\\
		&= \Delta_{\infty}(t, \omega)
		= \log \max_{i \in \mathbf{N}} \frac{\pi_t^i}{\tilde \pi_t^i}(\omega)  - \log \min_{k \in \mathbf{N}} \frac{\pi_t^k}{\tilde\pi_t^k}(\omega),
		\end{align*}
		so the first part of the lemma follows. For the second part, assume for contradiction that there exists $\omega \in \Omega$ such that $M_t(\omega) < 1$. Then, for all $j \in \mathbf{N}$
		\begin{equation*}
		\frac{\pi^j_t}{\tilde \pi_t^j}(\omega) \le M_t(\omega) < 1 \Longrightarrow \pi^j_t < \tilde \pi_t^j,
		\end{equation*}
		which implies that $\sum_{j} \pi_t^j < 1$, and contradicts the fact that $\pi_t$ is a probability distribution. The argument for $1/m_t(\omega)$ is analogous.
	\end{proof}
	
	\begin{lemma}\label{lemma:diff_pi_pi_tilde_jk}
		Assume $\mu, \nu \in \mathring{\mathcal{S}}^n$. For all $i,k \in \mathbf{N} \times \mathbf{N}$, define $T_{ik}(t) := \frac{\pi_t^i}{\pi_t^k} - \frac{\tilde\pi_t^i}{\tilde\pi_t^k}$. For all $t < \infty$, we have that $T_{j i^{\star}}(t) \le 0$ and $T_{j k^{\star}}(t) \ge 0$, where $i^{\star} = i^{\star}(t,\omega)$ and $k^{\star} = k^{\star}(t,\omega)$ are the maximizing/minimizing indices from Lemma~\ref{lemma:argmax}.
	\end{lemma}
	\begin{proof}
		Trivially, $T_{i^{\star}i^{\star}}(t) = T_{k^{\star}k^{\star}}(t) = 0$. Now consider $T_{j k^{\star}}(t)$ for $j \neq k^{\star}$. Note that
		\begin{gather*}
		T_{j k^{\star}}(t) 
		=  \frac{\pi_t^j}{\pi_t^{k^{\star}}} - \frac{\tilde\pi_t^j}{\tilde\pi_t^{k^{\star}}} 
		=  \bigg( \frac{\pi_t^j}{\tilde\pi_t^j} - \frac{\pi_t^{k^{\star}}}{\tilde\pi_t^{k^{\star}}} \bigg) \frac{\tilde\pi_t^j}{\pi_t^{k^{\star}}}, \quad \forall j \in \mathbf{N}, j \neq k^{\star}.
		\end{gather*}
		By Lemma \ref{lemma:argmax}, $k^{\star}$ minimizes $\frac{\pi_t^j}{\tilde \pi_t^j}$ over $j \in \mathbf{N}$, so we have that $ \frac{\pi_t^j}{\tilde\pi_t^j} - \frac{\pi_t^{k^{\star}}}{\tilde\pi_t^{k^{\star}}}  \ge 0$ for all $j \neq k^{\star}$. Moreover, $\tilde\pi_t^j / \pi_t^{k^{\star}} > 0$ as well, since $\pi_t$ and $\tilde\pi_t$ have positive entries for $t < \infty$. We conclude that $T_{j k^{\star}}(t) \ge 0$ for all $t < \infty$.
		
		For the case of $T_{j i^{\star}}(t)$ we argue in the same way by noting that
		\begin{gather*}
		T_{j i^{\star}}(t) 
		=  \frac{\pi_t^j}{\pi_t^{i^{\star}}} - \frac{\tilde\pi_t^j}{\tilde\pi_t^{i^{\star}}} 
		=  - \bigg( \frac{\pi_t^{i^{\star}}}{\tilde\pi_t^{i^{\star}}} - \frac{\pi_t^j}{\tilde\pi_t^j} \bigg) \frac{\tilde\pi_t^j}{\pi_t^{i^{\star}}}, \quad \forall j \in \mathbf{N}, j \neq i^{\star},
		\end{gather*}
		and using that $i^{\star}$ maximizes $\frac{\pi_t^j}{\tilde \pi_t^j}$.	
	\end{proof}
	
	\subsection{Proof of Theorem \ref{thm:contraction}}\label{sec:proof_contraction}
	
	We are now ready to prove Theorem \ref{thm:contraction}.
	
	\begin{proof}[Proof of Theorem~\ref{thm:contraction}]
		Let us start by considering \eqref{eq:evol_delta_infty}, and assuming $\mu, \nu \in \mathring{\mathcal{S}}^n$. Writing it out in full we have
		\begin{align} \label{eq:di_Delta_inf}
		\di \Delta_{\infty}(t)
		&= \sum_{i \in \mathbf{N}} \sum_{k \in \mathbf{N}} {\bf{1}}_{ \{(i^{\star}, k^{\star})(t) = (i,k) \} } \frac{\di}{\dt}  \bigg( \log  \frac{\pi_t^i}{\pi_t^k}   - \log  \frac{\tilde\pi_t^i}{\tilde\pi_t^k} \bigg) \dt \nonumber \\
		&= \bigg[ -\sum_{\substack{j = 0 \\ j\neq k^{\star}}}^{n} q_{j k^{\star}} \left( \frac{\pi_t^j}{\pi_t^{ k^{\star}}} - \frac{\tilde\pi_t^j}{\tilde\pi_t^{ k^{\star}}} \right)
		+ \sum_{\substack{j = 0 \\ j\neq i^{\star}}}^{n} q_{j i^{\star}} \left( \frac{\pi_t^j}{\pi_t^{i^{\star}}} - \frac{\tilde\pi_t^j}{\tilde\pi_t^{i^{\star}}} \right) \bigg] \dt,
		\end{align}
		where we have dropped the $(t, \omega)$-dependence of $(i^{\star},k^{\star})$ for readability.
		
		Rewriting \eqref{eq:di_Delta_inf} in the notation of Lemma~\ref{lemma:diff_pi_pi_tilde_jk}, we have
		\begin{align}\label{eq:di_Delta_sum}
		\di \Delta_{\infty} (t) &= -\bigg[ \sum_{\substack{j = 0 \\ j\neq k^{\star}}}^{n} q_{j k^{\star}} T_{j k^{\star}}(t)
		- \sum_{\substack{j = 0 \\ j\neq i^{\star}}}^{n} q_{j i^{\star}} T_{ji^{\star}}(t) \bigg] \dt \nonumber \\
		&= - \big( q_{i^{\star} k^{\star}} T_{i^{\star} k^{\star}}(t) -  q_{k^{\star} i^{\star}} T_{k^{\star} i^{\star}}(t) \big) \dt
		- \hspace{-8pt} \sum_{\substack{j = 0 \\ j\neq i^{\star},k^{\star}}}^{n} \hspace{-8pt} \bigg( q_{j k^{\star}} T_{j k^{\star}}(t) -  q_{j i^{\star}} T_{ji^{\star}}(t) \bigg) \dt.
		\end{align}
		and the right-hand side is non-positive, since the off-diagonal entries of the $Q$-matrix are non-negative by definition, and the differences $q_{i^{\star} k^{\star}} T_{i^{\star} k^{\star}}(t) -  q_{k^{\star} i^{\star}} T_{k^{\star} i^{\star}}(t)$ and $q_{j k^{\star}} T_{j k^{\star}}(t) -  q_{j i^{\star}} T_{ji^{\star}}(t)$ for $j \neq i^{\star}, k^{\star}$ are also all non-negative by Lemma~\ref{lemma:diff_pi_pi_tilde_jk}. We look for an upper bound on the Stieltjes measure $\di \Delta_{\infty}(t)$ on $[0, \infty)$. Dropping non-positive terms in the sum, we simplify to
		\begin{align*}
		\di \Delta_{\infty} (t) 
		&\le - \big( q_{i^{\star} k^{\star}} T_{i^{\star} k^{\star}}(t) -  q_{k^{\star} i^{\star}} T_{k^{\star} i^{\star}}(t) \big) \dt \\
		&= - \Bigg[ q_{i^{\star} k^{\star}}
		\bigg( \frac{\pi_t^{i^{\star}}}{\pi_t^{k^{\star}}} - \frac{\tilde\pi_t^{i^{\star}}}{\tilde\pi_t^{k^{\star}}} \bigg)
		+ q_{k^{\star} i^{\star}}
		\bigg( \frac{\tilde\pi_t^{k^{\star}}}{\tilde\pi_t^{i^{\star}}} - \frac{\pi_t^{k^{\star}}}{\pi_t^{i^{\star}}} \bigg) \Bigg] \dt \\
		&= - \Bigg[ q_{i^{\star} k^{\star}}
		\bigg( M_t - \frac{1}{m_t} \bigg) \frac{\tilde \pi_t^{i^{\star}}}{\pi_t^{k^{\star}}}
		+ q_{k^{\star} i^{\star}}
		\bigg( m_t - \frac{1}{M_t} \bigg)  \frac{ \pi_t^{k^{\star}}}{\tilde \pi_t^{i^{\star}}} \Bigg] \dt \\
		&\le - 2 \sqrt{q_{i^{\star} k^{\star}} q_{k^{\star} i^{\star}}} \bigg( \frac{M_t m_t - 1}{\sqrt{M_t m_t}} \bigg) \dt,
		\end{align*}
		where $M_t : = \frac{\pi_t^{i^{\star}}}{\tilde \pi_t^{i^{\star}}}$ and $m_t : = \frac{\tilde \pi_t^{k^{\star}}}{ \pi_t^{k^{\star}}}$ as in Lemma \ref{lemma:argmax}, and we have made use of the inequality $a + b \ge 2 \sqrt{ab}$ for $a,b \ge 0$. Recall that $\Delta_{\infty}(t) = \Delta_{i^{\star} k^{\star}} (t) = \log (M_t m_t) \ge 0$, since $M_t m_t \ge 1$ by Lemma~\ref{lemma:argmax}. Then we can rewrite the inequality above as
		\begin{align*}
		\di \Delta_{\infty} (t)
		&\le - 4 \sqrt{q_{i^{\star} k^{\star}} q_{k^{\star} i^{\star}}} \sinh \bigg(\frac{\Delta_{\infty}(t)}{2} \bigg)  \dt
		\end{align*}
		Now, if $\Delta_{\infty}(t) = 0$, then the theorem holds trivially, so we can assume $\Delta_{\infty}(t) >0$. Since $\sinh(x) >0 $ for $x >0$, we can divide both sides by $\sinh \big( \Delta_{\infty}(t)/2 \big)$. Integrating over $[s,t]$ yields
		\begin{equation*}
		\log \tanh \bigg( \frac{\Delta_{\infty}(t)}{4} \bigg) - \log \tanh \bigg( \frac{\Delta_{\infty}(s)}{4} \bigg) \le - \lambda (t-s),
		\end{equation*}
		where we have defined $\lambda := 2 \min_{i \neq j} \sqrt{q_{ij}q_{ji}}$, and it follows that
		\begin{equation*}
		\tanh \bigg( \frac{\Delta_{\infty}(t)}{4} \bigg) \le \tanh \bigg( \frac{\Delta_{\infty}(0)}{4} \bigg) e^{- \lambda t}.
		\end{equation*}
		
		Concavity and monotonicity of $\tanh (x)$ for $x \ge 0$ imply that  $\tanh(x)e^{-\lambda t} \leq \tanh(xe^{-\lambda t})$ and hence
		\begin{equation*}
		\Delta_{\infty}(t) \le \Delta_{\infty}(0) e^{-\lambda t},
		\end{equation*}
		and, recalling \eqref{eq:Delta_equal_H}, we are done.

		Finally, we lift the assumption that $\mu, \nu \in \mathring{\mathcal{S}}^n$. From Lemma \ref{lemma:hilbertsemicts}, we can choose sequences $\mu_m, \nu_m \in \mathring{\mathcal{S}}^n$ such that $\mathcal{H}(\mu_m, \nu_m)\to \mathcal{H}(\mu,\nu)$, and $(\mu_m,\nu_m)\to (\mu, \nu)$ in the Euclidean norm. Consider the corresponding filters $\pi_t(\mu_m), \tilde\pi_t(\nu_m)$ (the solutions to \eqref{eq:wonhamNorm} and \eqref{eq:wonhamWrong} initialized at $\mu_m$ and $\nu_m$ respectively).   As the Kushner--Stratonovich equations are Lipschitz on $\mathcal{S}^n$, by standard stability results for SDEs (e.g. \cite[Theorem~16.4.3]{cohen15}), we know that $\pi_t(\mu_m)\to \pi_t$ and $\tilde\pi_t(\nu_m) \to \tilde\pi_t$ in probability as $m\to \infty$. 

  We know from Lemma \ref{lemma:non_divergence_to_boundary2} that $\pi_t, \tilde\pi_t\in \mathring{\mathcal{S}}^n$. Using the continuity given in Lemma \ref{lemma:hilbertsemicts}, and applying the result above, we know that 
    \[\begin{split}\tanh\Big(\frac{\mathcal{H}(\pi_t, \tilde\pi_t)}{4}\Big) &=\lim_{m \to \infty}\Big[\tanh\Big(\frac{\mathcal{H}(\pi_t(\mu_m), \tilde\pi_t(\nu_m))}{4}\Big)\Big]\\&\leq \lim_{m \to \infty}\Big[\tanh\Big(\frac{\mathcal{H}(\mu_m, \nu_m)}{4}\Big)\Big]e^{- \lambda t} = \tanh\Big(\frac{\mathcal{H}(\mu, \nu)}{4}\Big)e^{- \lambda t}\end{split}\]
as desired. Monotonicity and concavity of $\tanh$ again complete the argument.
	\end{proof}

	\subsection{On the optimality of the contraction rate}\label{sec:optimality_rate}
	
	The deterministic contraction rate $\lambda = 2 \min_{i \neq j} \sqrt{q_{ij}q_{ji}}$ that we just proved is sharp for the case $\pi_t, \tilde \pi_t \in \mathcal{S}^1$ uniformly in $\mu, \nu \in \mathring{\mathcal{S}}^1$, in the sense that if we have $\rho \in \R$ s.t. $\Delta_{\infty}(t) \le \Delta_{\infty}(s) e^{- \rho (t-s)}$ $a.s.$ for all $s < t$, we know $\rho \le \lambda$. In this basic case, the maximum process is simply given by $\Delta_{\infty} = |\theta_0^1 - \tilde \theta_0^1|$. Consider the specific situation when $Q$ is symmetric, so that the diagonal entries are given by $q_{00} = q_{11} = -q$ and the off-diagonal entries by $q_{01} = q_{10} = q$. Then $\lambda = 2 q$. We compute
	\begin{align*}
	\di |\theta_0^1(t) - \tilde \theta_0^1(t)|
	&= \operatorname{sign} \big(\theta_0^1(t) - \tilde \theta_0^1(t)\big) \Big[- q \big(e^{\theta_0^1(t)} - e^{\tilde \theta_0^1(t)} + e^{- \tilde \theta_0^1(t)} - e^{-\theta_0^1(t)} \big)\Big] \dt \\
	&= \operatorname{sign} \big(\theta_0^1(t) - \tilde \theta_0^1(t)\big) \Big[- 2q \big[ (\theta_0^1(t) - \tilde \theta_0^1(t)) + \frac{1}{3!}((\theta_0^1(t))^3 - (\tilde \theta_0^1(t))^3) + \dots  \big] \Big] \dt \\
	&= - 2q |\theta_0^1(t) - \tilde \theta_0^1(t)| \Big( 1 + \frac{1}{3!}\big( (\theta_0^1(t))^2 + (\tilde \theta_0^1(t))^2 + \theta_0^1(t) \tilde \theta_0^1(t) \big) + \dots    \Big) \dt \\
	&= : - 2q |\theta_0^1(t) - \tilde \theta_0^1(t)| R_t \dt,
	\end{align*}
	from which we deduce
	\begin{equation}\label{eq:expression_1d_max_process}
	|\theta_0^1(t) - \tilde \theta_0^1(t)| = e^{- 2 q \int_s^t R_u \di u} |\theta_0^1(s) - \tilde \theta_0^1(s)|, \quad \text{ for all } 0 \le s \le t.
	\end{equation}
	Note that $R_t$ is close to 1 iff $\theta_0^1(t) \approx \tilde \theta_0^1(t) \approx 0$, which happens if $\pi_t$ and $\tilde \pi_t$ are near the centre of the simplex, i.e. $\pi_t \approx \tilde \pi_t \approx \big(\frac{1}{2}, \frac{1}{2} \big)$. 
	Let $\tau < \infty$ be a time such that $R_{\tau}\approx 1$. Note that such $\tau$ exists with positive probability, since the Brownian dynamics (under a change of measure) for $\theta_0^1(t)$ ensure that $\theta_0^1(t)$ must visit $0$ infinitely many times;~then by Theorem~\ref{thm:contraction}, for all large enough $\tau$, there exists $\varepsilon \ll 1$ such that $|\tilde \theta_0^1(\tau) - \theta_0^1(\tau) | \le \varepsilon$.
	
	Since $R_{t}$ is continuous in time, for all $\delta > 0$ there exists $t_{\delta}$ such that $R_s \in (1, 1 + \delta]$ for all $s \in [\tau, \tau + t_{\delta}]$. Then by \eqref{eq:expression_1d_max_process} for all $s \in [\tau, \tau + t_{\delta}]$,
	\begin{equation*}
	|\theta_0^1(\tau + \tau_{\delta}) - \tilde \theta_0^1(\tau + \tau_{\delta})| \ge e^{- 2 q (1 + \delta)t_{\delta}} \Big|\theta_0^1(\tau) - \tilde \theta_0^1(\tau) \Big|.
	\end{equation*}
	Since $\delta$ was arbitrary, we see that the bound $\lambda = 2q$ is achieved. We illustrate this in a simulated example in Figure~\ref{fig:hilber_error_1d}(left).
	
	\begin{figure}[H]
		\centering
		\includegraphics[width=0.49\textwidth]{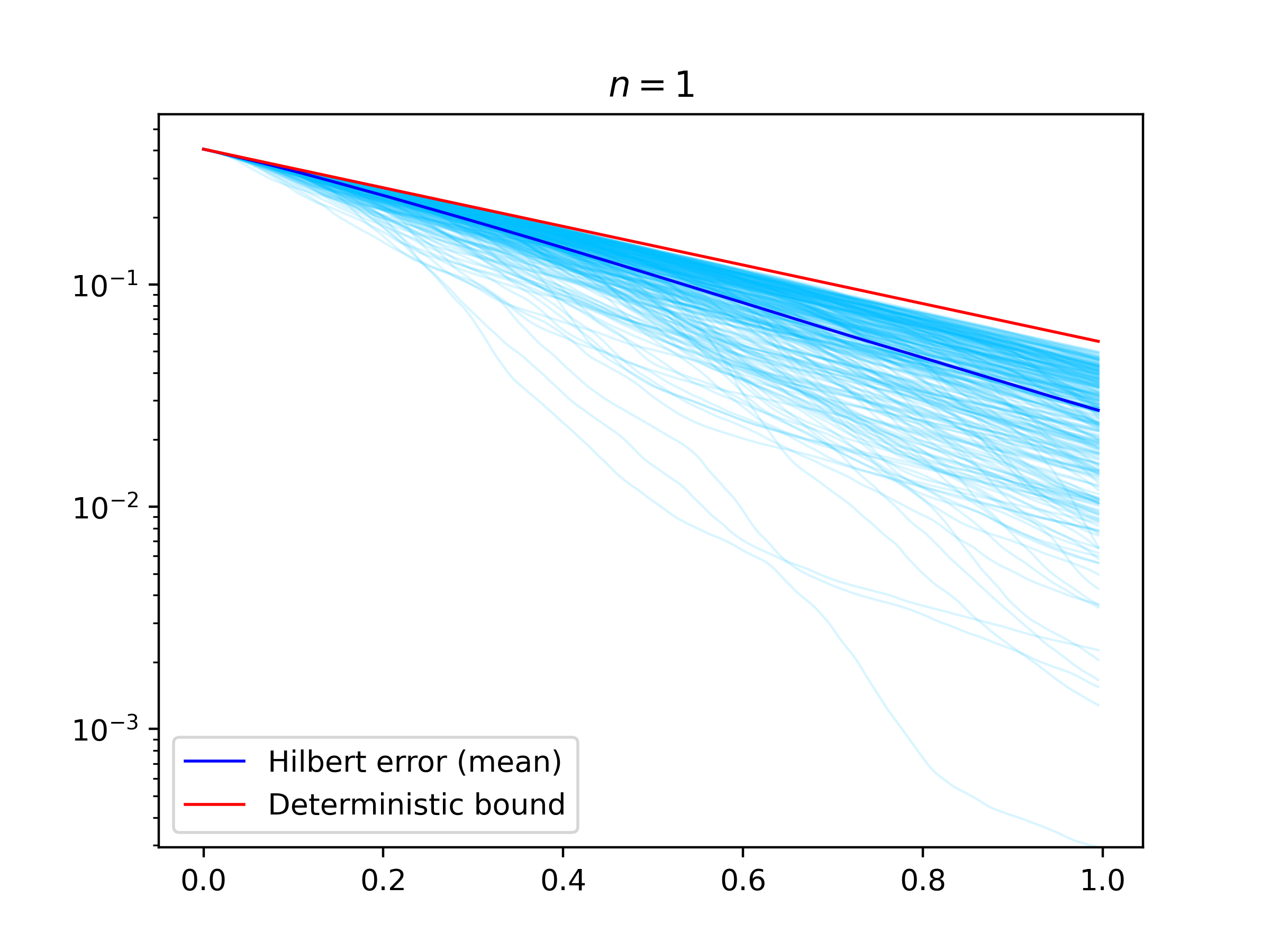}
		\includegraphics[width=0.49\textwidth]{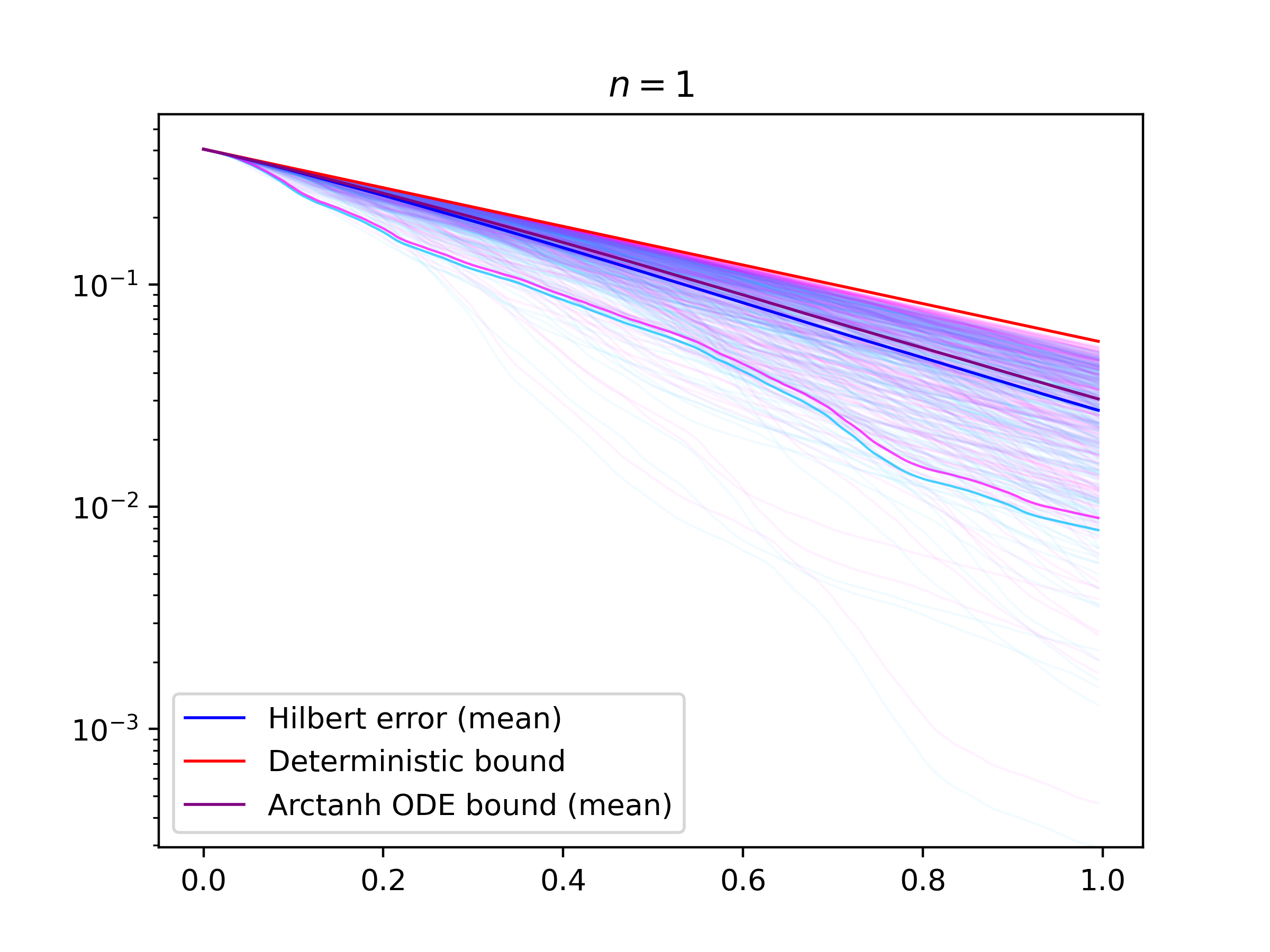}	
		\caption{On the left, in log scale, we plot 300 realizations of the Hilbert projective error $\hilbert (\pi_t, \tilde \pi_t)$ (in light blue), where $\pi_t, \tilde \pi_t \in \mathring{\mathcal{S}}^1 \subset \R^2$ are solutions to \eqref{eq:wonhamNorm} initialized at $\mu, \nu \in \mathring{\mathcal{S}}^1$ respectively. In blue we plot the sample mean, and in red we have the deterministic pathwise bound $\hilbert(\mu, \nu) e^{-\lambda t}$ from Theorem~\ref{thm:contraction}. We see that the bound is attained. For this simulation, the signal $X$ is a 2-state Markov chain with symmetric rate matrix $Q$ and jump rate $q=1$, initialized at its invariant distribution $\mu = \big( \frac{1}{2}, \frac{1}{2} \big)$. This gives $\pi_0 = \mu$ as the initial condition for the Wonham filter. The `wrong' filter $\tilde \pi_t$ is initialized at $\nu = \big(\frac{2}{5}, \frac{3}{5} \big)$. Fixing these parameters, and setting the sensor function to be $h = (-1, 1)$, we simulate 300 paths for the signal and the observation processes, compute the two filters for these paths, and plot their Hilbert error. On the right, we plot once more the realizations of the Hilbert error (very light blue), and add to the same plot the pathwise ODE bounds (fuchsia) given by solving numerically, for each realization of $\tilde \pi_t$, the ODE \eqref{eq:ode_tanh_stability} from Proposition~\ref{prop:pathwise_bound_stability_ode}. In purple we plot the mean of the ODE bounds;~out of 300 pairs of Hilbert error and ODE bound, we highlight one at random.
		}
		\label{fig:hilber_error_1d}
	\end{figure}
	
	On the other hand, as the dimension of the state-space increases, numerical experiments suggest that $\lambda$ becomes less optimal. Recalling the notation from the previous subsection and looking back at our proof, it is easy to pinpoint the cause of this sub-optimality to having discarded the negative sum of terms of the form $q_{j k^{\star}} T_{j k^{\star}} - q_{j i^{\star}} T_{j i^{\star}}$ on the right-hand side of equation \eqref{eq:di_Delta_sum}. In particular, there are $n-2$ such negative terms, and they can take values in $(0, \infty)$, which suggests that, as $n$ increases, the derivative of $\Delta_{\infty}$ becomes more negative (and potentially quite substantially so), and $\Delta_{\infty}$ should in fact tend to 0 faster than our bound indicates.
	
	Unfortunately, we have not been able to find a uniform bound from below of the form $K_{q} \Delta_{\infty}$ for $\sum_j (q_{j k^{\star}} T_{j k^{\star}} - q_{j i^{\star}} T_{j i^{\star}})$, where $K_q$ is some constant depending only on $Q$. However, assuming one can observe the path of the wrongly initialized filter $\tilde \pi_t$, we can provide a sharper, $\tilde \pi_t$-dependent bound for the decay rate.
	
	We will need the following classical result, which we include for completeness.
	
	\begin{lemma}[Comparison principle]\label{lemma:comparison}
		Let $X_t \in \R$ be an absolutely continuous process such that its almost everywhere derivative satisfies
		\begin{equation*}
		\di X_t \le \alpha(t, X_t) \dt, \qquad X_0 = x_0,
		\end{equation*}
		on $[0,\infty)$, where $x \mapsto \alpha(t,x)$ is locally Lipschitz continuous. Let $u_t$ be the unique solution (up to its first explosion time $T>0$) to the ODE
		\begin{equation}\label{eq:ode_u_lemma}
		\frac{\di u_t}{\dt} = \alpha(t, u_t), \qquad u_0 = x_0.
		\end{equation} 
		Then $X_t \le u_t$ for all $t < T$.
	\end{lemma}
	\begin{proof}
		First of all, recall that standard results in ODE theory (see e.g.~\cite[Theorem~2.5]{teschl_odes}) give that \eqref{eq:ode_u_lemma} has a unique solution $u_t$ up to its first explosion time $T > 0$. For $t < T$, consider $H_t = X_t - u_t$. Note that $H_0 = 0$, and that $H_t$ is absolutely continuous with  a.e.~derivative satisfying
		\begin{equation*}
		\frac{\di H_t}{\dt} \le  \alpha(t, X_t) -  \alpha(t, u_t).
		\end{equation*} 
		Assume for a contradiction that there exists $\tau < T$ s.t.~$H_{\tau} > 0$. By continuity of $H_t$, there exists $t_0 \in [0, \tau)$ such that $H_{t_0} = 0$ and $H_{s} \ge 0$ for all $s \in [t_0, \tau]$. Moreover, by continuity of $X_t$ and $u_t$, there exists $R \in \R$ such that $X_s, u_s \in (-R, R)$ for all $s \in [t_0, \tau]$. Then
		\begin{align*}
		H_{\tau} &= \int_{t_0}^{\tau} \hspace{-3pt} \frac{\di H_s}{\ds} \ds 
		\le \int_{t_0}^{\tau} \hspace{-5pt} \big( \alpha(s, X_s) - \alpha(s, u_s) \big) \ds
		\le \int_{t_0}^{\tau} \hspace{-5pt} C_R(s) |X_s - u_s| \ds = \int_{t_0}^{\tau} \hspace{-5pt} C_R(s) H_s \ds,
		\end{align*}
		where $C_R(t) \ge 0$ is the Lipschitz constant of $\alpha(t, x_t)$ for $x_t \in (-R, R)$, and we have used that $H_s = X_s - u_s > 0$ on $[t_0, \tau]$ by assumption. Then Gr\"onwall's inequality yields that $H_{\tau} \le 0$, which is a contradiction. Therefore $X_t \le u_t$ for all $t < T$.
	\end{proof}

\begin{prop}\label{prop:pathwise_bound_stability_ode}
	Suppose $\mu^i, \nu^i > 0$ for all $i \in \mathbf{N}$.
	For all $t \ge 0$, let $u_t \in (0,1)$ be the unique solution to the ODE with random coefficients given by
	\begin{equation}\label{eq:ode_tanh_stability}
	\frac{\di u_t}{\dt} = - \tilde \lambda^{\star} (t, u_t) u_t,
	\qquad u_0 = \tanh \bigg( \frac{\hilbert( \mu, \nu)}{4}   \bigg),
	\end{equation}
	where
	\begin{equation}\label{eq:lambda_star_tilde}
	\tilde \lambda^{\star}(t, u_t)
	= \min_{i \neq k} \bigg\{\bigg( q_{ik} \frac{\tilde \pi_t^i}{\tilde \pi_t^k}
	+ \hspace{-5pt} \sum_{\substack{ j \neq i,k, \\ j \notin \mathcal{\tilde J}^i_k(t, u_t)}} \hspace{-5pt} q_{j k} \frac{\tilde \pi_t^j}{\tilde \pi_t^k}\bigg) \frac{1 + u_t}{1-u_t}
	+ \bigg( q_{ki} \frac{\tilde \pi_t^k}{\tilde \pi_t^i}
	+ \hspace{-5pt} \sum_{\substack{ j \neq i,k, \\ j \in \mathcal{\tilde J}^i_k(t, u_t)}} \hspace{-5pt} q_{j i} \frac{\tilde \pi_t^j}{\tilde \pi_t^i}\bigg) \frac{1 - u_t}{1+u_t} \bigg\},
	\end{equation}
	and
	\begin{equation*}
	\mathcal{\tilde J}^i_k(t, u_t) := \Big\{ j \in \mathbf{N} \, : \, \frac{q_{j k}}{\tilde \pi_t^k} \ge \frac{q_{j i}}{\tilde \pi_t^i} \Big( \frac{1-u_t}{1+u_t} \Big)^2 \Big\}.
	\end{equation*}
	Then for all $t < \infty$,
	\begin{equation}\label{eq:tanh_less_u_stability}
	\tanh \bigg( \frac{\hilbert (\pi_t, \tilde \pi_t)}{4} \bigg) \le u_t.
	\end{equation}
	
	In particular, $\tilde 
	\lambda^{\star}(t, u_t) \ge \tilde \lambda^{\star}_t$, where
	\begin{align}
	\tilde \lambda^{\star}_t 
	&:= 2\min_{i \neq k}
	\left\{ \min_{S \subseteq \mathbf{N}} \sqrt{  q_{ik}q_{k i}
		+  \hspace{-7pt} \sum_{j \in S,\, j \neq i,k}  \hspace{-7pt} q_{ik} q_{j i} \frac{\tilde \pi_t^j}{\tilde \pi_t^k}
		+ \hspace{-7pt} \sum_{l \notin S, \, l \neq i,k} \hspace{-7pt} q_{ki} q_{l k} \frac{\tilde \pi_t^l}{\tilde \pi_t^i}
		+  \hspace{-3pt}   \sum_{\subalign{\: j &\in S, \\ j &\neq i,k}} \sum_{\subalign{\: l &\notin S, \\ l &\neq i,k}}  \hspace{-3pt} q_{ji} q_{l k} \frac{\tilde \pi_t^j \tilde \pi_t^l}{\tilde \pi_t^i \tilde \pi_t^k} } \right\}, \nonumber \\
	&\ge 2\min_{i \neq k} \left\{ \sqrt{ q_{i k} q_{k i}
		+ \sum_{j \neq i, k} \min \bigg\{ \frac{ q_{j i} q_{i k}}{\tilde \pi_t^k}, \frac{q_{j k} q_{k i}}{\tilde \pi_t^i} \bigg\}  \tilde \pi_t^j } \right\} =: \tilde \lambda_t, \label{eq:pathwise_rate_hilbert}
	\end{align}
	which gives that for all $t < \infty$,
	\begin{equation}\label{eq:pathwise_ineq_tanh}
	\tanh \bigg( \frac{\hilbert (\pi_t, \tilde \pi_t)}{4} \bigg) \le \tanh \bigg( \frac{\hilbert (\mu, \nu)}{4} \bigg) e^{- \int_0^t \tilde \lambda^{\star}_s \ds}.
	\end{equation}
\end{prop}
\begin{proof}
	Recall the notation from the proof of Theorem~\ref{thm:contraction}. Consider \eqref{eq:di_Delta_sum} and apply the chain-rule to derive the dynamics of $\tanh (\Delta_{\infty}(t)/4)$, to yield
	\begin{align}
	\di \tanh \bigg( \frac{\Delta_{\infty}(t)}{4}   \bigg) 
	= &- \frac{1}{4}\cosh^{-2} \bigg( \frac{\Delta_{\infty}(t)}{4}  \bigg) \big( q_{i^{\star} k^{\star}} T_{i^{\star} k^{\star}}(t) -  q_{k^{\star} i^{\star}} T_{k^{\star} i^{\star}}(t) \big) \dt \nonumber \\
	&-\frac{1}{4} \cosh^{-2} \bigg( \frac{\Delta_{\infty}(t)}{4}  \bigg) \hspace{-8pt} \sum_{\substack{j = 0 \\ j\neq i^{\star},k^{\star}}}^{n} \hspace{-8pt} \bigg( q_{j k^{\star}} T_{j k^{\star}}(t) -  q_{j i^{\star}} T_{ji^{\star}}(t) \bigg) \dt. \label{eq:di_tanh_delta_infty}
	\end{align}
	We consider the terms on the right-hand side of the above equation one by one. Start from $T_{i^{\star} k^{\star}}(t)$ and notice that
	\begin{align*}
	\fontdimen14\scriptfont2=4pt
	\fontdimen17\scriptfont2=0pt
	T_{i^{\star} k^{\star}}(t)
	&= e^{\tilde \theta^{i^{\star}}_{k^{\star}}(t)} \big( e^{\theta^{i^{\star}}_{k^{\star}}(t) - \tilde \theta^{i^{\star}}_{k^{\star}}(t)} - 1 \big)
	= e^{\tilde \theta^{i^{\star}}_{k^{\star}}(t)} \big( e^{\Delta_{\infty}(t)} - 1 \big)
	= e^{\tilde \theta^{i^{\star}}_{k^{\star}}(t) + \frac{\Delta_{\infty}(t)}{2}} \big( e^{\frac{\Delta_{\infty}(t)}{2}} - e^{-\frac{\Delta_{\infty}(t)}{2}} \big) \\
	&= 2 e^{\tilde \theta^{i^{\star}}_{k^{\star}}(t) + \frac{\Delta_{\infty}(t)}{2}} \sinh \bigg( \frac{\Delta_{\infty}(t)}{2}   \bigg).
	\end{align*}
	Recalling the identity $\sinh (2x) = 2 \sinh(x) \cosh(x)$, we have that
	\begin{equation*}
	\frac{T_{i^{\star} k^{\star}}(t)}{\cosh^2 \Big( \frac{\Delta_{\infty}(t)}{4}  \Big)}
	=  4 e^{\tilde \theta^{i^{\star}}_{k^{\star}}(t) + \frac{\Delta_{\infty}(t)}{2}} \tanh \bigg( \frac{\Delta_{\infty}(t)}{4}  \bigg)
	= 4 \frac{\tilde \pi_t^{i^{\star}}}{\tilde \pi_t^{k^{\star}}} e^{\frac{\Delta_{\infty}(t)}{2}} \tanh \bigg( \frac{\Delta_{\infty}(t)}{4}  \bigg).
	\end{equation*}
	Similarly,
	\begin{equation*}
	-\frac{T_{k^{\star} i^{\star}}(t)}{\cosh^2 \Big( \frac{\Delta_{\infty}(t)}{4}  \Big)} =  4 e^{-\tilde \theta^{i^{\star}}_{k^{\star}}(t) - \frac{\Delta_{\infty}(t)}{2}} \tanh \bigg( \frac{\Delta_{\infty}(t)}{4}  \bigg)
	= 4 \frac{\tilde \pi_t^{k^{\star}}}{\tilde \pi_t^{i^{\star}}} e^{ - \frac{\Delta_{\infty}(t)}{2}} \tanh \bigg( \frac{\Delta_{\infty}(t)}{4}  \bigg).
	\end{equation*}
	Now, for $j \neq i^{\star}, k^{\star}$, consider $q_{j k^{\star}} T_{j k^{\star}}(t) -  q_{j i^{\star}} T_{ji^{\star}}(t)$. Note that
	\begin{equation*}
	\Delta_{j k^{\star}}(t)
	= \theta^j_{k^{\star}} (t) - \tilde \theta^j_{k^{\star}} (t)
	= \big(\theta^{i^{\star}}_{k^{\star}} (t) - \tilde\theta^{i^{\star}}_{k^{\star}}  (t)\big) + \big(\theta^j_{i^{\star}} (t) - \tilde \theta^j_{i^{\star}} (t) \big)
	=\Delta_{\infty}(t) + \Delta_{j i^{\star}}.
	\end{equation*}
	By Lemma~\ref{lemma:diff_pi_pi_tilde_jk}, and recalling that $\log$ is increasing, we have that $\Delta_{j k^{\star}}(t) \ge 0$ and $\Delta_{j i^{\star}} \le 0$. Moreover, by definition of $\Delta_{\infty}(t)$, we have that $\Delta_{j k^{\star}}(t)\le \Delta_{\infty}(t)$ and $|\Delta_{j i^{\star}}(t)| \le \Delta_{\infty}(t)$. Then we can write
	\begin{align*}
	q_{j k^{\star}} T_{j k^{\star}}(t) -  q_{j i^{\star}} T_{ji^{\star}}(t)
	&= q_{j k^{\star}} e^{\tilde \theta^j_{k^{\star}}}(e^{\Delta_{j k^{\star}}(t)} - 1) + q_{j i^{\star}} e^{\tilde \theta^j_{i^{\star}}}(1 - e^{\Delta_{j k^{\star}}(t) - \Delta_{\infty}(t)}) \\
	&= \tilde \pi_t^j \bigg[ \frac{q_{j k^{\star}}}{\tilde \pi_t^{k^{\star}}} (e^{\Delta_{j k^{\star}}(t)} - 1) + \frac{q_{j i^{\star}}}{\tilde \pi_t^{i^{\star}}}(1 - e^{\Delta_{j k^{\star}}(t) - \Delta_{\infty}(t)}) \bigg],
	\end{align*}
	and in particular if $\frac{q_{j k^{\star}}}{\tilde \pi_t^{k^{\star}}} \ge \frac{q_{j i^{\star}}}{\tilde \pi_t^{i^{\star}}} e^{- \Delta_{\infty}(t)}$, then $q_{j k^{\star}} T_{j k^{\star}}(t) -  q_{j i^{\star}} T_{ji^{\star}}(t)$ is increasing in $\Delta_{j k^{\star}}(t)$; otherwise it is decreasing. Therefore
	\begin{align*}
	\frac{q_{j k^{\star}}}{\tilde \pi_t^{k^{\star}}} \ge \frac{q_{j i^{\star}}}{\tilde \pi_t^{i^{\star}}} e^{- \Delta_{\infty}(t)}
	&\Longrightarrow \min_{0 \le \Delta_{j k^{\star}}(t) \le \Delta_{\infty}(t)} \big\{q_{j k^{\star}} T_{j k^{\star}}(t) -  q_{j i^{\star}} T_{ji^{\star}}(t)\big\}
	= q_{j i^{\star}}  \frac{\tilde \pi_t^j}{\tilde \pi_t^{i^{\star}}}(1 - e^{ - \Delta_{\infty}(t)}), \\
	\frac{q_{j k^{\star}}}{\tilde \pi_t^{k^{\star}}} < \frac{q_{j i^{\star}}}{\tilde \pi_t^{i^{\star}}} e^{- \Delta_{\infty}(t)}
	&\Longrightarrow \min_{0 \le \Delta_{j k^{\star}}(t) \le \Delta_{\infty}(t)} \big\{q_{j k^{\star}} T_{j k^{\star}}(t) -  q_{j i^{\star}} T_{ji^{\star}}(t)\big\}
	= q_{j k^{\star}} \frac{\tilde \pi_t^j}{\tilde \pi_t^{k^{\star}}} (e^{\Delta_{\infty}(t)} - 1).
	\end{align*}
	For all $t < \infty$, and all $i,k \in \mathbf{N} \times \mathbf{N}$, let
	\begin{equation*}
	\tilde J^i_k(t, \Delta_{\infty}(t)) := \Big\{ j \in \mathbf{N} \, : \, \frac{q_{j k}}{\tilde \pi_t^{k}} \ge \frac{q_{j i}}{\tilde \pi_t^{i}} e^{- \Delta_{\infty}(t)} \, \mathrm{and} \, j \neq i, k \Big\},
	\end{equation*}
	and ${}^{\mathsf{c}}\!\tilde J^i_k (t, \Delta_{\infty}(t)) : = \mathbf{N} \setminus \big(\tilde J^i_k(t, \Delta_{\infty}(t)) \cup \{i,k\} \big)$. Putting all the above estimates together, we can bound (in the sense of Lebesgue--Stieltjes measures) the right-hand side of \eqref{eq:di_tanh_delta_infty} as
	\begin{align}
	\di \tanh \bigg( \frac{\Delta_{\infty}(t)}{4}   \bigg) 
	\le &- \bigg[ q_{i^{\star} k^{\star}} \frac{\tilde \pi_t^{i^{\star}}}{\tilde \pi_t^{k^{\star}}} e^{\frac{\Delta_{\infty}(t)}{2}}
	+ q_{ k^{\star}i^{\star}} \frac{\tilde \pi_t^{k^{\star}}}{\tilde \pi_t^{i^{\star}}} e^{ - \frac{\Delta_{\infty}(t)}{2}} \bigg] 
	\tanh \bigg( \frac{\Delta_{\infty}(t)}{4}  \bigg) \dt \nonumber \nonumber \\
	&- \bigg[ \sum_{j \in \tilde J^{i^{\star}}_{k^{\star}}(t,\Delta_{\infty}(t))} \hspace{-15pt} q_{j i^{\star}}  \frac{\tilde \pi_t^j}{\tilde \pi_t^{i^{\star}}} e^{ - \frac{\Delta_{\infty}(t)}{2}} \hspace{-3pt}
	+ \hspace{-15pt} \sum_{j \in {}^{\raisebox{0.6pt} {$\scriptscriptstyle \mathsf{c}$}}\!\tilde J^i_k (t, \Delta_{\infty}(t))} \hspace{-15pt}  q_{j k^{\star}} \frac{\tilde \pi_t^j}{\tilde \pi_t^{k^{\star}}} e^{ \frac{\Delta_{\infty}(t)}{2}} \bigg]  \tanh \bigg( \frac{\Delta_{\infty}(t)}{4} \bigg)  \dt \nonumber \\
	\le &- \tilde\lambda(t, \Delta_{\infty}(t)) \tanh \bigg( \frac{\Delta_{\infty}(t)}{4} \bigg)  \dt, \label{eq:di_tanh_bound_stability}
	\end{align}
	where we have defined
	\begin{align*}
	\tilde \lambda(t, \Delta_{\infty}(t))
	:= \min_{i \neq k} \bigg\{  q_{i k} \frac{\tilde \pi_t^i}{\tilde \pi_t^k} e^{\frac{\Delta_{\infty}(t)}{2}} \hspace{-5pt}
	+  q_{ k i} \frac{\tilde \pi_t^k}{\tilde \pi_t^i} e^{ - \frac{\Delta_{\infty}(t)}{2}}  \hspace{-5pt}
	+ \hspace{-16pt} \sum_{j \in \tilde J^i_k(t, \Delta_{\infty}(t))}\hspace{-16pt} q_{j i}  \frac{\tilde \pi_t^j}{\tilde \pi_t^i} e^{ - \frac{\Delta_{\infty}(t)}{2}} \hspace{-5pt}
	+ \hspace{-16pt} \sum_{j \in {}^{\raisebox{0.6pt} {$\scriptscriptstyle \mathsf{c}$}}\!\tilde J^i_k(t, \Delta_{\infty}(t))} \hspace{-16pt} q_{j k} \frac{\tilde \pi_t^j}{\tilde \pi_t^k} e^{ \frac{\Delta_{\infty}(t)}{2}} \bigg\}.
	\end{align*}
	
	Using the inequality $a + b \ge 2 \sqrt{ab}$ for $a,b \ge 0$, we have, $\forall (i,k) \in \mathbf{N} \times \mathbf{N}$, and $\forall t < \infty$,
	\begin{gather*}
	q_{i k} \frac{\tilde \pi_t^i}{\tilde \pi_t^k} e^{\frac{\Delta_{\infty}(t)}{2}} \hspace{-3pt}
	+  q_{ k i} \frac{\tilde \pi_t^k}{\tilde \pi_t^i} e^{ - \frac{\Delta_{\infty}(t)}{2}} \hspace{-3pt}
	+ \hspace{-15pt} \sum_{j \in \tilde J^i_k(t, \Delta_{\infty}(t))}\hspace{-15pt} q_{j i}  \frac{\tilde \pi_t^j}{\tilde \pi_t^i} e^{ - \frac{\Delta_{\infty}(t)}{2}} \hspace{-3pt}
	+ \hspace{-15pt} \sum_{j \in {}^{\raisebox{0.6pt} {$\scriptscriptstyle \mathsf{c}$}}\!\tilde J^i_k(t, \Delta_{\infty}(t))} \hspace{-15pt} q_{j k} \frac{\tilde \pi_t^j}{\tilde \pi_t^k} e^{ \frac{\Delta_{\infty}(t)}{2}} \\
	\ge 2 \Bigg[
	\frac{e^{\frac{\Delta_{\infty}(t)}{2}}}{{\tilde \pi_t^k} } \Bigg( q_{i k} \tilde \pi_t^i
	+ \hspace{-8pt} \sum_{j \in {}^{\raisebox{0.6pt} {$\scriptscriptstyle \mathsf{c}$}}\!\tilde J^i_k(t, \Delta_{\infty}(t))} \hspace{-15pt} q_{j k} \tilde \pi_t^j \Bigg) 
	\Bigg]^{\frac{1}{2}}
	\Bigg[
	\frac{e^{\frac{-\Delta_{\infty}(t)}{2}}}{{\tilde \pi_t^i} } \Bigg( q_{k i} \tilde \pi_t^k
	+ \hspace{-8pt} \sum_{j \in \tilde J^i_k(t, \Delta_{\infty}(t))} \hspace{-15pt} q_{j i} \tilde \pi_t^j \Bigg) 
	\Bigg]^{\frac{1}{2}} \\
	\ge 2 \left(		q_{ik}q_{k i}
	+ \hspace{-8pt} \sum_{j \in \tilde J^i_k(t, \Delta_{\infty}(t))} \hspace{-15pt} q_{ik} q_{j i} \frac{\tilde \pi_t^j}{\tilde \pi_t^k}
	+ \hspace{-8pt} \sum_{l \in {}^{\raisebox{0.6pt} {$\scriptscriptstyle \mathsf{c}$}}\!\tilde J^i_k(t, \Delta_{\infty}(t))} \hspace{-15pt} q_{ki} q_{l k} \frac{\tilde \pi_t^l}{\tilde \pi_t^i}
	+ \hspace{-8pt} \sum_{\subalign{j &\in \tilde J^i_k(t, \Delta_{\infty}(t)), \\ l &\in {}^{\raisebox{0.6pt} {$\scriptscriptstyle \mathsf{c}$}}\!\tilde J^j_k(t, \Delta_{\infty}(t)) }} \hspace{-10pt} q_{ji} q_{l k} \frac{\tilde \pi_t^j \tilde \pi_t^l}{\tilde \pi_t^i \tilde \pi_t^k}
	\right)^{\frac{1}{2}},
	\end{gather*}	
	which yields
	\begin{equation}\label{eq:tilde_lambda_star}
	\tilde\lambda(t, \Delta_{\infty}(t)) \ge
	2 \min_{i\neq k} \min_{S \subseteq \mathbf{N}}
	\left(		q_{ik}q_{k i}
	+  \hspace{-3pt}   \sum_{\subalign{\: j &\in S, \\ j &\neq i,k}} \hspace{-3pt}  q_{ik} q_{j i} \frac{\tilde \pi_t^j}{\tilde \pi_t^k}
	+  \hspace{-3pt} \sum_{\subalign{\: l &\notin S, \\ l &\neq i,k}} \hspace{-3pt} q_{ki} q_{l k} \frac{\tilde \pi_t^l}{\tilde \pi_t^i}
	+ \hspace{-3pt}   \sum_{\subalign{\: j &\in S, \\ j &\neq i,k}} \sum_{\subalign{\: l &\notin S, \\ l &\neq i,k}}  \hspace{-3pt} q_{ji} q_{l k} \frac{\tilde \pi_t^j \tilde \pi_t^l}{\tilde \pi_t^i \tilde \pi_t^k}
	\right)^{\frac{1}{2}}.
	\end{equation}
	Then, bounding the right-hand side of \eqref{eq:di_tanh_bound_stability} and applying a Gr\"onwall's argument (for absolutely continuous processes) yields \eqref{eq:pathwise_ineq_tanh}. The inequality \eqref{eq:pathwise_rate_hilbert} follows immediately by minimizing further the right-hand side of \eqref{eq:tilde_lambda_star}, and in particular
	\begin{equation}\label{eq:pathwise_rate_easy}
	\tilde\lambda(\tilde \pi_t, \Delta_{\infty}(t)) \ge \tilde \lambda_t := 2 \min_{i \neq k} \left\{ \sqrt{ q_{i k} q_{k i}
	+ \sum_{j \neq i, k} \min \bigg\{ \frac{ q_{j i} q_{i k}}{\tilde \pi_t^k}, \frac{q_{j k} q_{k i}}{\tilde \pi_t^i} \bigg\}  \tilde \pi_t^j } \right\} > 0,
	\end{equation}
	since by assumption $Q$ has strictly positive non-diagonal entries and $\tilde \pi_t \in \mathring{\mathcal{S}}^n$ for all $t < \infty$ by Lemma~\ref{lemma:non_divergence_to_boundary}.
	
	Now let $X_t : = \tanh (\Delta_{\infty}(t)/4)$. Then we can rewrite \eqref{eq:di_tanh_bound_stability} as
	\begin{equation*}
	\di X_t \le - \tilde \lambda^{\star} (t, X_t ) X_t \dt,
	\end{equation*}
	where
	\begin{align*}
	\tilde \lambda^{\star}(t, X_t)
	&:= \tilde \lambda(t, 4 \arctanh(X_t)) \\
	&\,= \min_{i \neq k} \bigg\{\bigg( q_{ik} \frac{\tilde \pi_t^i}{\tilde \pi_t^k}
	+ \hspace{-5pt} \sum_{j \in {}^{\raisebox{0.6pt} {$\scriptscriptstyle \mathsf{c}$}}\!\mathcal{\tilde J}^i_k(t, X_t)} \hspace{-5pt} q_{j k} \frac{\tilde \pi_t^j}{\tilde \pi_t^k}\bigg) \frac{1 + X_t}{1-X_t}
	+ \bigg( q_{ki} \frac{\tilde \pi_t^k}{\tilde \pi_t^i}
	+ \hspace{-5pt} \sum_{j \in \mathcal{\tilde J}^i_k(t, X_t)} \hspace{-5pt} q_{j k} \frac{\tilde \pi_t^j}{\tilde \pi_t^k}\bigg) \frac{1 - X_t}{1+X_t} \bigg\},
	\end{align*}
	and $\mathcal{\tilde J}^i_k(t, X_t) := \tilde J^i_k(t, 4\arctanh(X_t))$ and ${}^{\mathsf{c}}\!\mathcal{\tilde J}^i_k(t, X_t) = \mathbf{N}\setminus\big(\mathcal{\tilde J}^i_k(t, X_t\big) \cup \{i,k\})$. We want to compare $X_t$ to the solution to \eqref{eq:ode_tanh_stability}. Next, we argue that \eqref{eq:ode_tanh_stability} has a well-defined solution for all $t < \infty$. Then the proposition follows from a direct application of Lemma~\ref{lemma:comparison}.
	
	Let $u_t$ be a solution to \eqref{eq:ode_tanh_stability}. Note that $\hilbert (\mu, \nu) \in [0, \infty)$ since $\mu, \nu \in \mathring{\mathcal{S}}^n$ by assumption, so $u_0 \in [0,1)$. If $\hilbert (\mu, \nu) = 0$, then $u_0 = 0$, and $u_t = 0$ for all $t \ge 0$, so the proposition holds trivially. So from now on, assume $u_0 \in (0,1)$. Now, the coefficient $\tilde\lambda^{\star}(\tilde \pi_t, u_t)$ depends on the process $\tilde \pi_t$, which is fixed $\omega$-by-$\omega$. Observe that $\tilde\lambda^{\star}(\tilde \pi_t, u_t)$ blows up when $u_t \uparrow 1$.
	Since $u_0 < 1$, the explosion time $T$ such that $u_{T-} \hspace{-4pt}= 1$ is strictly positive. Recall that by Lemma~\eqref{lemma:non_divergence_to_boundary} $\tilde \pi_t \in \mathring{\mathcal{S}}^n$ for all $t < \infty$. Then on the interval $[0, T)$, $x \mapsto \tilde\lambda^{\star}(\tilde \pi_t, x)x$ is locally Lipschitz continuous (with Lipschitz constant dependent on $t, \omega$ and $x$) and standard results in ODE theory (see e.g.~\cite[Theorem~2.5]{teschl_odes}) give that \eqref{eq:ode_tanh_stability} has a unique solution $u_t$ in $[0,T)$. On the other hand, $\tilde\lambda^{\star}(\tilde \pi_t, u_t) \ge \tilde \lambda_t \ge 2 \min_{i \neq k}\sqrt{q_{ik}q_{ki}}$ is strictly positive for $u_t \in (-1, 1)$, so $- \tilde\lambda^{\star}(\tilde \pi_t, u_t) u_t$ is strictly negative for $u_t \in (0,1)$. Then $u_t \le u_0<1$ for all $t \ge 0$, so in fact the explosion time $T = \infty$ and \eqref{eq:ode_tanh_stability} has a unique solution for all $t\ge 0$. Moreover, $- \tilde\lambda^{\star}(\tilde \pi_t, u_t) u_t$ tends to $0$ as $u_t$ approaches 0, hence it readily follows that $u_t \in (0,u_0]$ for all $t \ge 0$.
\end{proof}

By symmetry, the bounds of Proposition~\ref{prop:pathwise_bound_stability_ode} can also be expressed in terms of the true filter $\pi_t$.
\begin{corollary}\label{cor:pathwise_bound_stability_true_filter}
	Suppose $\mu^i, \nu^i > 0$ for all $i \in \mathbf{N}$, and $\pi_t$ is observed. For all $t \ge 0$, let $u_t \in (0,1)$ be the unique solution to the ODE with random coefficients given by
	\begin{equation}\label{eq:ode_tanh_stability_true_filter}
	\frac{\di u_t}{\dt} = - \lambda^{\star} (t, u_t) u_t,
	\qquad u_0 = \tanh \bigg( \frac{\hilbert( \mu, \nu)}{4}   \bigg),
	\end{equation}
	where
	\begin{equation}\label{eq:lambda_star_true_filter}
	\lambda^{\star} (t, u_t)
	= \min_{i \neq k} \bigg\{\bigg( q_{ik} \frac{ \pi_t^i}{ \pi_t^k}
	+ \hspace{-5pt} \sum_{\substack{ j \neq i,k, \\ j \in \mathcal{J}^i_k(t, u_t)}} \hspace{-5pt} q_{j k} \frac{ \pi_t^j}{ \pi_t^k}\bigg) \frac{1 - u_t}{1+u_t}
	+ \bigg( q_{ki} \frac{ \pi_t^k}{ \pi_t^i}
	+ \hspace{-5pt} \sum_{\substack{ j \neq i,k, \\ j \in \mathcal{J}^i_k(t, u_t)}} \hspace{-5pt} q_{j i} \frac{ \pi_t^j}{ \pi_t^i}\bigg) \frac{1 + u_t}{1-u_t} \bigg\},
	\end{equation}
	and
	$
	\mathcal{J}^i_k(t, u_t) := \Big\{ j \in \mathbf{N} \, : \, \frac{q_{j k}}{ \pi_t^k} \le \frac{q_{j i}}{\pi_t^i} \Big( \frac{1+u_t}{1-u_t} \Big)^2 \Big\}
	$.
	Then for all $t < \infty$,
	\begin{equation*}
	\tanh \bigg( \frac{\hilbert (\pi_t, \tilde \pi_t)}{4} \bigg) \le u_t.
	\end{equation*}
	In particular, for all $t < \infty$,
	$
	\tanh \Big( \frac{\hilbert (\pi_t, \tilde \pi_t)}{4} \Big) \le \tanh \Big( \frac{\hilbert (\mu, \nu)}{4} \Big) e^{- \int_0^t \lambda^{\star}_s \ds},
	$
	where $\lambda^{\star}_t \ge \lambda_t$, with $\lambda^{\star}_t$ and $\lambda_t$ defined equivalently to $\tilde \lambda^{\star}_t$ and $\tilde \lambda_t$ in \eqref{eq:pathwise_rate_hilbert}, with $\pi_t$ in place of $\tilde \pi_t$.
\end{corollary}
\begin{proof}
	Similar to the proof of Proposition~\ref{prop:pathwise_bound_stability_ode}.
\end{proof}

	\begin{figure}[H]
		\centering
		\includegraphics[width=0.49\textwidth]{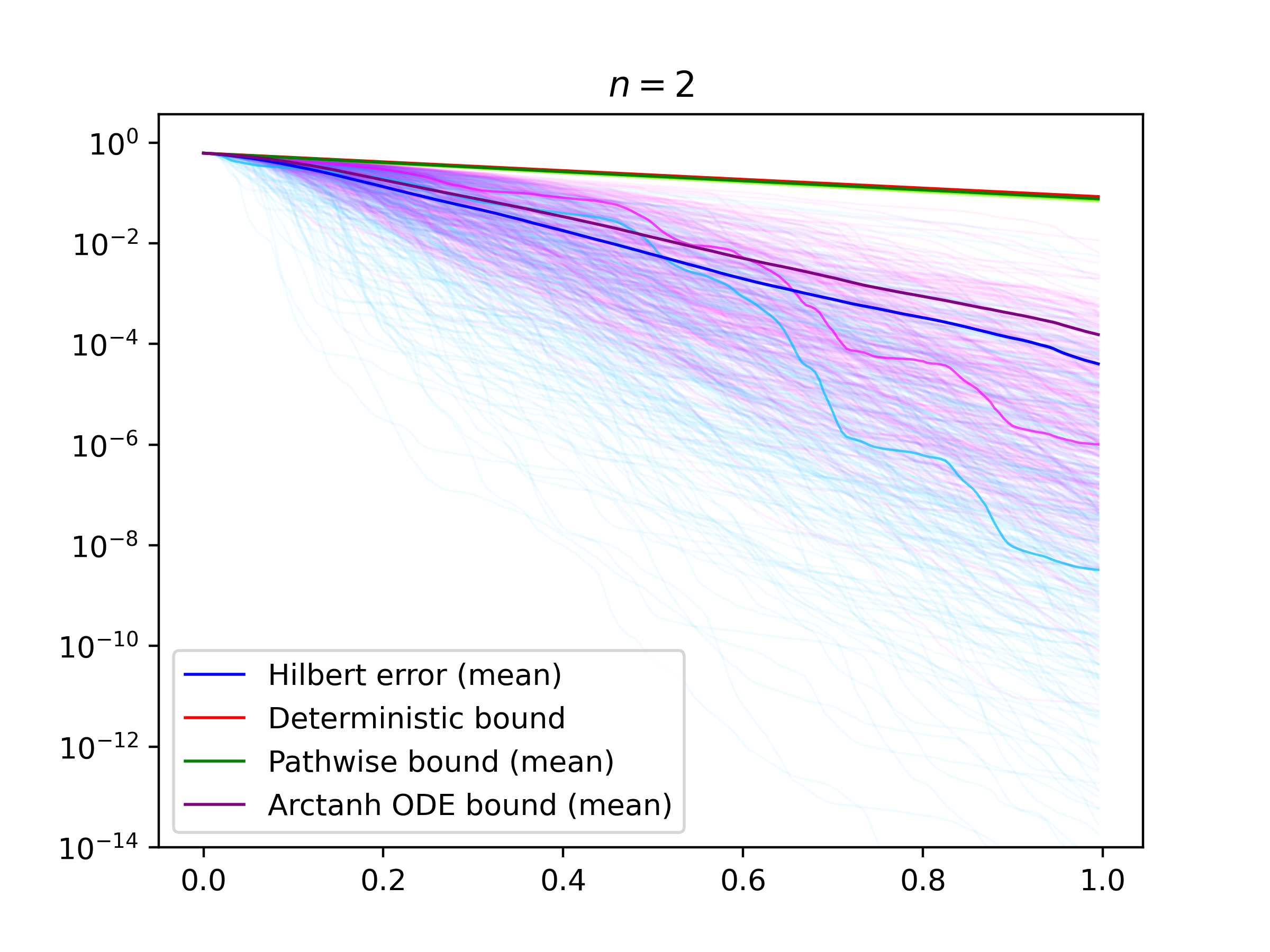}
		\includegraphics[width=0.49\textwidth]{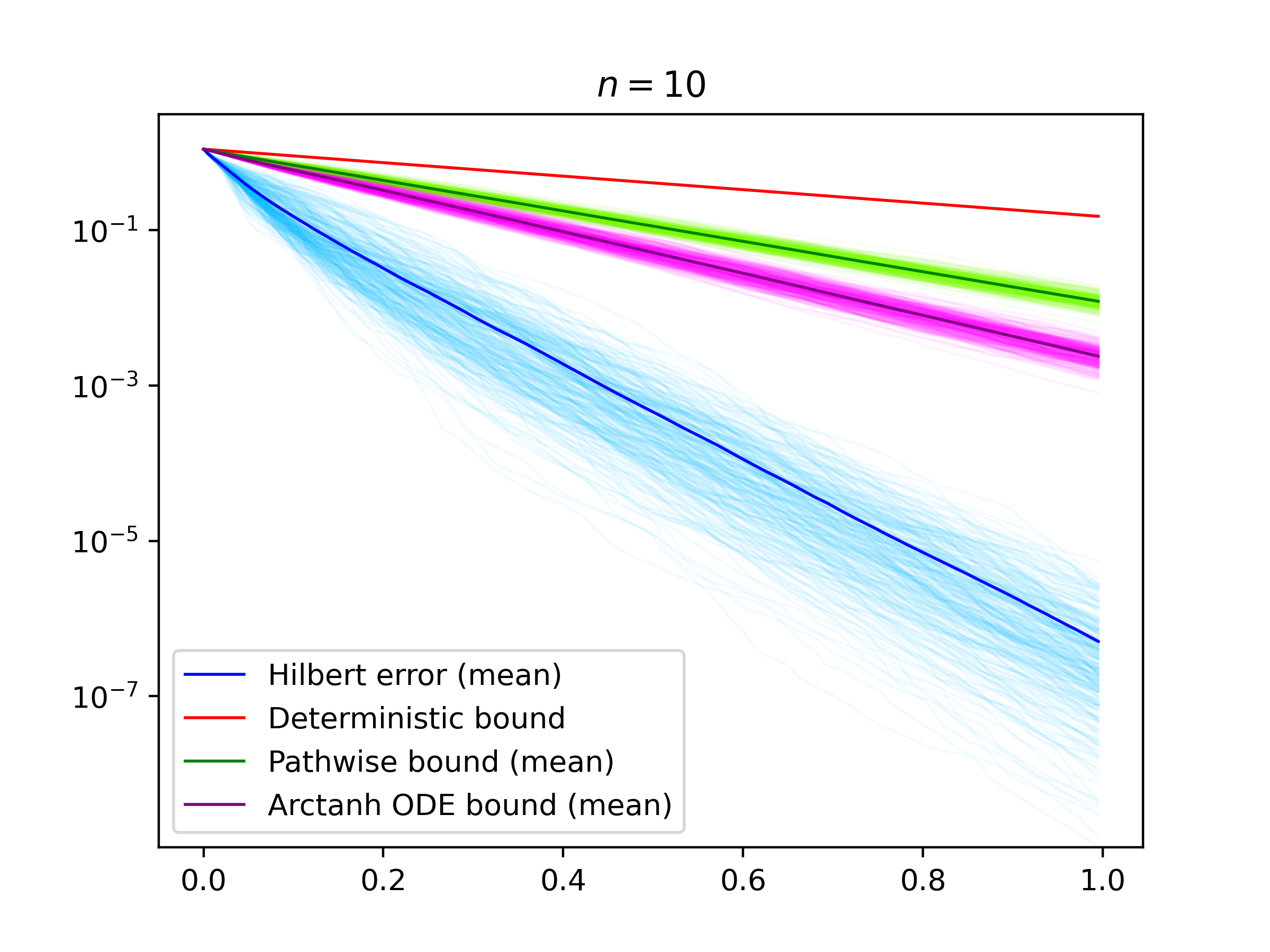}
		\includegraphics[width=0.49\textwidth]{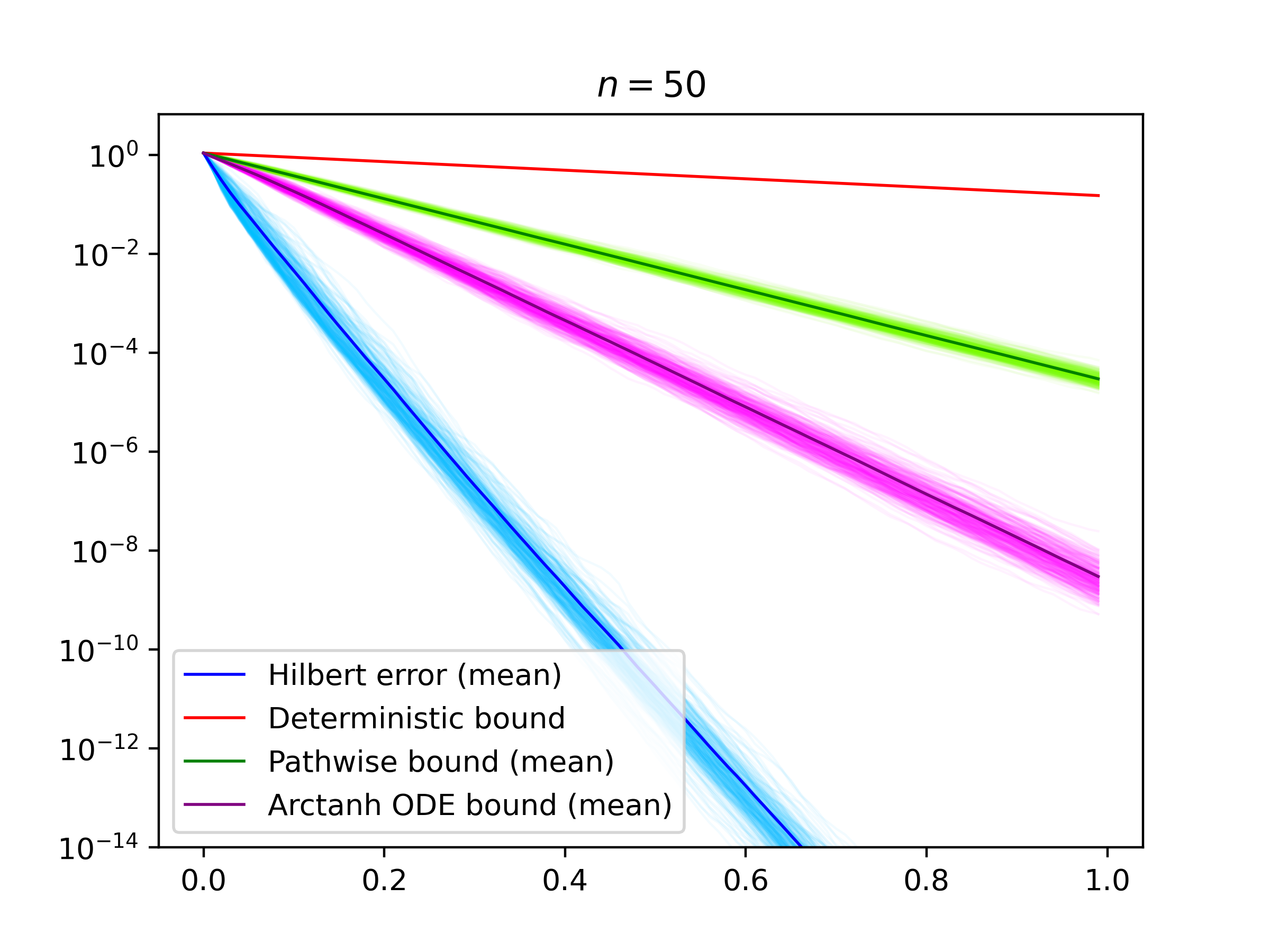}
		\includegraphics[width=0.49\textwidth]{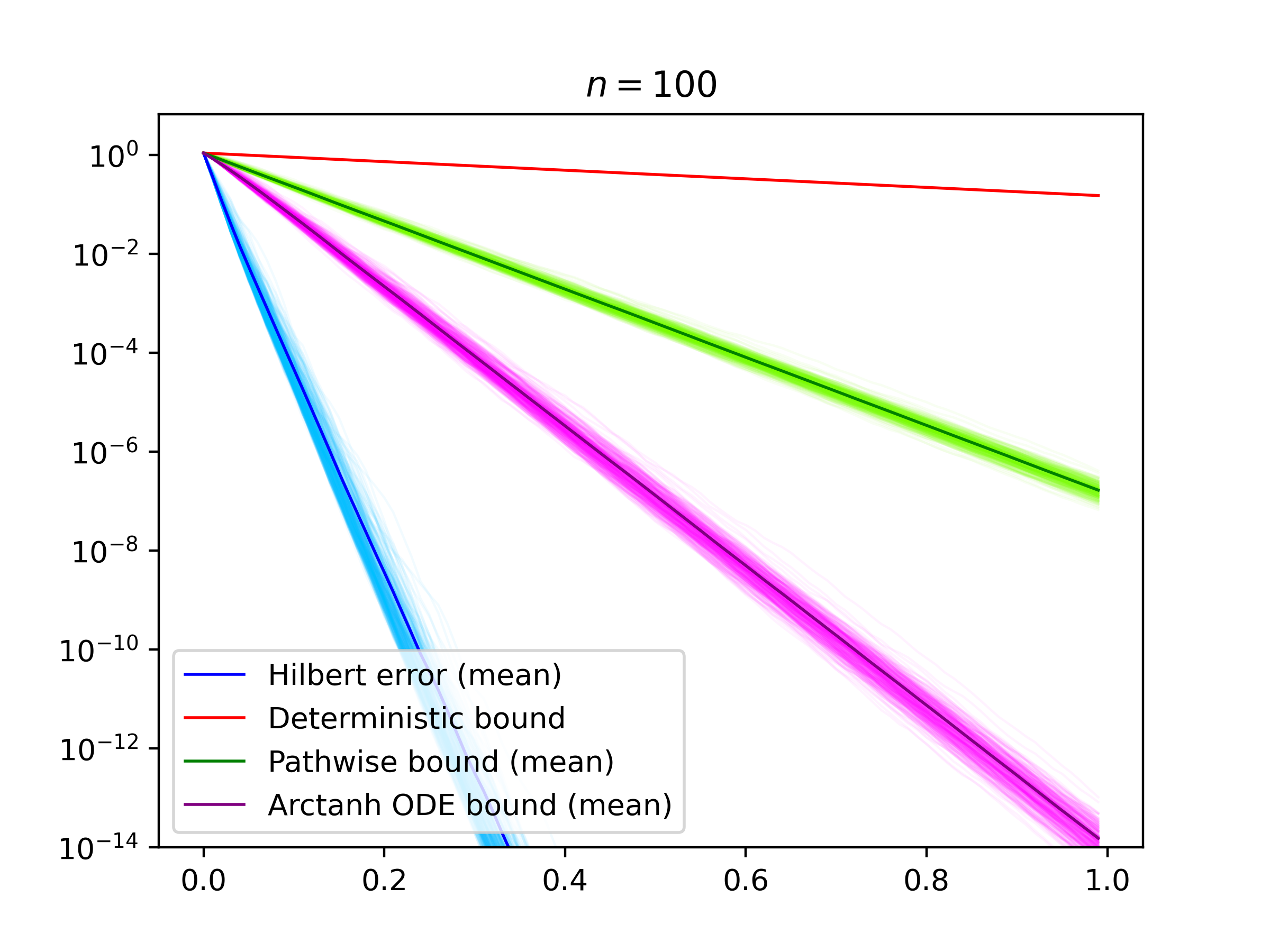}		
		\caption{For dimensions $n=2, 20, 50, 100$, and $\pi_t, \tilde \pi_t \in \mathring{\mathcal{S}}^n$, initialized at $\mu, \nu \in \mathring{\mathcal{S}}^{n}$ respectively, we plot 300 realizations of the Hilbert projective error $\hilbert (\pi_t, \tilde \pi_t)$ (light blue), of the pathwise bound $\hilbert (\mu, \nu) e^{- \int_0^t \tilde \lambda_s \ds}$ (light green), and of the ODE bound from Proposition~\ref{prop:pathwise_bound_stability_ode} (fuchsia), all in log scale, for $t \in [0,1]$. In blue, darker green and purple we have the respective sample means. In red we plot the deterministic bound $\hilbert(\mu, \nu) e^{-\lambda t}$ from Theorem~\ref{thm:contraction}. In the case of $n=2$, we highlight one realization of the Hilbert error (selected randomly), and its corresponding ODE bound. For this simulation, we keep the structure of the rate matrix fixed across all dimensions, so that the deterministic rate $\lambda$ is also fixed, and it is equal to 2 throughout. We see that as the dimension increases, the deterministic contraction rate becomes less and less optimal, and that the improvement gained by computing the pathwise rate or the ODE bound instead is significant (although the bound still remains far from sharp). For further details about this simulation, see Appendix~\ref{app:numerics}.}
		\label{fig:hilber_error_2d_100d}
	\end{figure}

	In \eqref{eq:pathwise_rate_hilbert} we state a lower bound $\tilde \lambda_t$ for $\tilde \lambda^{\star}_t$ because, numerically, finding $\tilde \lambda^{\star}_t$ by minimizing over all possible subsets of $\mathbf{N}$ for each $t$ can be costly, especially in high dimensions. On the other hand, $\tilde \lambda_t$ is easy to compute. We illustrate the performance of the bounds from Proposition~\ref{prop:pathwise_bound_stability_ode} in Figure~\ref{fig:hilber_error_2d_100d}: we plot both the pathwise bound $\hilbert (\mu, \nu) e^{- 2 \int_0^t \tilde \lambda_s \ds}$, which follows directly from \eqref{eq:pathwise_ineq_tanh}, using the rate $\tilde \lambda_t$, and the ODE bound given by $4 \arctanh(u_t)$, where $u_t$ is the numerical solution to \eqref{eq:ode_tanh_stability}, and compare them with the deterministic rate from Theorem~\ref{thm:contraction}.

	As we can see from the plots in Figure~\ref{fig:hilber_error_2d_100d}, even when using the pathwise contraction rate or the ODE bound from Proposition~\ref{prop:pathwise_bound_stability_ode}, our bounds are not tight in dimension $n \ge 2$. This affects our simulations for the error bounds in Section~\ref{sec:numerics_approx} as well. Since further algebraic manipulations in the spirit of what we have attempted so far do not seem likely to yield a better bound, one could think of improving our estimates by looking instead at the rate of decay of the expected Hilbert error, which from our simulations seems very well behaved, or even at the expected contraction rate. To proceed in either of these directions, one would need to find a way to estimate the expectation of the argmax and argmin of the ratios between the components of $\pi_t$ and $\tilde \pi_t$. 
 
 Our numerical experiments also suggest, at least for the examples we consider, that there is a concentration of measure phenomenon occurring in high-dimensional examples, where a much faster convergence rate than we have established will hold with overwhelming probability. We leave the study of this problem open for future research.

	\section{Robustness and error bounds}\label{sec:robustness}
	
	The contraction results of Section~\ref{sec:contraction} allow us to investigate the behaviour of the error when approximate filters, rather than the optimal filter, are employed.
	
	\subsection{Continuity of the Wonham filter with respect to the model parameters}
	
	In this section, we recover a version of Chigansky and Van Handel's results on robustness of the Wonham filter with respect to the model parameters (see \cite{chigansky07}). Note that our approach is entirely different from \cite{chigansky07}, and we obtain robustness in terms of the Hilbert error instead of the $\ell^1$-norm. Since the Hilbert metric is stronger than $\ell^1$ (see Lemma 1 in \cite{atar-zeitouni97}), the error estimates we obtain here are tighter then those in \cite{chigansky07}.
	
	Consider an approximate Wonham filter with incorrect model parameters
	\begin{equation}\label{eq:wonham_wrong_params}
	\di \tilde\pi_{t} = \tilde{Q}^{T} \tilde\pi_{t} \dt + \left( \tilde{H} - \tilde\pi_{t}^{\top}\tilde{h} \, \mathbb{I}_{n+1} \right) \tilde\pi_{t} \left( \di Y_{t} - \tilde\pi_{t}^{\top}\tilde{h} \dt \right), \quad \tilde\pi_0 = \nu,
	\end{equation}
	where $\tilde Q = (\tilde q_{ij})$ and $\tilde h$ are respectively a transition intensities matrix and a bounded sensor function different from $Q$ and $h$, and $\tilde H = \operatorname{diag}(\tilde h)$ the diagonal matrix with entries $(\tilde H)_{ii} = \tilde h^i$. We are interested in the Hilbert error $\hilbert (\pi_t, \tilde \pi_t)$.
	
	\begin{notation}
		Compared to previous sections, $(\tilde \pi_t)_{t \ge0}$ now denotes the solution to \eqref{eq:wonham_wrong_params}, while $(\pi_t)_{t \ge0}$ is still the solution to \eqref{eq:wonhamNorm}.
	\end{notation}
	
	\begin{rmk}
		Once more, we recall that all our results can be extended to the case of multidimensional observations ($d \neq 1$) and invertible $\sigma$.
		In particular, note that the arguments we present here allow easily for $\sigma \neq 1$ in \eqref{eq:observation} and \eqref{eq:wonhamNorm} (which would correspond to invertible $\sigma \neq \mathbb{I}_d$ in higher dimensions). This would add another `misspecified' parameter $\tilde \sigma$ to \eqref{eq:wonham_wrong_params}. In the proofs we present below, $\sigma$ and $\tilde \sigma$ can be directly incorporated into respectively $h$ and $\tilde h$ in the equations \eqref{eq:wonhamNorm} and \eqref{eq:wonham_wrong_params} for the `right' and `wrong' Wonham filter. This case cannot be treated in \cite{chigansky07} since the arguments therein require knowledge of the quadratic variation of the observations $Y$ (see also \cite[Remark~4.2]{chigansky07}).
	\end{rmk}
	
	The rest of this section is devoted to proving the following theorem.
	
	\begin{thm}[Model robustness]\label{thm:robustness_misspecified_model_params}
		Let $\pi_t$ be the solution to \eqref{eq:wonhamNorm} and $\tilde \pi_t$ the solution to \eqref{eq:wonham_wrong_params}. Assume $\mu^i, \nu^i > 0$ for all $i \in \mathbf{N}$ and $q_{ij}, \tilde q_{ij} > 0$ for all $i \neq j$. For all $t < \infty$,
		\begin{align*}
		\expect{\hilbert(\pi_t, \tilde \pi_t)}
		&\le \hilbert(\mu,\nu) e^{-\lambda t}
		+ K_q \int_0^t e^{-\lambda(t-s)} \expect{\frac{1}{\min_k \tilde \pi^k_s}} \ds + K_h \int_0^t e^{-\lambda(t-s)} \ds \nonumber \\
		&\quad + \frac{1}{4}\sum_{(i,k)} \sum_{(j,l) \neq (i,k)} \expect{\int_0^t e^{-\lambda(t-s)} \di L^{0}_s(\Delta_{ik}(\cdot) - \Delta_{jl}(\cdot))},
		\end{align*}
		where $\lambda = 2 \min_{i \neq j} \sqrt{q_{ij}q_{ji}}$, $K_q = 2 \max_{j,k} |\tilde q_{jk} - q_{jk}| $ and $K_h = 2 \max_j |h^j| \max_i |h^i - \tilde h^i| + \max_i |(h^i)^2 - (\tilde h^i)^2|$, and $L^0_t(\Delta_{ik}(\cdot) - \Delta_{jl}(\cdot))$ denotes the local time at 0 of the process $(\Delta_{ik}(t) - \Delta_{jl}(t))_{t \ge 0}$, where $\Delta_{ik}$ evolves according to \eqref{eq:delta_diff_misspecified} below for all $(i,k) \in \mathbf{N} \times \mathbf{N}$.
		
		Moreover, for all $t < \infty$, we have that the local time terms disappear as $\tilde h \rightarrow h$ for $\tilde h$ in a compact set around $h$, in the sense that there exists a constant $\tilde C < \infty$ such that
		\begin{equation*}
		\lim_{\tilde h\rightarrow  h} \expect{\hilbert(\pi_t, \tilde \pi_t)} \le \hilbert(\mu,\nu) e^{-\lambda t}
		+ K_q \tilde C (1 - e^{-  \lambda}).
		\end{equation*}
	\end{thm}
	
	Even in the case where $\tilde h$ remains fixed, this result gives us good control over $\hilbert (\pi_t, \tilde \pi_t)$, as shown by the next Proposition.
	
	\begin{prop}\label{prop:finte_local_time}
		The error terms in Theorem~\ref{thm:robustness_misspecified_model_params} stay finite as $t \rightarrow \infty$. Specifically,
		\begin{equation*}
		\sup_{t \ge 0} \expect{ \int_0^t e^{- \lambda(t-s)} \di L^0_s(\Delta_{ik}(\cdot) - \Delta_{jl}(\cdot))} < \infty, \quad \forall (i,k), (j,l) \in \mathbf{N} \times \mathbf{N}, (i,k) \neq (j,l).
		\end{equation*}
	\end{prop}
	
	\begin{rmk}
		We have stated Theorem~\ref{thm:robustness_misspecified_model_params} above with the decay rate $\lambda$ a function of $Q$, the `true' dynamics of the Markov chain, and the (first) error term a function of the `misspecified' process $\tilde \pi_t$. However, nothing in our proof prevents us from doing the opposite, if we so wish:~the theorem still holds if we replace $\lambda$ with $\tilde \lambda = 2 \min_{i \neq j} \sqrt{\tilde q_{ij} \tilde q_{ji}}$, and $\tilde \pi_t$ with $\pi_t$. Note that the error term due to the misspecification of $h$ (and the local time terms) stay the same.
	\end{rmk}

	We now set out to prove these results. As in section \ref{sec:contraction}, we start by transforming $\pi_t, \tilde \pi_t$ into $\theta_t, \tilde \theta_t$ and derive the dynamics of $\Delta_{ik}(t) = \theta^i_k (t) - \tilde \theta^i_k (t) = \log \frac{\pi^i}{\pi^k}(t) - \log \frac{\tilde \pi^i}{\tilde \pi^k}(t)$ for $(i,k) \in \mathbf{N} \times \mathbf{N}$ (recalling that $\Delta_{ii}= 0$). Note that $\theta^i_k$, $\tilde \theta^i_k$ and $\hilbert (\pi_t, \tilde \pi_t)$ are all a.s.~well-defined for finite $t \ge 0$, (since Lemma \ref{lemma:non_divergence_to_boundary} holds equivalently when the dynamics of the filter $\tilde \pi$ are given by parameters $\tilde Q$ and $\tilde h$).
	
	Applying It\^o's formula to \eqref{eq:wonhamNorm} and \eqref{eq:wonham_wrong_params}, we derive the dynamics of the difference process $\Delta_{ik}(t)$ as
	\begin{align}
	\di \Delta_{ik}(t)
	&= -\sum_{\substack{j = 0 \\ j\neq k}}^{n} \Bigg( q_{jk} \frac{\pi^j_t}{\pi^k_t} - \tilde q_{jk} \frac{\tilde\pi^j_t}{\tilde\pi^k_t} \Bigg) \dt
	+ \sum_{\substack{j = 0 \\ j\neq i}}^{n} \Bigg( q_{ji} \frac{\pi^j_t}{\pi^i_t} - \tilde q_{ji} \frac{\tilde\pi^j_t}{\tilde\pi^i_t} \Bigg)\dt
	+ (q_{ii} - \tilde q_{ii} - q_{kk} + \tilde q_{kk}) \dt  \nonumber \\
	&\quad + (h^i - \tilde h^i - h^k + \tilde h^k) (\di B_t + \pi_t^{\top} h \dt)
	+ \frac{1}{2} \Big( (h^k)^2 - (\tilde h^k)^2 - (h^i)^2 + (\tilde h^i)^2 \Big) \dt, \nonumber \\
	\Delta_{ik}(0) &= \log \frac{\mu^i}{\mu^j} - \log \frac{\nu^i}{\nu^j}, \label{eq:delta_diff_misspecified} 
	\end{align}
	where we have again introduced the innovation process $B_t = Y_t - \int_0^t \pi_s^{\top} h \ds$.
	
	We note straight away that if $\tilde h = h$, the stochastic term disappears, as well as all the drift terms involving $h$ and $\tilde h$. So, in this simple case, we recover once again $C^1$ dynamics for $\Delta_{ik}(t)$ and arguments similar to the ones in Section \ref{sec:contraction} apply, to yield the following estimate for all $t < \infty$
	\begin{align}
	\hilbert (\pi_t, \tilde \pi_t) 
	&\le \hilbert (\mu, \nu) e^{- \int_0^t \tilde \lambda_s \ds} + \int_0^t e^{- \int_s^t \tilde \lambda_r \di r}  \max_{i,k} \bigg\{ \frac{(\delta Q^\top \tilde \pi_s)^i}{\tilde \pi^i_s} - \frac{(\delta Q^\top \tilde \pi_s)^k}{\tilde \pi^k_s} \bigg\} \ds \nonumber \\
	&\le e^{-\lambda t} \hilbert (\mu, \nu) + 2 \max_{i,k} |\tilde q_{ik} -  q_{ik}| \int_0^t e^{- \lambda (t-s)} \frac{1}{\min_{j} \tilde \pi^j(s)} \ds \nonumber \\
	&\le e^{- \lambda t} \hilbert (\mu, \nu) + \frac{2}{\lambda} \max_{i,k} |\tilde q_{ik} -  q_{ik}| \,  \frac{1}{\min_{j \in \mathbf{N}, \, s \in [0, t]} \tilde \pi^j(s)} (1 - e^{- \lambda t}), \label{eq:error_bound_Q_tilde}
	\end{align} 
	where $\delta Q : = \tilde Q - Q$ and $\min_{j \in \mathbf{N}, \, s \in [0, t]} \tilde \pi^j(s) \neq 0$ almost surely by Lemma \ref{lemma:non_divergence_to_boundary}, and $\tilde \lambda_t$ and $\lambda$ are respectively the pathwise contraction rate from \eqref{eq:pathwise_rate_hilbert} in Proposition~\ref{prop:pathwise_bound_stability_ode} and the deterministic rate of Theorem~\ref{thm:contraction}. Tighter bounds for this error, in the spirit of Proposition~\ref{prop:pathwise_bound_stability_ode}, are possible, but we will state them later in Section~\ref{sec:approximate_filter}, when we treat the error of general approximate filters (of which a filter with misspecified model parameters is a specific example).
	
	In the case $\tilde h \neq h$, the strategy of proof developed in Section \ref{sec:contraction} cannot be applied directly. It is not unlikely that, by carefully modifying the arguments to account for the stochastic terms, for example by iterated application of Tanaka's formula, one could derive dynamics for $\Delta_{\infty} (t) = \max_{(i,k) \in \mathbf{N} \times \mathbf{N}} \Delta_{ik}(t)$ similar to \eqref{eq:evol_delta_infty}. Here, however, we present a different strategy. We start by introducing a few definitions.
	
	Recall the following smooth approximations of the maximum and the argmax.
	\begin{defn}\label{def:smooth_max}
		Let $\alpha \in (0, \infty)$ and let $\mathbf{X}_t = \{X^0_t, \dots, X^n_t \}$ be a family of real-valued, continuous random variables. Define the \textit{LogSumExp} function $LSE_{\alpha}(\mathbf{X}_{\cdot})( t)$ as
		\begin{equation}\label{eq:logsumexp}
		LSE_{\alpha}(\mathbf{X}_{\cdot})(t) = \frac{1}{\alpha} \log \, \sum_i e^{\alpha X^i_t}.
		\end{equation}
		\\
		For $\mathbf{g}_t = \{g^i_t(x)\}_{i=0}^n$ a family of real-valued functions, define the \textit{SoftArgMax} (also known as \textit{SoftMax}) function $S^{arg}_{\alpha}(\mathbf{X}_{\cdot}, \mathbf{g}_{\cdot})(t)$ as
		\begin{equation}\label{eq:soft_argmax}
		S^{arg}_{\alpha}(\mathbf{X}_{\cdot}, \mathbf{g}_{\cdot})(t) = \frac{ \sum_{i} g^i_t(X^i_t) e^{\alpha X^i_t}}{\sum_k e^{\alpha X^k_t} }.
		\end{equation}
	\end{defn}
	Note that, for each $\omega$, we have pointwise convergence of $LSE_{\alpha}(\mathbf{X}_{\cdot})$ to $\max_i X^i$ and of $S^{arg}_{\alpha}(\mathbf{X}_{\cdot}, \mathbf{g}_{\cdot}(\mathbf{X}_{\cdot}))(t)$ to $\frac{1}{|\mathcal{I}|}\sum_{j \in \mathcal{I}} g^j_t(X^j_t)$ where $ \mathcal{I} = \argmax_i X^i$
	as $\alpha \rightarrow  \infty$ for every $t$ (see Appendix~\ref{app:smooth_max}).
	
	We will also need the definition of the \textit{local time} $L_a^t$ of a continuous semimartingale in the arguments that follow. In particular, we will only be concerned with local times of continuous semimartingales whose finite variation part is absolutely continuous (e.g. It\^o processes). In this case we can take the following definition for the local time of a semimartingale $Z$ with absolutely continuous finite variation part (adapting from Revuz and Yor \cite[Chapter~6,~Corollary~1.9]{revuz_yor} and noting that if the finite variation part of $Z$ is absolutely continuous, then the proof of \cite[Chapter~6,~Theorem~1.7]{revuz_yor} yields that $L_t^a$ has a a.s. bicontinuous modification in $a$ and $t$).
	\begin{defn}\label{def:local_time}
		Let $Z_t = M_t + A_t$ be a real valued continuous semimartingale, where $M$ is a local martingale and $A$ is an absolutely continuous finite variation process. We take the \textit{local time} of $Z$ at $a \in \R$, at time $t$, to be the process $L^t_a$ continuous in $t \in \R^+$ and $a$ given by
		\begin{equation*}
		L_t^a = \lim_{\varepsilon \rightarrow 0} \frac{1}{2 \varepsilon} \int_0^t \mathds{1}_{(a - \varepsilon, a + \varepsilon)}(Z_s) \di \langle Z \rangle_s \quad a.s.
		\end{equation*}
	\end{defn}

	\subsubsection{Proof of Theorem \ref{thm:robustness_misspecified_model_params}}
	
	Our strategy for the proof of Theorem~\ref{thm:robustness_misspecified_model_params} is as follows. We use \eqref{eq:logsumexp} to define a smooth approximation of the maximal process $\Delta_{\infty}$, and we derive its dynamics through It\^o's formula. Then a bit of care is required when taking the limit as $\alpha \rightarrow \infty$, as some of the integrands converge to Dirac masses when the maximal process is attained at multiple indices at the same time. In Appendix~\ref{app:smooth_max} we show how to deal with these terms, which determine the emergence of local times in our estimates.
	
	\begin{proof}[Proof of Theorem~\ref{thm:robustness_misspecified_model_params}]
		Consider the family of processes $\boldsymbol{\Delta}_t = \{\Delta_{ik}(t)\}_{i \in \mathbf{N}, k \in \mathbf{N}}$ with evolution equations given by \eqref{eq:delta_diff_misspecified}. We apply It\^o's Lemma to derive the dynamics of $LSE_{\alpha}(\boldsymbol{\Delta}_{\cdot})(t)$
		\begin{align*}
		\di LSE_{\alpha}(\boldsymbol{\Delta}_{\cdot})(t)
		&= \sum_{(j,l)} \frac{e^{\alpha \Delta_{jl}(t)} }{\sum_{(i,k)} e^{\alpha \Delta_{ik}(t)}} \di \Delta_{jl}(t) \\
		&\quad + \frac{1}{2} \sum_{(j,l)} \sum_{(u,v)} \alpha e^{\alpha\Delta_{jl}(t)} \bigg( \frac{\mathds{1}_{\{(u,v) = (j,l)\}}}{\sum_{(i,k)} e^{\alpha \Delta_{ik}(t)}}
		- \frac{e^{\alpha \Delta_{uv}(t)}}{(\sum_{(i,k)} e^{\alpha \Delta_{ik}(t)})^2}   \bigg) \di \langle \Delta_{jl}(\cdot), \Delta_{uv}(\cdot)  \rangle_t,
		\end{align*}
		where all the summations happen over the set of double indices $\mathbf{N} \times \mathbf{N}$.
		
		Recalling our notation $T_{ik}(t) :=  \frac{\pi_t^i}{\pi_t^k} - \frac{\tilde\pi_t^i}{\tilde\pi_t^k}$ from Section \ref{sec:contraction}, we add and subtract terms appropriately to the dynamics of $\Delta_{ik}$ to yield
		\begin{align*}
		\di \Delta_{ik}(t)
		&= \bigg( -\sum_{\substack{j = 0 \\ j\neq k}}^{n} q_{jk} T_{jk}(t) 
		+ \sum_{\substack{j = 0 \\ j\neq i}}^{n} q_{ji} T_{ji}(t) \bigg) \dt
		+ \bigg( \frac{(\delta Q^\top \tilde \pi_t)^k}{\tilde \pi^k_t} - \frac{(\delta Q^\top \tilde \pi_t)^i}{\tilde \pi^i_t} \bigg) \dt
		\nonumber \\
		&\quad + (h^i - \tilde h^i - h^k + \tilde h^k) \di B_t + \big(h^i - \tilde h^i - h^k + \tilde h^k \big) \pi_t^{\top} h \dt \\
		&\quad + \frac{1}{2} \Big((h^k)^2 - (\tilde h^k)^2 -(h^i)^2 + (\tilde h^i)^2 \Big) \dt \nonumber \\
		&= : C^{q, 1}_{ik}(t) \dt + C^{q, 2}_{ik}(t) \dt + C^{h,1}_{ik} \di B_t + C^{h,1}_{ik} \pi_t^{\top} h \dt + \frac{1}{2} C^{h,2}_{ik} \dt,
		\end{align*}
		where again $\delta Q = \tilde Q - Q$.
		
		Then we have, for all $s \le t < \infty$
		\begin{align*}
		LSE_{\alpha}(\boldsymbol{\Delta}_{\cdot})(t)
		&= LSE_{\alpha}(\boldsymbol{\Delta}_{\cdot})(s) \\
		&\quad+ \int_s^t \sum_{(j,l)} \frac{e^{\alpha \Delta_{jl}(r)} }{\sum_{(i,k)} e^{\alpha \Delta_{ik}(r)}} \bigg( C^{q, 1}_{jl}(r) + C^{q, 2}_{jl}(r) + C^{h,1}_{jl} \, \pi_r^{\top} h + \frac{1}{2} C^{h,2}_{jl} \bigg) \di r \\
		&\quad + \int_s^t \sum_{(j,l)} \frac{e^{\alpha \Delta_{jl}(r)} }{\sum_{(i,k)} e^{\alpha \Delta_{ik}(r)}} \, C^{h,1}_{jl} \di B_r \\
		&\quad + \frac{1}{2} \int_s^t  \sum_{(j,l)} \sum_{(u,v) \neq (j,l)} \frac{\alpha e^{ \alpha (\Delta_{jl}(r)+\Delta_{uv}(r))}}{(\sum_{(i,k)} e^{\alpha \Delta_{ik}(r)})^2}  \, \bigg( (C^{h,1}_{jl})^2 - C^{h,1}_{jl} C^{h,1}_{uv} \bigg) \di r \\
		&=: LSE_{\alpha}(\boldsymbol{\Delta}_{\cdot})(s) + I_1 + I_2 + \frac{1}{2}I_3.
		\end{align*}
		We want to take the limit, on both sides, as $\alpha \rightarrow \infty$.
		
		Denote
		\begin{equation*}
		\Delta_{\infty}(t) = \max_{(j,l) \in \mathbf{N} \times \mathbf{N}} \Delta_{jl}(t),
		\end{equation*}
		and we immediately have that $LSE_{\alpha}(\boldsymbol{\Delta}_{\cdot})(t)$ converges to $\Delta_{\infty}(t)$ as $\alpha \rightarrow \infty$.
		
		We let $\lambda = 2 \min_{i \neq k} \sqrt{q_{ik} \, q_{ki}}$, and for $j \in \mathbf{N}$ define the error terms
		\begin{equation*}
		\mathcal{E}_t^{1, j} = \frac{(\delta Q^\top \tilde \pi_s)^j}{\tilde \pi^j_s}, \quad
		\mathcal{E}^{2, j} = (\tilde h^j)^2 - ( h^j)^2, \quad
		\mathcal{E}^{3, j} = \tilde h^j -  h^j.
		\end{equation*} For each time $t$, define by $\mathcal{I}_t \subset \mathbf{N}\times \mathbf{N}$ the argmax of $\boldsymbol{\Delta}_t$, i.e.~the set of double indices $(i,k)$ such that $\Delta_{ik}(t) = \Delta_{\infty}(t)$ for all $(i,k) \in \mathcal{I}_t$. Let $|\mathcal{I}_t|$ denote the size of $\mathcal{I}_t$. Let us consider $I_1$, $I_2$ and $I_3$ one at a time.
		
		Start with $I_1$. We recognize as integrands $S^{arg}_{\alpha}(\boldsymbol{\Delta}_{\cdot}, \mathbf{C}^{q,1}(\cdot))(r)$,  $S^{arg}_{\alpha}(\boldsymbol{\Delta}_{\cdot}, \mathbf{C}^{q,2}(\cdot))(r)$, and  $S^{arg}_{\alpha}(\boldsymbol{\Delta}_{\cdot}, \mathbf{C}^{h,1})(r) \, \pi_r^{\top} h$, as well as $S^{arg}_{\alpha}(\boldsymbol{\Delta}_{\cdot}, \mathbf{C}^{h,2})(r)$. These are bounded by $\max_{i,j \in \mathbf{N} \times \mathbf{N}} C^{q,1}_{ij}$ and $\max_{i,j \in \mathbf{N} \times \mathbf{N}} C^{q,2}_{ij}$ respectively, which are continuous in time and therefore integrable on $[0,t]$, and by $\max_k |h^k| \, \max_{i,j \in \mathbf{N} \times \mathbf{N}} C^{h,1} < \infty$ and $\max_{i,j \in \mathbf{N} \times \mathbf{N}} C^{h,2} < \infty$ which are bounded by assumptions on $h$, and therefore integrable. Then we can apply the dominated convergence theorem to bring the limit inside the integral
		and Lemma \ref{lemma:smooth_arg_max} yields, for all $s \le t$,
		\begin{align*}
		\lim_{\alpha \rightarrow \infty} I_1
		&= \int_s^t \frac{1}{|\mathcal{I}_r|}\sum_{(i,k) \in \mathcal{I}_r} \Big( C^{q, 1}_{ik}(r) + C^{q, 2}_{ik}(r) + C^{h,1}_{ik} \, \pi_r^{\top} h + \frac{1}{2} C^{h,3}_{ik} \Big) \di r \\
		&\le - \int_s^t \frac{4}{|\mathcal{I}_r|}\sum_{(i,k) \in \mathcal{I}_r} \sqrt{q_{ik} \, q_{k i}} \, \sinh \bigg(\frac{\Delta_{ik}(r)}{2} \bigg) \di r
		+ \int_s^t \frac{1}{|\mathcal{I}_r|}\sum_{(i,k) \in \mathcal{I}_r} \bigg( \frac{(\delta Q^\top \tilde \pi_r)^k}{\tilde \pi^k_r} - \frac{(\delta Q^\top \tilde \pi_r)^i}{\tilde \pi^i_r} \bigg) \di r \\
		&\quad+  \max_j |h^j| \int_s^t \frac{1}{|\mathcal{I}_r|}\sum_{(i,k) \in \mathcal{I}_r}  \big( h^i - \tilde h^i - h^k + \tilde h^k \big) \di r \\
		&\quad+ \int_s^t \frac{1}{2 |\mathcal{I}_r|}\sum_{(i,k)\in \mathcal{I}_r}  \Big((h^k)^2 - (\tilde h^k)^2 -(h^i)^2 + (\tilde h^i)^2 \Big) \di r \\
		&\le - \lambda \int_s^t \Delta_{\infty}(r) \di r
		+ \int_s^t \max_{i,k} \Big\{ \mathcal{E}_r^{1, i} - \mathcal{E}_r^{1, k}  +\frac{1}{2} \big(  \mathcal{E}^{2, k} - \mathcal{E}^{2, i} \big)     \Big\} \di r \\
		&\quad + \max_j |h^j|  \max_{i,k} \big\{  \mathcal{E}^{3, i} - \mathcal{E}^{3, k} \big\} (t-s),
		\end{align*}
		where we have bounded $C^{q, 1}_{ik}(r)$ as in the proof of Theorem~\ref{thm:contraction}, and noted that $2\sinh(x/2) \ge x$ for $x \ge 0$.
		
		Similarly, we can swap limit and integration when dealing with $I_2$ by dominated convergence for stochastic integrals, and we get
		\begin{align*}
		\lim_{\alpha \rightarrow \infty} I_2
		&= \int_s^t \frac{1}{|\mathcal{I}_r|}\sum_{(i,k) \in \mathcal{I}_r} (h^i - \tilde h^i - h^k + \tilde h^k) \di B_r.
		\end{align*}
		
		Finally, recalling \eqref{eq:delta_diff_misspecified} and noting that the processes $\{\Delta_{ik}\}$ are continuous semimartingales of the form \eqref{eq:example_semimart} considered in Appendix~\ref{app:smooth_max}, Proposition~\ref{prop:local_time_bound} applies and we have
		\begin{equation*}
		\lim_{\alpha \rightarrow \infty} I_3 \le \frac{1}{2} \sum_{(i,k)} \sum_{(j,l) \neq (i,k)} \Big( L^{0}_t(\Delta_{ik}(\cdot) - \Delta_{jl}(\cdot)) - L^{0}_s(\Delta_{ik}(\cdot) - \Delta_{jl}(\cdot)) \Big) \quad \text{a.s.,}
		\end{equation*}
		where $L^{0}_t(\Delta_{ik}(\cdot) - \Delta_{jl}(\cdot))$ denotes the local time at 0, at time $t$, of the difference process $(\Delta_{ik}(r) - \Delta_{jl}(r))_{r\ge 0}$.
		
		Putting all these estimates together, we have that, for all $s \le t$,
		\begin{align*}
		\Delta_{\infty}(t)
		&\le \Delta_{\infty}(s)- \lambda \int_s^t \Delta_{\infty}(r) \di r
		+ \int_s^t \max_{i,k} \Big\{ \mathcal{E}_r^{1, i}  - \mathcal{E}_r^{1, k}  +\frac{1}{2} \big(  \mathcal{E}^{2, k} - \mathcal{E}^{2, i} \big)     \Big\} \di r \\
		&\quad + \max_j |h^j|  \max_{i,k} \big\{  \mathcal{E}^{3, i} - \mathcal{E}^{3, k} \big\} (t-s)
		+ \int_s^t \frac{1}{|\mathcal{I}_r|}\sum_{(i,k) \in \mathcal{I}_r} (h^i - \tilde h^i - h^k + \tilde h^k) \di B_r \\
		&\quad+ \frac{1}{4}\sum_{(i,k)} \sum_{(j,l) \neq (i,k)} \Big( L^{0}_t(\Delta_{ik}(\cdot) - \Delta_{jl}(\cdot)) - L^{0}_s(\Delta_{ik}(\cdot) - \Delta_{jl}(\cdot)) \Big).
		\end{align*}
		Taking expectation with respect to the reference measure $\mathbb{P}$ on both sides, the stochastic integral vanishes since the integrand is bounded, so we have
		\begin{align*}
		\di \E \left[ \Delta_{\infty}(t) \right]
		&\le -  \lambda \expect{ \Delta_{\infty}(t)} \dt
		+ \expect{\max_{i,k} \Big\{ \mathcal{E}_t^{1, i}  - \mathcal{E}_t^{1, k}  +\frac{1}{2} \big(  \mathcal{E}^{2, k} - \mathcal{E}^{2, i} \big)     \Big\}} \dt \\
		&\quad + \max_j |h^j|  \max_{i,k} \big\{  \mathcal{E}^{3, i} - \mathcal{E}^{3, k} \big\} \dt 
		+ \frac{1}{4}\sum_{(i,k)} \sum_{(j,l) \neq (i,k)} \di \expect{ L^{0}_t(\Delta_{ik}(\cdot) - \Delta_{jl}(\cdot)) },
		\end{align*}
		where the left-hand side and the last term on the right-hand side are to be understood as Lebesgue--Stieltjes measures. Using the chain rule to find the dynamics of $e^{\lambda t} \E \left[ \Delta_{\infty}(t) \right]$, and integrating both sides of the resulting differential inequality yields that, for all $ s \le t < \infty$,
		\begin{align*}
		\E \left[ \Delta_{\infty}(t) \right]
		&\le \expect{\Delta_{\infty}(s)} e^{-\lambda (t-s)}
		+  \int_s^t e^{-\lambda(t-r)} \E \bigg[ \max_{i,k} \Big\{ \mathcal{E}_r^{1, i}  - \mathcal{E}_r^{1, k} +\frac{1}{2} \big(  \mathcal{E}^{2, k} - \mathcal{E}^{2, i} \big)     \Big\} \bigg] \di r  \nonumber \\
		&\quad + \max_j |h^j|  \max_{i,k} \big\{  \mathcal{E}^{3, i} - \mathcal{E}^{3, k} \big\} \int_s^t e^{-\lambda(t-r)}  \di r \\
		&\quad+ \frac{1}{4}\sum_{(i,k)} \sum_{(j,l) \neq (i,k)} \int_s^t e^{-\lambda(t-r)} \di \expect{L^{0}_r(\Delta_{ik}(\cdot) - \Delta_{jl}(\cdot))}.
		\end{align*}
		Bounding the error terms, since the local time is of finite variation (hence also its expectation), we can apply integration by parts for Stieltjes integrals twice and use Fubini--Tonelli on the last term of the right-hand side to obtain
		\begin{align}
		\expect{\Delta_{\infty}(t)}
		&\le \expect{\Delta_{\infty}(s)} e^{-\lambda (t-s)}
		+ 2 \max_{i,k} |\tilde q_{ik} - q_{ik}| \int_s^t e^{-\lambda(t-r)} \expect{\frac{1}{\min_j \tilde \pi^j_r}} \di r \nonumber \\
		&\quad + \Big( 2 \max_j |h^j| \max_i |h^i - \tilde h^i| + \max_i |(h^i)^2 - (\tilde h^i)^2| \Big) \int_s^t e^{-\lambda(t-r)} \di r \nonumber \\
		&\quad + \frac{1}{4}\sum_{(i,k)} \sum_{(j,l) \neq (i,k)} \E \bigg[ \int_s^t e^{-\lambda(t-r)} \di L^{0}_r(\Delta_{ik}(\cdot) - \Delta_{jl}(\cdot)) \bigg], \label{eq:expectation_delta_misspecified}
		\end{align}
		for all $s \le t < \infty$, which is what we set out to prove.
		
		We now move on to the second part of the theorem. First of all, analogously to \cite[Lemma~3.6]{chigansky07}, one can get an explicit bound on $\E [ (\min_j \tilde \pi^j_t)^{-1}  ]$ which depends continuously on the parameters $(\nu,\tilde{Q},\tilde{h})$ for $\nu \in \mathring{\mathcal{S}}^n$. In particular, we have
		\begin{equation*}
		\expect{\frac{1}{\min_j \tilde \pi^j_t}} \le \max_j \left\{ \frac{1}{\nu^j} \exp \Big\{- \tilde q_{jj} t + \max_k (\tilde h^j - \tilde h^k)^2 t \Big\} \right\} 
		\end{equation*}
		for all $t < \infty$, and the first integral on the right-hand side of \eqref{eq:expectation_delta_misspecified} is controlled as we take the limit as $\tilde h \rightarrow h$, for $\tilde h$ in a compact set around $h$. The second term clearly vanishes as $\tilde h \rightarrow h$.
		
		Next, we move to the integrals against the local times. Let $\tilde \pi_t (v)$ denote the unique solution to \eqref{eq:wonham_wrong_params} with $v \in \R^{n+1}$ in place of $\tilde h$. Note that the drift of each process $\Delta_{ik} - \Delta_{jl}$, for $(i,k), (j,l) \in \mathbf{N} \times \mathbf{N}$ and $(i,k) \neq (j,l)$, is then given by $b^{ik,jl}_t ( \tilde h)$, where 
		\begin{align}
		b^{ik,jl}_t ( v) 
		&:= -\sum_{r=0}^{n} \Bigg( q_{rk} \frac{\pi^r_t}{\pi^k_t} - \tilde q_{rk} \frac{\tilde\pi^r_t(v)}{\tilde\pi^k_t(v)} \Bigg)
		+ \sum_{r=0}^{n} \Bigg( q_{ri} \frac{\pi^r_t}{\pi^i_t} - \tilde q_{ri} \frac{\tilde\pi^r_t(v)}{\tilde\pi^i_t(v)} \Bigg) \nonumber \\
		&\quad+\sum_{r=0}^{n} \Bigg( q_{rl} \frac{\pi^r_t}{\pi^l_t} - \tilde q_{rl} \frac{\tilde\pi^r_t(v)}{\tilde\pi^l_t(v)} \Bigg) 
		- \sum_{r=0}^{n} \Bigg( q_{rj} \frac{\pi^r_t}{\pi^j_t} - \tilde q_{rj} \frac{\tilde\pi^r_t(v)}{\tilde\pi^j_t(v)} \Bigg) \nonumber \\
		&\quad+ (h^i - v^i - h^k + v^k -  h^j + v^j + h^l - v^l) \pi_t^{\top} h  \nonumber \\
		&\quad+ \frac{1}{2} \Big( (h^k)^2 - (v^k)^2 - (h^i)^2 + (v^i)^2 - (h^l)^2 + (v^l)^2 + (h^j)^2 - (v^j)^2 \Big). \label{eq:drift_delta_ik_jl}
		\end{align}
		Consider the difference of $b^{ik,jl} ( \tilde h) $ and $b^{ik,jl} ( h) $ on $[0,t]$. Using that $\tilde \pi(h)$ and $\tilde \pi(\tilde h)$ live in the simplex, we get
		\begin{align*}
		\E\biggl[ &\sup_{s \leq t} \big| b^{ik,jl}_s ( h) - b^{ik,jl}_s (\tilde h)    \big| \biggr]^2  \\
		&\le  \!\!\!\! \sum_{u \in \{i,k,j,l\}} \sum_{r \neq u} \tilde q_{ru} \E\Biggl[\sup_{s \le t}\bigg( \frac{1}{ \tilde \pi_s^u(h) \tilde \pi_s^u (\tilde h)} \bigg)^{\! 2}\Biggr] \E\Biggl[ \sup_{s \le t} \bigg( \big|  \tilde \pi_s^u (\tilde h) - \tilde \pi^u_s (h)   \big| + \big|  \tilde \pi_s^r (\tilde h) - \tilde \pi^r_s (h)  \big|  \bigg)^{\!2} \Biggr].
		\end{align*} 
		For all $u \in \mathbf{N}$, a minor extension of \cite[Lemma~3.6]{chigansky07} gives that the first expectation is controlled uniformly in $\tilde h$, for $\tilde h$ belonging to a compact set around $h$. Since $\tilde \pi_t $ lives in the simplex, the SDE \eqref{eq:wonham_wrong_params} has Lipschitz coefficients, and we can apply standard stability arguments (such as \cite[Theorem~16.4.3]{cohen15}) to see that the second expectation tends to $0$ as $\tilde h \rightarrow h$. Hence we have ucp convergence $ b^{ik,jl} (\tilde h) \rightarrow b^{ik,jl} ( h)$ on $[0,t]$ as $\tilde h \rightarrow h$.
		
		Now fix an arbitrary sequence $\{ \tilde h_n\}_{n \in \N}$ such that $\tilde h_n \rightarrow h$. By the above, we can take a subsequence $\{ \tilde h_{n_r}\}_{r \in \N}$ such that $b^{ik,jl}_s (\tilde h_{n_r})$ converges uniformly to $b^{ik,jl}_s ( h)$ on $[0,t]$ a.s. From now on, when we write $\tilde h \rightarrow h$, we mean the limit along this subsequence. Denote by $(\Delta_{ik} - \Delta_{jl})_t^{\star}$ the limit of $(\Delta_{ik} - \Delta_{jl})_t$ as $\tilde h \rightarrow h$. Using this uniform convergence, we get that, a.s., for all $s\in [0,t]$,
		\begin{align*}
		(\Delta_{ik} - \Delta_{jl})_s^{\star} &= (\Delta_{ik} - \Delta_{jl})_0
		+ \lim_{\tilde h \rightarrow h}
		\int_0^s b^{ik,jl}_r ( \tilde h) \di r \\
		&\quad + \lim_{\tilde h \rightarrow h} (h^i - \tilde h^i - h^k + \tilde h^k - h^j + \tilde h^j + h^l - \tilde h^l) B_s \\
		&= (\Delta_{ik} - \Delta_{jl})_0
		+ \int_0^s b^{ik,jl}_r ( h) \di r,
		\end{align*}
		and $(\Delta_{ik} - \Delta_{jl})_s^{\star}$ is absolutely continuous with derivative $b^{ik,jl}_s ( h)$.
		
		Now, by Tanaka's formula we have that
		\begin{equation*}
		L^0_t(\Delta_{ik} - \Delta_{jl}) = |(\Delta_{ik} - \Delta_{jl})_t| - |(\Delta_{ik} - \Delta_{jl})_0| + \int_0^t \operatorname{sign}((\Delta_{ik} - \Delta_{jl})_s) \di (\Delta_{ik} - \Delta_{jl})_s,
		\end{equation*}
		with the convention $\operatorname{sign}(0) = -1$.
		Taking the limit as $\tilde h \rightarrow h$ on both sides of the equation above, the stochastic integral vanishes, and applying dominated convergence to the integral involving $b^{ik,jl}$, we have
		\begin{equation}\label{eq:limit_local_time}
		\lim_{\tilde h \rightarrow h} L^0_t(\Delta_{ik} - \Delta_{jl}) = |(\Delta_{ik} - \Delta_{jl})_t^{\star}| - |(\Delta_{ik} - \Delta_{jl})_0| + \int_0^t \lim_{\tilde h \rightarrow h}\operatorname{sign}((\Delta_{ik} - \Delta_{jl})_s) b^{ik,jl}_s(\tilde h) \ds.
		\end{equation}
		Consider the limit inside the integral. Note that for all $s \le t$ such that $b^{ik,jl}_s( h) \neq 0$ and $(\Delta_{ik} - \Delta_{jl})_s^{\star} \neq 0$, we have a.s.
		\begin{equation*}
		\lim_{\tilde h \rightarrow h}\operatorname{sign}((\Delta_{ik} - \Delta_{jl})_s) b^{ik,jl}_s(\tilde h)
		=  \operatorname{sign}((\Delta_{ik} - \Delta_{jl})_s^{\star}) b^{ik,jl}_s( h).
		\end{equation*}
		Now consider $s \le t$ such that $b^{ik,jl}_s( h) = 0$. Then we have 
		\begin{equation*}
		\lim_{\tilde h \rightarrow h}\operatorname{sign}((\Delta_{ik} - \Delta_{jl})_s) b^{ik,jl}_s (\tilde h)
		=  0 = \operatorname{sign}((\Delta_{ik} - \Delta_{jl})_s^{\star}) b^{ik,jl}_s( h),
		\end{equation*}
		for all such $s$. Finally, consider times $s \le t$ such that $b^{ik,jl}_s( h) \neq 0$ but $(\Delta_{ik} - \Delta_{jl})_s^{\star} = 0$. Then potentially we have $\operatorname{sign}((\Delta_{ik} - \Delta_{jl})_s) b^{ik,jl}_s(\tilde h) \nrightarrow \operatorname{sign}((\Delta_{ik} - \Delta_{jl})_s^{\star}) b^{ik,jl}_s( h)$ as $\tilde h$ goes to $h$. However, by Lemma~\ref{lemma:dupuis}, the set
		\begin{equation*}
		\left\{ s \, : \, (\Delta_{ik} - \Delta_{jl})_s^{\star} = 0, \frac{\di}{\ds} (\Delta_{ik} - \Delta_{jl})_s^{\star} =  b^{ik,jl}_s( h) \neq 0    \right\}
		\end{equation*}
		has Lebesgue measure zero. So finally we can conclude that a.s.
		\begin{equation*}
		\lim_{\tilde h \rightarrow h}\operatorname{sign}((\Delta_{ik} - \Delta_{jl})_s) b^{ik,jl}_s
		=  \operatorname{sign}((\Delta_{ik} - \Delta_{jl})_s^{\star})b^{ik,jl}_s( h), \quad \text{for a.a.} \: s \le t,
		\end{equation*}
		and hence, by absolute continuity of $(\Delta_{ik} - \Delta_{jl})_t^{\star}$, the right-hand side of \eqref{eq:limit_local_time} is 0. Thus we have proven a.s.~convergence of $L^0_t(\Delta_{ik} - \Delta_{jl}) \rightarrow 0$ as $\tilde h \rightarrow h$ along the subsequence $\{ \tilde h_{n_r}\}_{r \in \N}$, which implies convergence in probability along the same subsequence. On the other hand, the original sequence $\{ \tilde h_n \}_{n \in \N}$ was arbitrary, so we can repeat the argument above along any sequence and always find a subsequence along which $L^0_t(\Delta_{ik} - \Delta_{jl})$ converges to 0 in probability. It follows that $L^0_t(\Delta_{ik} - \Delta_{jl})$ vanishes in probability as $\tilde h \rightarrow h$. By Tanaka's formula, we can also check, similarly to how the ucp convergence was deduced, that $\E [L^0_t(\Delta_{ik} - \Delta_{jl})^2]$ is bounded uniformly in $\tilde h$, for $\tilde h$ in a compact set around $h$, and thus Vitali's convergence theorem gives $\E [L^0_t(\Delta_{ik} - \Delta_{jl})] \rightarrow 0$ as $\tilde h \rightarrow h$. This yields the theorem.
	\end{proof}
	
	\begin{proof}[Proof of Proposition~\ref{prop:finte_local_time}]
		We focus on the local time terms, since by similar arguments to \cite[Proposition~3.7]{chigansky07}, we immediately have that $\sup_{t > 0} \E [(\min_k \tilde \pi_t^k)^{-1}] < \infty$. Let $(i,k),(j,l) \in \mathbf{N} \times \mathbf{N}$, with $(i,k) \neq (j,l)$. Recall that by Tanaka's formula we can write the local time at 0 of $X_t := (\Delta_{ik} - \Delta_{jl})_t$ as
		\begin{equation*}
		L^0_t(X) = |X_t| - |X_0| + \int_0^t \operatorname{sign}(X_s) \di X_s.
		\end{equation*}
		Then we have
		\begin{align}
		\E \bigg[ \int_0^t e^{-  \lambda (t-s)} \di L_s^0(X_{\cdot})     \bigg]
		&= \E \bigg[ \int_0^t e^{- \lambda (t-s)} \di |X_{s}|    \bigg]
		+ \E \bigg[ \int_0^t e^{- \lambda (t-s)} \operatorname{sign}(X_s) \di X_s    \bigg] \nonumber \\
		&= \E \bigg[ \int_0^t e^{- \lambda (t-s)} \di |X_{s}|    \bigg]
		+ \E \bigg[ \int_0^t e^{- \lambda (t-s)} \operatorname{sign}(X_s) b^{ik,jl}_s (\tilde h) \ds    \bigg] \nonumber \\
		&\le \E \bigg[ \int_0^t e^{- \lambda (t-s)} \di |X_{s}|    \bigg]
		+ \sup_{s \le t } \E \Big[ \big| b^{ik,jl}_s ( \tilde h)  \big| \Big] \int_0^t e^{- \lambda (t-s)} \ds, \label{eq:expectation_neg_exp_local_time}
		\end{align}
		where $ b^{ik,jl}_s (\tilde h)$ is the drift of $X_t$, defined in \eqref{eq:drift_delta_ik_jl}. Since
		\begin{equation*}
		|b^{ik,jl}_t| \le K_q  \frac{1}{ \min_i \pi_t^i} + K_{\tilde q}  \frac{1}{ \min_i \tilde \pi_t^i} + K_h,
		\end{equation*}
		where $K_q$, $K_{\tilde q}$ and $K_h$ are constants only depending on $Q$, $\tilde Q$, $h$ and $\tilde h$, it follows that the second term in \eqref{eq:expectation_neg_exp_local_time} is finite as we take the supremum over all $t > 0$.  As for the first term, integrating by parts twice we have that
		\begin{align*}
		\E \bigg[ \int_0^t e^{- \lambda (t-s)} \di |X_{s}|    \bigg] 
		&= \E \big[ |X_t| \big] - |X_0| e^{-  \lambda t} - \lambda \E \Big[ \int_0^t |X_s| e^{- \lambda (t-s)} \ds \Big],
		\end{align*}
		and since $X_t = \Delta_{ik} - \Delta_{jl}$, $\Delta_{ik} = \log \pi_i/ \pi_k - \log \tilde \pi_i / \tilde \pi_k$ and $|\log (x)| \le 1/x$, we can again bound the right-hand side by multiples of $\E \big[ 1/\min_i \pi_t^i  \big]$ and $\E \big[ 1/\min_i \tilde \pi_t^i  \big]$, which remain finite as we take a supremum over $t > 0$.
	\end{proof}
	
	\subsection{Error bounds for an approximate filter}\label{sec:approximate_filter}
	
	The approach we presented in the previous subsection allows for a more general result. We can proceed exactly as before to compute the error of a general approximate filter, rather than simply the filter with modified $Q$ and $h$ (and $\sigma$, if we allow for $\sigma \neq 1$). The discussion in this section will yield the proofs of Theorem~\ref{thm:expected_hilbert_bounds} and Theorem~\ref{thm:pathwise_decay_approx_error}.
	
	Consider a general approximate filtering model given by \eqref{eq:wonhamApprox}, i.e.
	\begin{equation}\label{eq:approximate_filter}
	\di \tilde \pi_t  = \tilde f_t \dt + \tilde g_t \di Y_t, \quad \tilde \pi_0 = \nu,
	\end{equation}
	where $\tilde f_t, \tilde g_t$ are $ \R^{n+1}$-valued predictable processes. We will also need the following assumption:
	\begin{assumption}\label{assmp:approx_filter}
		With probability 1, $\tilde \pi_t \in \mathring{\mathcal S}^n$ for all $t < \infty$. Moreover, $\tilde f_t$ and $\tilde g_t$ are locally bounded and satisfy the integrability condition
		\begin{equation*}
		\expect{\int_0^t \max_i\frac{|\tilde f^i_s|}{\pi^i_s} \ds  + \bigg(\int_0^t \max_i \Big( \frac{\tilde g^i_s}{\pi^i_s}\Big)^2 \ds \bigg)^{1/2}} <\infty
		\end{equation*}
		for all $t<\infty$.
	\end{assumption}

	Note that the Wonham filter SDE \eqref{eq:wonhamNorm}, or the Wonham filter with misspecified model parameters given by \eqref{eq:wonham_wrong_params}, immediately satisfy Assumption~\assumptionref{assmp:approx_filter} by Lemma~\ref{lemma:non_divergence_to_boundary} and (a simple extension of) \cite[Lemma~3.6]{chigansky07}.

	We start by proving an intermediate result.
		\begin{prop}[Dynamics of the Hilbert error of an approximate filter]\label{prop:bounds_approx_filter}
		Let $\pi_t$ be the solution to \eqref{eq:wonhamNorm} and $\tilde \pi_t$ the solution to \eqref{eq:approximate_filter}. Suppose $\mu^i, \nu^i > 0 \quad \forall i$ and $q_{ij} > 0$ for all $i \neq j$. Under Assumption~\assumptionref{assmp:approx_filter}, for all $s \le t < \infty$, we have
		\begin{align}
		\hilbert (\pi_t, \tilde \pi_t)
		&\le \hilbert (\pi_s, \tilde \pi_s)- 2 \int_s^t \hspace{-5pt} \kappa_r \sinh \Big( \frac{\hilbert (\pi_r, \tilde \pi_r)}{2}\Big) \di r
		+\hspace{-3pt} \int_s^t \hspace{-5pt}\max_{i,k} \Big\{ \mathcal{E}_r^{1, i}  - \mathcal{E}_r^{1, k}  +\frac{1}{2} \big(  \mathcal{E}_r^{2, k} - \mathcal{E}_r^{2, i} \big)     \Big\} \di r \nonumber \\
		&\quad + \max_j |h^j|  \int_s^t \max_{i,k} \big\{  \mathcal{E}_r^{3, i} - \mathcal{E}_r^{3, k} \big\} \di r
		+ \int_s^t \frac{1}{|\mathcal{I}_r|}\sum_{(i,k) \in \mathcal{I}_r} \big(\mathcal{E}_r^{3, i} - \mathcal{E}_r^{3, k} \big) \di B_r \nonumber \\
		&\quad+ \frac{1}{4}\sum_{(i,k)} \sum_{(j,l) \neq (i,k)} \int_s^t \di L^{0}_r(\Delta_{ik}(\cdot) - \Delta_{jl}(\cdot)), \label{eq:integral_ineq_error}
		\end{align}
		where $B_t = Y_t - \int_0^t \pi_s^{\top} h \ds$ is the innovation process. For $j \in \mathbf{N}$ the error terms are given by
		\begin{equation*}
		\mathcal{E}_t^{1, j}  = \Bigg( \sum_{m=0}^{n} q_{mj} \frac{\tilde \pi_t^m}{\tilde \pi_t^j} \Bigg) - \frac{\tilde f_t^j}{\tilde \pi_t^j}, \quad
		\mathcal{E}_t^{2, j}  = (h^j)^2 - \frac{(\tilde g^j_t)^2}{(\tilde \pi_t^j)^2}, \quad
		\mathcal{E}_t^{3, j}  = h^j - \frac{\tilde g^j_t}{\tilde \pi_t^j}, \quad
		\end{equation*}
		and the processes $(\Delta_{ik}(t))_{t \ge 0}$ for $(i,k) \in \mathbf{N} \times \mathbf{N}$ are defined as $\Delta_{ik}(t) = \log \frac{\pi_t^i}{\pi_t^k} - \log \frac{\tilde \pi_t^i}{\tilde \pi_t^k}$. The set $\mathcal{I}_t = \{ (i,k) \, : \, \Delta_{ik}(t) = \hilbert (\pi_t, \tilde \pi_t)\}$ is the argmax of these processes for all $ t < \infty$, and $L_t^0(\Delta_{ik}(\cdot) - \Delta_{jl}(\cdot))$ denotes the local time at $0$ of the difference process $(\Delta_{ik} - \Delta_{jl})$ for all $(i,k)$, $(j,l) \in \mathbf{N}^\times \mathbf{N}$. The decay coefficient $\kappa_t>0$ can be taken to be any of
		\begin{equation}\label{eq:possible_rates}
		\kappa_t = \left\{ \begin{array}{l}
		\lambda, \\[10pt]
		\tilde \lambda^{\star} \Big( t, \tanh \Big(\frac{\hilbert (\pi_t, \tilde \pi_t)}{4}\Big) \Big),  \\[10pt]
		\lambda^{\star} \Big( t, \tanh \Big(\frac{\hilbert (\pi_t, \tilde \pi_t)}{4}\Big) \Big),  
		\end{array} \right.
		\end{equation}
		where $\lambda$ is the deterministic rate from Theorem~\ref{thm:contraction}, and $\tilde \lambda^{\star}$ and $\lambda^{\star}$ are the functions defined in Proposition~\ref{prop:pathwise_bound_stability_ode} and Corollary~\ref{cor:pathwise_bound_stability_true_filter} (in \eqref{eq:lambda_star_tilde} and \eqref{eq:lambda_star_true_filter} respectively).
	\end{prop}
	
	\begin{proof}
		Assumption~\assumptionref{assmp:approx_filter} allows us to move our analysis from the simplex to $\R^n$ by defining the usual transformations $\theta_k^i \, : \, (0,1)^{\times 2} \rightarrow \R$. The dynamics of $\tilde \theta_k^i = \log \frac{\tilde \pi_t^i}{\tilde \pi_t^k}$ are given by
		\begin{equation*}
		\di \log \frac{\tilde \pi^i}{\tilde \pi^k}(t)
		= \frac{1}{\tilde \pi_t^i} \big( \tilde f_t^i \dt + \tilde g_t^i \di Y_t \big)
		- \frac{1}{\tilde \pi_t^k} \big( \tilde f_t^k \dt + \tilde g_t^k \di Y_t \big) \\
		+ \frac{1}{2} \Bigg( \bigg( \frac{\tilde g^k_t}{\tilde \pi_t^k}\bigg)^2 - \bigg(  \frac{\tilde g^i_t}{\tilde \pi_t^i} \bigg)^2      \Bigg) \dt,
		\end{equation*}
		so that, letting $\theta_k^i(t) = \log \frac{\pi_t^i}{\pi_t^k}$ and $\Delta_{ik}(t) = \theta_k^i(t) - \tilde \theta_k^i(t)$, defining the innovation process $B_t = Y_t - \int_0^t \pi_s^{\top}h \ds$, and recalling \eqref{eq:evol_theta_ik} for the dynamics of $\theta_k^i(t)$, we have
		\begin{align}\label{eq:delta_diff_approximate_filter}
		\di \Delta_{ik} (t)
		&= \Bigg[ \frac{1}{\pi_t^i} \Big( \sum_{j=0}^n q_{ji} \pi_t^j  \Big) - \frac{\tilde f_t^i}{\tilde \pi_t^i}  \Bigg] \dt
		+ \Bigg[ \frac{\tilde f_t^k}{\tilde \pi_t^k}  - \frac{1}{\pi_t^k} \Big( \sum_{j=0}^n q_{jk} \pi_t^j  \Big) \Bigg] \dt
		\nonumber \\
		&\quad+ \frac{1}{2} \Bigg( (h^k)^2 - (h^i)^2 + \bigg( \frac{\tilde g^i_t}{\tilde \pi_t^i} \bigg)^2 - \bigg( \frac{\tilde g^k_t}{\tilde \pi_t^k}  \bigg)^2  \Bigg) \dt
		+ \Big( h^i - h^k - \frac{\tilde g^i_t}{\tilde \pi_t^i} + \frac{\tilde g^k_t}{\tilde \pi_t^k}   \Big) (\di B_t + \pi_t^{\top}h \dt).
		\end{align}
		This equation might seem a bit daunting at first, but it can be treated exactly as we did in the case of misspecified $Q$ and $h$.
		Adding and subtracting terms as appropriate, we can rewrite \eqref{eq:delta_diff_approximate_filter} as
		\begin{align}\label{eq:delta_diff_approximate_filter_errs}
		\di \Delta_{ik} (t)
		&= \bigg( \sum_{\substack{j = 0 \\ j\neq i}}^{n} q_{ji} \bigg( \frac{\pi_t^j}{\pi_t^i} - \frac{\tilde \pi_t^j}{\tilde \pi_t^i}   \bigg) 
		-\sum_{\substack{j = 0 \\ j\neq k}}^{n} q_{jk} \bigg( \frac{\pi_t^j}{\pi_t^k} - \frac{\tilde \pi_t^j}{\tilde \pi_t^k}   \bigg)  \bigg) \dt 
		+ \big( \mathcal{E}_t^{1,i} - \mathcal{E}_t^{1,k}  \big) \dt \nonumber \\
		&\quad + \frac{1}{2} \big( \mathcal{E}_t^{2,k}   - \mathcal{E}_t^{2,i} \big)\dt
		+ \big( \mathcal{E}_t^{3,i}  - \mathcal{E}_t^{3,k} \big) (\di B_t + \pi_t^{\top}h \dt),
		\end{align}
		where for all $j \in \mathbf{N}$ we have defined the error terms
		\begin{equation}\label{eq:error_terms}
		\mathcal{E}_t^{1, j}  = \Bigg( \sum_{m=0}^{n} q_{mj} \frac{\tilde \pi_t^m}{\tilde \pi_t^j}\Bigg)  - \frac{\tilde f_t^j}{\tilde \pi_t^j}, \qquad
		\mathcal{E}_t^{2, j}  = (h^j)^2 - \bigg( \frac{\tilde g^j_t}{\tilde \pi_t^j}\bigg)^2, \qquad
		\mathcal{E}_t^{3, j}  = h^j - \frac{\tilde g^j_t}{\tilde \pi_t^j}.
		\end{equation}
		
		Now we proceed as in the proof of Theorem~\ref{thm:robustness_misspecified_model_params} by letting $\boldsymbol{\Delta}_t = \{\Delta_{ik}(t)\}_{i \in \mathbf{N}, k \in \mathbf{N}}$ be the family of processes with evolution equations given by \eqref{eq:delta_diff_approximate_filter_errs}, defining the process $LSE_{\alpha}(\Delta_{\cdot})(t)$ and its dynamics, and finally taking $\alpha \rightarrow \infty$ to yield our error estimates.
		
		Letting once more $T_{ik}(t) = \frac{\pi_t^i}{\pi_t^k} - \frac{\tilde \pi_t^i}{\tilde \pi_t^k}$ for all $(i,k) \in \mathbf{N} \times \mathbf{N}$, we have, for all $s \le t < \infty$,
		\begin{align*}
		LSE_{\alpha}&(\boldsymbol{\Delta}_{\cdot})(t)
		= LSE_{\alpha}(\boldsymbol{\Delta}_{\cdot})(s) 
		+ \int_s^t S^{arg}_{\alpha} \Big(\boldsymbol{\Delta}_{\cdot}, \sum_{j\neq i}^{n} q_{ji} T_{ji}(\cdot)
		-\sum_{j\neq k}^{n} q_{jk} T_{jk}(\cdot) \Big)(r) \di r \\
		& + \int_s^t S^{arg}_{\alpha} \big(\boldsymbol{\Delta}_{\cdot}, \mathcal{E}_{\cdot}^{1,i} - \mathcal{E}_{\cdot}^{1,k} \big)(r) \di r
		+ \frac{1}{2} \int_s^t S^{arg}_{\alpha} \big(\boldsymbol{\Delta}_{\cdot}, \mathcal{E}_{\cdot}^{2,k} - \mathcal{E}_{\cdot}^{2,i} \big)(r) \di r \\
		& + \int_s^t S^{arg}_{\alpha} \big(\boldsymbol{\Delta}_{\cdot}, \mathcal{E}_{\cdot}^{3,i} - \mathcal{E}_{\cdot}^{3,k} \big)(r) (\di B_r + \pi_s^{\top}h \di r) \\
		& + \frac{1}{2} \int_s^t  \sum_{(j,l)} \sum_{(u,v) \neq (j,l)} \frac{\alpha e^{ \alpha (\Delta_{jl}(r)+\Delta_{uv}(r))}}{(\sum_{(i,k)} e^{\alpha \Delta_{ik}(r)})^2}  \, \bigg( \big( \mathcal{E}_r^{3,j}  - \mathcal{E}_r^{3,l} \big)^2 - \big( \mathcal{E}_r^{3,j}  - \mathcal{E}_r^{3,l} \big) \big( \mathcal{E}_r^{3,u}  - \mathcal{E}_r^{3,v} \big) \bigg) \di r.
		\end{align*}
		
		Note that by Lemma~\ref{lemma:non_divergence_to_boundary} and Assumption~\assumptionref{assmp:approx_filter}, the first four integrands on the right-hand side have enough regularity to apply dominated convergence for Lebesgue or stochastic integrals when taking the limit as $\alpha \rightarrow \infty$. For the final term, we invoke once more Proposition~\ref{prop:local_time_bound}, which is justified by Assumption~\assumptionref{assmp:approx_filter}, to bound the integral in terms of the local times of the difference processes $\Delta_{ik} - \Delta_{jl}$ as we let $\alpha \rightarrow \infty$. This yields, for all $s \le t < \infty$,
		\begin{align}
		\Delta_{\infty}(t)
		&\le \Delta_{\infty}(s) + \int_s^t \frac{1}{|\mathcal{I}_r|} \sum_{(i,k) \in \mathcal{I}_r} \bigg( \sum_{j\neq i}^{n} q_{ji} T_{ji}(r)
		-\sum_{j\neq k}^{n} q_{jk} T_{jk}(r) \bigg) \di r \nonumber \\
		&\quad+ \int_s^t \max_{i,k} \Big\{ \mathcal{E}_r^{1, i}  - \mathcal{E}_r^{1, k}  +\frac{1}{2} \big(  \mathcal{E}_r^{2, k} - \mathcal{E}_r^{2, i} \big)     \Big\} \di r
		+ \max_j |h^j|  \int_s^t \max_{i,k} \big\{  \mathcal{E}_r^{3, i} - \mathcal{E}_r^{3, k} \big\} \di r \nonumber \\
		&\quad + \int_s^t \frac{1}{|\mathcal{I}_r|}\sum_{(i,k) \in \mathcal{I}_r} \big(\mathcal{E}_r^{3, i} - \mathcal{E}_r^{3, k}\big) \di B_r 
		+ \frac{1}{4}\sum_{(i,k)} \sum_{(j,l) \neq (i,k)} \int_s^t \di L^{0}_r(\Delta_{ik}(\cdot) - \Delta_{jl}(\cdot)), \label{eq:integral_ineq_error_bis}
		\end{align}
		where for all $r \in [s,t]$, we have defined $\mathcal{I}_r = \{ (i,k) \, : \, \Delta_{i k}(r) = \Delta_{\infty}(r)  \} \subset \mathbf{N}\times \mathbf{N}$ to be the argmax of $\boldsymbol{\Delta}_r$, and let $|\mathcal{I}_r|$ denote its size.
		
		Finally, consider the drift terms in the first integral on the right-hand side. By Lemma~\ref{lemma:argmax}, for all $r \in [s,t]$ we have that
		\begin{equation*}
		\Delta_{ik}(r) = \Delta_{\infty}(r) = \hilbert (\pi_r, \tilde \pi_r) = \log \max_{j} \frac{\pi_r^j}{\tilde \pi_r^j} - \log \min_j \frac{\pi_r^j}{\tilde \pi_r^j} = : \log M_r - \log \frac{1}{m_r}, \quad \forall (i,k) \in \mathcal{I}_r,
		\end{equation*}
		where $M_r \ge 1$ and $1/m_r \le 1$ are respectively the pointwise maximum and minimum ratio between the components of $\pi_r$ and $\tilde \pi_r$. By Lemma~\ref{lemma:diff_pi_pi_tilde_jk} we have that, for all $r \in [s,t]$ and for all $(i,k) \in \mathcal{I}_r$, $T_{ji}(r) \le 0$ and $T_{jk}(r) \ge 0$, for all $j \in \mathbf{N}$. Then the first integral on the right-hand side is negative and in particular
		\begin{equation*}
		\int_s^t \hspace{-3pt} \frac{1}{|\mathcal{I}_r|} \hspace{-2pt} \sum_{(i,k) \in \mathcal{I}_r} \hspace{-5pt} \bigg( \sum_{j\neq i}^{n} q_{ji} T_{ji}(r)
		-\sum_{j\neq k}^{n} q_{jk} T_{jk}(r) \bigg) \di r
		\le - \int_s^t \hspace{-5pt} \min_{(i,k) \in \mathcal{I}_r} \hspace{-4pt} \bigg( \sum_{j\neq k}^{n} q_{jk} T_{jk}(r) - \sum_{j\neq i}^{n} q_{ji} T_{ji}(r) \bigg) \di r.
		\end{equation*}
		Now we can minimize the integrand with algebraic calculations as in the proof of Theorem~\ref{thm:contraction}, or Proposition~\ref{prop:pathwise_bound_stability_ode}, or Corollary~\ref{cor:pathwise_bound_stability_true_filter}, which yields
		\begin{align*}
		\int_s^t \frac{1}{|\mathcal{I}_r|} \sum_{(i,k) \in \mathcal{I}_r} \bigg( \sum_{j\neq i}^{n} q_{ji} T_{ji}(r)
		-\sum_{j\neq k}^{n} q_{jk} T_{jk}(r) \bigg) \di r
		\le - 2 \int_s^t \kappa_r \sinh \bigg( \frac{ \Delta_{\infty}(r)}{2}\bigg) \di r,
		\end{align*}
		where the decay rate $\kappa_r$ can be chosen to be any of the coefficients $\lambda$ from Theorem~\ref{thm:contraction}, $\tilde \lambda^{\star}(t, \tanh (\Delta_{\infty}(t)/4))$ from \eqref{eq:lambda_star_tilde} in Proposition~\ref{prop:pathwise_bound_stability_ode}, or $\lambda^{\star}(t, \tanh (\Delta_{\infty}(t)/4))$ from \eqref{eq:lambda_star_true_filter} in Corollary~\ref{cor:pathwise_bound_stability_true_filter}.
	\end{proof}

	Theorem~\ref{thm:expected_hilbert_bounds} and Theorem~\ref{thm:pathwise_decay_approx_error} now follow easily from the above proposition.
	
	\begin{proof}[Proof of Theorem~\ref{thm:expected_hilbert_bounds}]
		Start from \eqref{eq:integral_ineq_error}. For all $t < \infty$, we bound $\kappa_t$ from below by the deterministic rate $\lambda = 2 \min_{i \neq k} \sqrt{q_{ik}q_{ki}}$. Moreover, recall that $2 \sinh (x/2) \ge x$ for $x \ge 0$. Substitute both these bound in the first integral in the right-hand side of \eqref{eq:integral_ineq_error}.
		We take expectation and note that the stochastic integral vanishes, since it is a martingale (as the integrand is locally $L^2$-integrable by assumption~\assumptionref{assmp:approx_filter}). A modification of the standard Gr\"onwall argument to deal with Lebesgue--Stieltjes measures (as in the proof of Theorem~\ref{thm:robustness_misspecified_model_params}) concludes the proof.
	\end{proof}
	
	\begin{proof}[Proof of Theorem~\ref{thm:pathwise_decay_approx_error}]
		If $\mathcal{E}^{3,i}_t = 0$ for all $i \in \mathbf{N}$ and $t < \infty$, then $\mathcal{E}^{2,i}_t = 0$ as well.
		Then \eqref{eq:delta_diff_approximate_filter_errs} reduces to
		\begin{equation*}
		\di \Delta_{ik} (t)
		= \bigg( \sum_{\substack{j = 0 \\ j\neq i}}^{n} q_{ji} \bigg( \frac{\pi_t^j}{\pi_t^i} - \frac{\tilde \pi_t^j}{\tilde \pi_t^i}   \bigg) 
		-\sum_{\substack{j = 0 \\ j\neq k}}^{n} q_{jk} \bigg( \frac{\pi_t^j}{\pi_t^k} - \frac{\tilde \pi_t^j}{\tilde \pi_t^k}   \bigg)  \bigg) \dt 
		+ \big( \mathcal{E}_t^{1,i} - \mathcal{E}_t^{1,k}  \big) \dt,
		\end{equation*}
		for all $(i,k) \in \mathbf{N} \times \mathbf{N}$, so we recover $C^1$ dynamics for the difference processes $\Delta_{ik}(t)$. A $C^1$ process does not generate local time, so \eqref{eq:integral_ineq_error} simplifies to
		\begin{equation*}
		\Delta_{\infty}(t)
		\le \Delta_{\infty}(s) - 2 \int_s^t \kappa_r \sinh \bigg( \frac{ \Delta_{\infty}(r)}{2}\bigg) \di r
		+ \int_s^t \max_{i,k} \big\{ \mathcal{E}_r^{1, i}  - \mathcal{E}_r^{1, k}\big\} \di r,
		\end{equation*}
		for all $s \le t$, where $\kappa_t>0$ is the coefficient given by any of the rates in \eqref{eq:possible_rates}. Now the second part of the theorem follows easily, by first bounding $\kappa_r$ from below by a positive (measurable) process $ \gamma_t$ given by one of
		\begin{equation*}
		\gamma_t = \left\{ \begin{array}{l}
		\lambda, \\[10pt]
		\tilde \lambda^{\star}_t \le \tilde \lambda^{\star} \Big( t, \tanh \Big(\frac{\hilbert (\pi_t, \tilde \pi_t)}{4}\Big) \Big), \\[10pt]
		\lambda^{\star}_t \le \lambda^{\star} \Big( t, \tanh \Big(\frac{\hilbert (\pi_t, \tilde \pi_t)}{4}\Big) \Big), 
		\end{array} \right.
		\end{equation*}
		where $\tilde \lambda^{\star}_t$ and $\lambda^{\star}_t$ are as in Proposition~\ref{prop:pathwise_bound_stability_ode} and Corollary~\ref{cor:pathwise_bound_stability_true_filter}. Then, recalling once more that $2\sinh (x/2) \ge x$ for $x \ge 0$, the usual Gr\"onwall argument yields
		\begin{equation*}
		\Delta_{\infty}(t) \le \Delta_{\infty}(0) e^{- \int_0^t \gamma_s \ds}
		+ \int_0^t e^{- \int_s^t \gamma_r \di r} \max_{i,k}  \big\{ \mathcal{E}_s^{1,i}  - \mathcal{E}_s^{1,k} \big\} \ds,
		\end{equation*}
		which is \eqref{eq:hilbert_pathwise_approx_bound} for $\gamma_t = \tilde \lambda^{\star}_t$.

		We now look for a tighter bound. Consider the process $X_t = \tanh \big( \Delta_{\infty}(t)/4\big)$. Applying the chain rule we have
		\begin{align*}
		\di X_t 
		= \frac{1}{4} \cosh^{-2} \bigg( \frac{\Delta_{\infty}(t)}{4}  \bigg) \di \Delta_{\infty}(t)
		&\le - \kappa_t \frac{\sinh \big( \frac{\Delta_{\infty}(t)}{2}  \big)}{2 \cosh^2 \big( \frac{\Delta_{\infty}(t)}{4}  \big)} \dt + \frac{1}{4} \frac{\max_{i,k} \big\{ \mathcal{E}_r^{1, i}  - \mathcal{E}_r^{1, k}\big\}}{\cosh^2 \big( \frac{\Delta_{\infty}(t)}{4}  \big)} \\
		&\le - \kappa_t X_t \dt + \frac{1}{2} \max_{i,k} \big\{ \mathcal{E}_t^{1, i}  - \mathcal{E}_t^{1, k}\big\}\frac{X_t}{\sinh \big( 2 \arctanh(X_t) \big)} \dt,
		\end{align*}
		where we have used the identity $\sinh (2 x) = 2 \sinh(x) \cosh(x)$. Since $\sinh ( 2 \arctanh(x)) = \frac{2x}{1-x^2}$, we can rewrite the above as
		\begin{equation*}
		\di X_t  \le \alpha(t, X_t) \dt , \quad \textrm{where} \quad \alpha(t, X_t) = -  \tilde \lambda^{\star}(t, X_t) X_t + \frac{1}{4} \max_{i,k} \big\{ \mathcal{E}_t^{1, i}  -\mathcal{E}_t^{1, k}\big\} \big( 1 - X_t^2 \big),
		\end{equation*}
		where we have substituted $\tilde \lambda^{\star}(t, X_t)$ for $\kappa_t$ for clarity in the exposition below (but the arguments are analogous whether $\kappa_t$ is the deterministic rate $\lambda$ from Theorem~\ref{thm:contraction}, the pathwise rate $\tilde \lambda^{\star}_t$ from Proposition~\ref{prop:pathwise_bound_stability_ode}, or the coefficient $\lambda^{\star}(t, X_t)$ or the pathwise rate $\lambda^{\star}_t$ from Corollary~\ref{cor:pathwise_bound_stability_true_filter}). Bounding $\tilde \lambda^{\star}(t, X_t)$ from below by $\tilde \lambda^{\star}_t$, and $(1- X_t^2)$ from above by 1, another application of Gr\"onwall yields \eqref{eq:tanh_pathwise_approx_bound}.
		
		We recall \eqref{eq:lambda_star_tilde} for the definition of $\tilde \lambda^{\star}$. Note that the mapping $x \mapsto  \alpha(t,x)$ is locally Lipschitz continuous (with Lipschitz constant dependent on $\omega, t$ and $x$), since $x \mapsto \tilde \lambda^{\star}(t, x)x$ is locally Lipschitz continuous and $\max_{i,k} \big\{ \mathcal{E}_t^{1, i}  - \mathcal{E}_t^{1, k}\big\}$ is locally bounded by Lemma~\ref{lemma:non_divergence_to_boundary} and Assumption~\assumptionref{assmp:approx_filter}.
		Now let $u_t$ be the solution to the ODE with random coefficients given by
		\begin{equation}\label{eq:ode_u_error}
		\frac{\di u_t}{\dt} = \alpha(t, u_t), \qquad
		u_0 = X_0 = \tanh \bigg(\frac{\Delta_{\infty}(0)}{4}\bigg),
		\end{equation}
		where $\alpha$, or, specifically, $\tilde \lambda^{\star}$ and $\max_{i,k} \big\{ \mathcal{E}_t^{1, i}  - \mathcal{E}_t^{1, k}\big\}$ depend on the process $\tilde \pi_t$, which is fixed for each $\omega$. Recall that, since $\mu, \nu \in \mathring{\mathcal{S}}^n$ by assumption, $\hilbert (\mu, \nu) < \infty$, and therefore $u_0 \in (0,1)$. Since the right-hand side is locally Lipschitz, \eqref{eq:ode_u_error} has a unique solution $u_t$ up to its first explosion time $T > 0$ (again, see e.g.~\cite[Theorem~2.5]{teschl_odes}). Now, if $T < \infty$, then $T$ is the first time such that $u_{T-}\hspace{-4pt} =1$. By continuity of $u_t$, for all $\varepsilon > 0$ there exists a $\delta > 0$ such that for $s \in (T - \delta, T)$, we have $u_s \in (1-\varepsilon, 1)$. Then
		\begin{align*}
			1
			&= u_{T - \delta} + \int_{T - \delta}^{T} \alpha(s, u_s) \ds	\\
			&\le u_{T - \delta} - \frac{\delta(2- \varepsilon)(1-\varepsilon)}{\varepsilon} \inf_{s \in [T - \delta, T]} \min_{i \neq k} \bigg\{ q_{ik} \frac{\tilde \pi_s^i}{\tilde \pi_s^k} \bigg\}
			+ \frac{\delta \varepsilon(2-\varepsilon)}{4} \sup_{s \in [T - \delta, T]} \max_{i,k} \big\{ \mathcal{E}_s^{1, i}  - \mathcal{E}_s^{1, k}\big\},
		\end{align*}
		using (strict) positivity of $\lambda^{\star}$ and $\max_{i,k} \big\{ \mathcal{E}_t^{1, i}  - \mathcal{E}_t^{1, k}\big\}$, and that $\frac{1 + u_t}{1-u_t} \ge \frac{(2- \varepsilon)(1-\varepsilon)}{\varepsilon}$ and $(1-u_s^2) \le \varepsilon(2-\varepsilon)$ for $s \in (T - \delta, T)$. Since, for small enough $\varepsilon$, the negative term dominates the positive term, this implies $1 < u_{T-\delta}$, which is strictly less than $1$, and therefore a contradiction. Then $T = \infty$ and \eqref{eq:ode_u_error} has a unique solution for all $t \ge 0$.

		A similar argument proves that $u_t > 0$ for all $t \ge 0$, and finally Lemma~\ref{lemma:comparison} yields the theorem.
	\end{proof}

	\begin{proof}[Proof of Corollary~\ref{cor:exchanging_pi_with_pi_tilde}]
	Analogous to the proof of Theorem~\ref{thm:pathwise_decay_approx_error}.
	\end{proof}
	
	\subsubsection{A numerical example}\label{sec:numerics_approx}
	We conclude this section, and the paper, by testing our bounds in a couple of simulations.
	
	We consider a Wonham filter with approximate model parameters, whose dynamics are given by \eqref{eq:wonham_wrong_params}. We assume $h$ to be known, so $\tilde h = h$. We approximate $Q$ by applying a non-negative factorization algorithm: we subtract the diagonal from $Q$, approximate the resulting positive matrix using the \texttt{NMF} class from the python package \texttt{sklearn.decomposition}, and reconstruct the diagonal to ensure all the rows sum to 0 to yield $\tilde Q$. In this setting, there is no error due to the misspecification of $h$, so we do not have to worry with estimating the local time terms. We consider the error bounds given in Theorem~\ref{thm:pathwise_decay_approx_error}.
	
	In the figure below we compare this approximate filter with the Wonham filter for a 3-state and a 6-state Markov chain. In each case, we take the $Q$ matrix to be given by
	\begin{equation*}
	Q = \left( \begin{array}{ccc}
	-3 & 1 & 2 \\
	1 & -3 & 2 \\
	1.5 & 1.5 & -3
	\end{array}
	\right),
	\qquad
	Q = \left( \begin{array}{cccccc}
	-9 & 3 & 1 & 1.5 & 2.5 & 1 \\
	1 & -7.5 & 1 & 2 & 2.3 & 1.2 \\
	3 & 2 & -8 & 1& 1 & 1 \\
	2 & 1.3 & 1 & -6 & 0.7 & 1 \\
	1.1 & 1 & 0.9 & 3 & -9 & 3 \\
	1 & 1 & 3& 2 & 2.5 & -9.5
	\end{array}
	\right),
	\end{equation*}
	and the sensor function $h$ to be
	\begin{equation*}
	h = \left( -1, 0, 1 \right), \qquad
	h = \left(-3, -2, -1, 1, 2, 3 \right).
	\end{equation*}
	For the 3-state Markov chain, we take the initial law of the signal $X$ to be given by its ergodic distribution, i.e. $\mathrm{law} (X_0) = \big(0.3, 0.3, 0.4 \big)$. This is also the initial condition for the Wonham filter $\pi_t$. The approximate rate matrix $\tilde Q$ for the approximate filter $\tilde \pi_t$ is obtained using a 2-channel NMF approximation of $Q$. We take the initial condition for $\tilde \pi_t$ to be $\tilde \pi_0 = \big( 0.2, 0.2, 0.6 \big)$. In the 6-state case, we start the signal $X$ quite close to the boundary of $\mathcal{S}^5$, with its initial law given by $\mu = \big( 0.5, 0.04, 0.09, 0.2, 0.04, 0.13   \big)$, which is also the initial condition for $\pi_t$. For the rate matrix $\tilde Q$ for $\tilde \pi_t$ we use a 4-channel NMF approximation of $Q$. We start $\tilde \pi_t$ also relatively close to the boundary of $\mathcal{S}^5$, but near a different edge from $\mu$, and take $\tilde \pi_0 = \big( 0.25, 0.1, 0.06, 0.07, 0.22, 0.3 \big)$.
	
	For transparency, we write here the matrices $\tilde Q$ (rounded to the second significant digit) resulting from the NMF approximation in each case:
	\begin{equation*}
		\tilde Q \approx \left( \begin{array}{ccc}
		-2.5 & 0.5 & 2 \\
		0.5 & -2.5 & 2 \\
		1.5 & 1.5 & -3
		\end{array}
		\right),
		\qquad
		\tilde Q \approx \left( \begin{array}{cccccc}
		-9 & 3.04 & 1.04 & 1.54 & 2.43 & 0.95 \\
		0.94 & -7.25 & 1.70 & 2.02 & 1.58 & 1.01 \\
		2.92 & 2.05 & -7.8 & 0.69 & 0.88 & 1.26 \\
		2.11 & 1.24 & 0.52 & -5.34 & 0.84 & 0.62 \\
		1.13 & 0.86 & 0.63 & 3.02 & -8.68 & 3.04 \\
		1.02 & 0.77 & 2.64 & 1.92 & 2.92 & -9.28
		\end{array}
		\right).
		\end{equation*}

	In Figure~\ref{fig:approx_filters}, on the left, both for the 3-state and the 6-state nonlinear filter, we plot 100 realizations of the Hilbert error between $\pi_t$ and $\tilde \pi_t$ (and their sample mean) in blue, and of the error bounds from Theorem~\ref{thm:pathwise_decay_approx_error} (and their sample means). Since these bounds are path-by-path, each realization of the error between $\pi_t$ and $\tilde \pi_t$ has three corresponding error bounds: in fuchsia we plot $4 \arctanh (u_t)$, where $u_t$ is the (numerical) solution to the ODE \eqref{eq:thm_approx_pathwise_ode}; in green we plot the bound \eqref{eq:hilbert_pathwise_approx_bound} where the decay rate is given by $\tilde \lambda_t$ as defined in \eqref{eq:thm_approx_pathwise_rate}; in red we plot again the bound \eqref{eq:hilbert_pathwise_approx_bound}, but using the deterministic decay rate $\lambda$ from Theorem~\ref{thm:contraction} instead. The error terms \eqref{eq:error_terms_approx} are evaluated pathwise at each time-step. In the pictures on the right, for the same simulations, we plot 100 realizations of $\tanh (\hilbert(\pi_t, \tilde \pi_t)/4)$ (blue), and of the numerical solution $u_t$ to \eqref{eq:thm_approx_pathwise_ode} (fuchsia). 
	
	\begin{figure}[htp]
		\centering
		\includegraphics[width=0.49\textwidth]{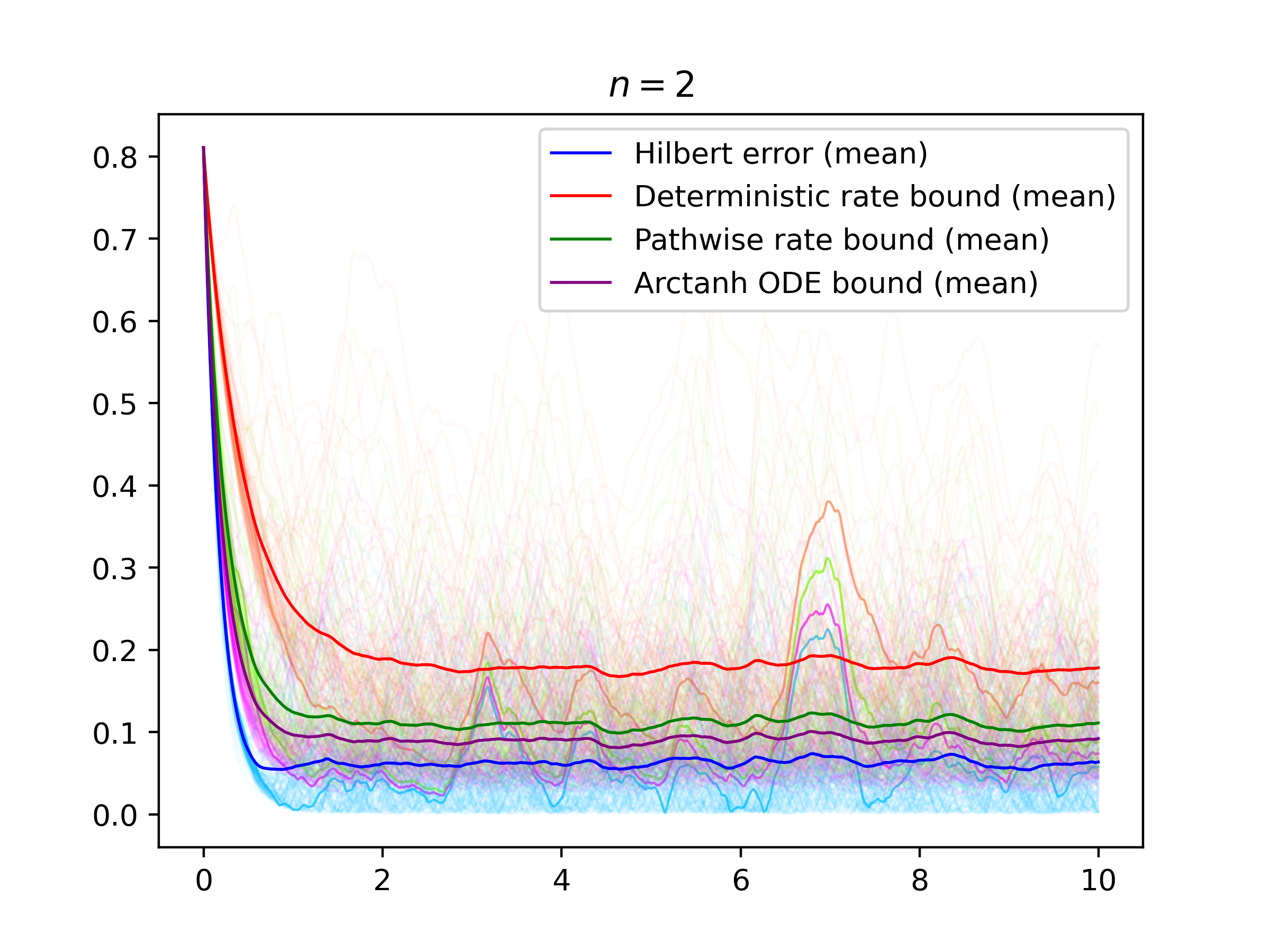}
		\includegraphics[width=0.49\textwidth]{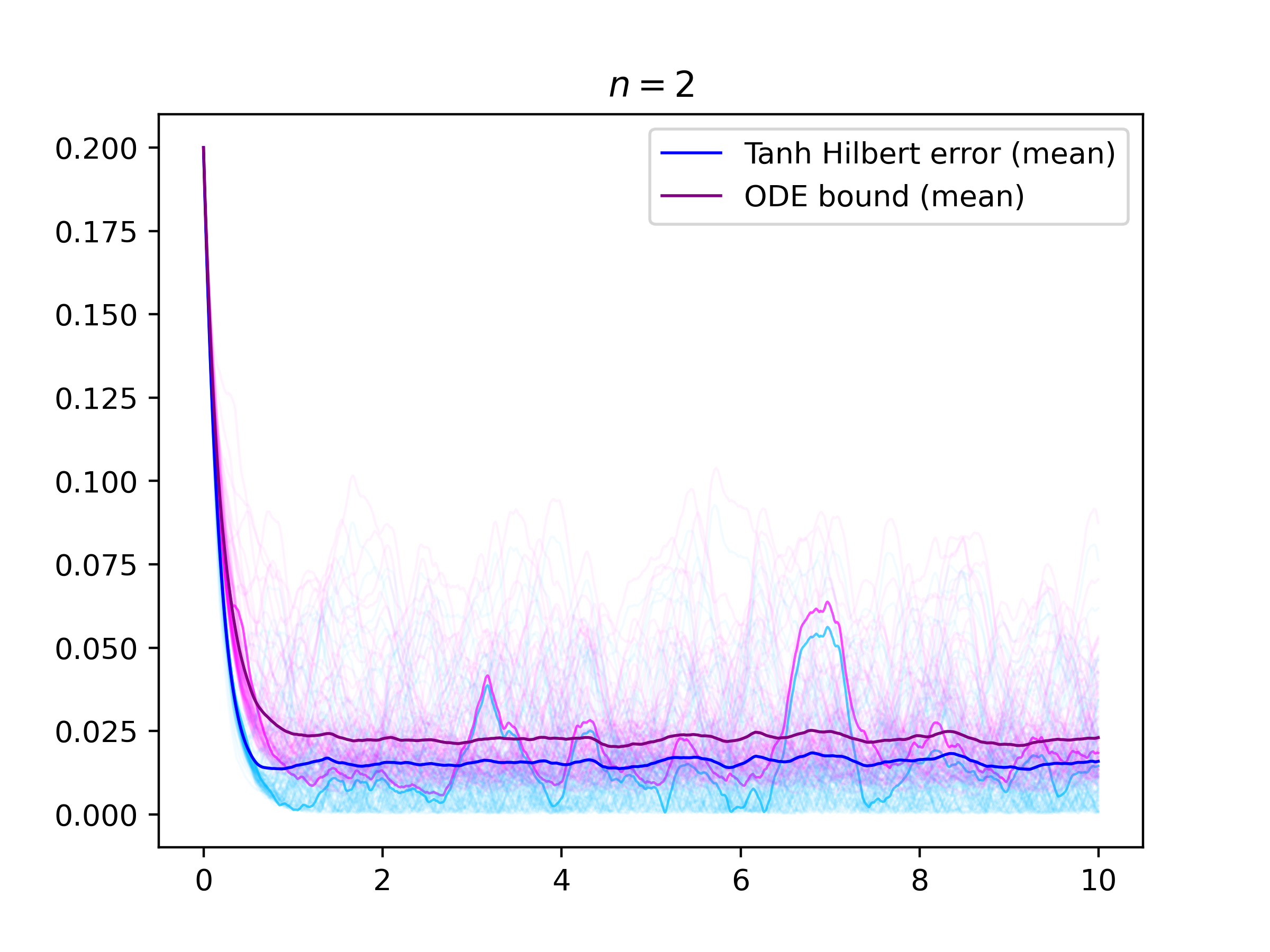}
		\includegraphics[width=0.49\textwidth]{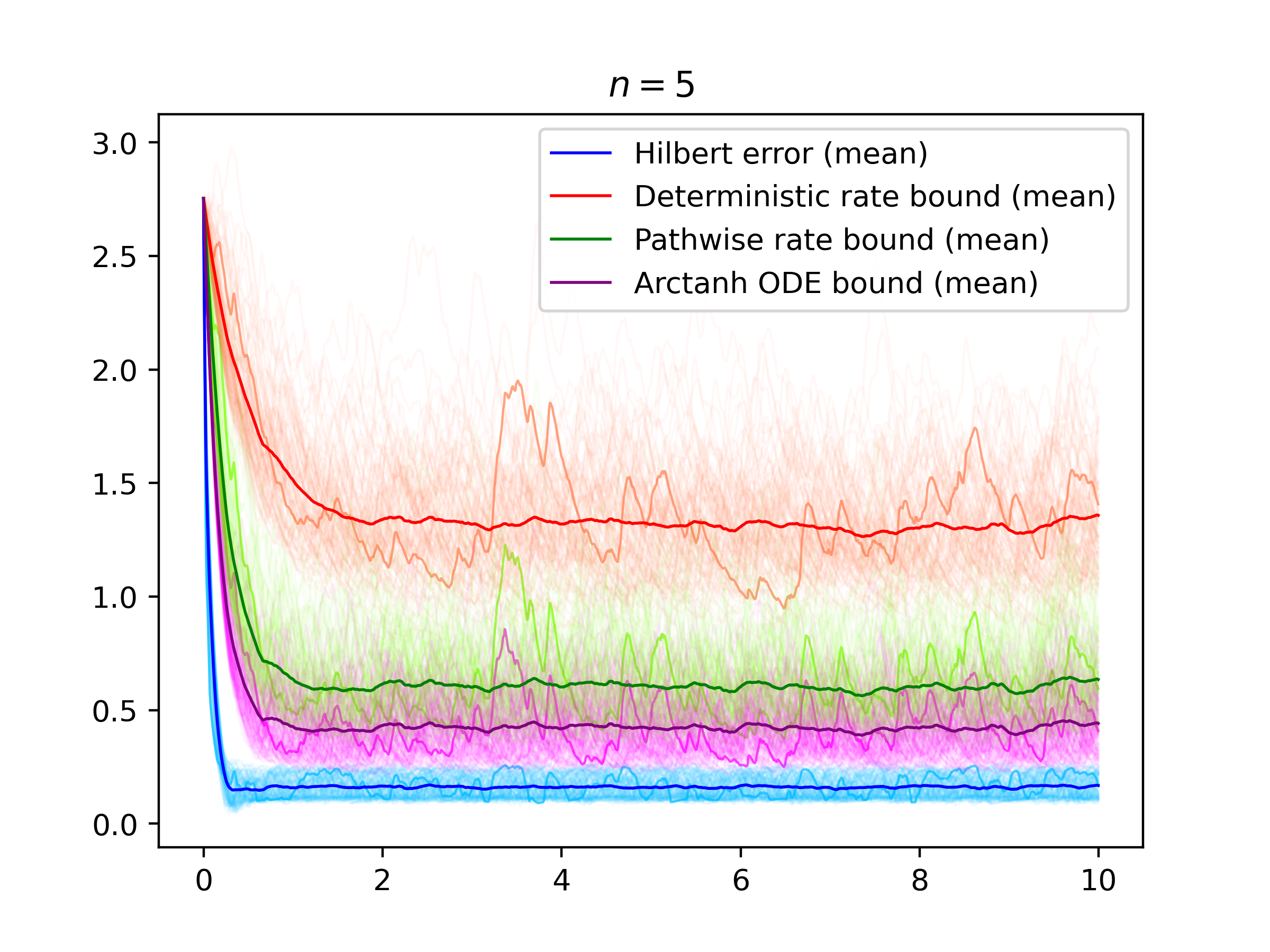}
		\includegraphics[width=0.49\textwidth]{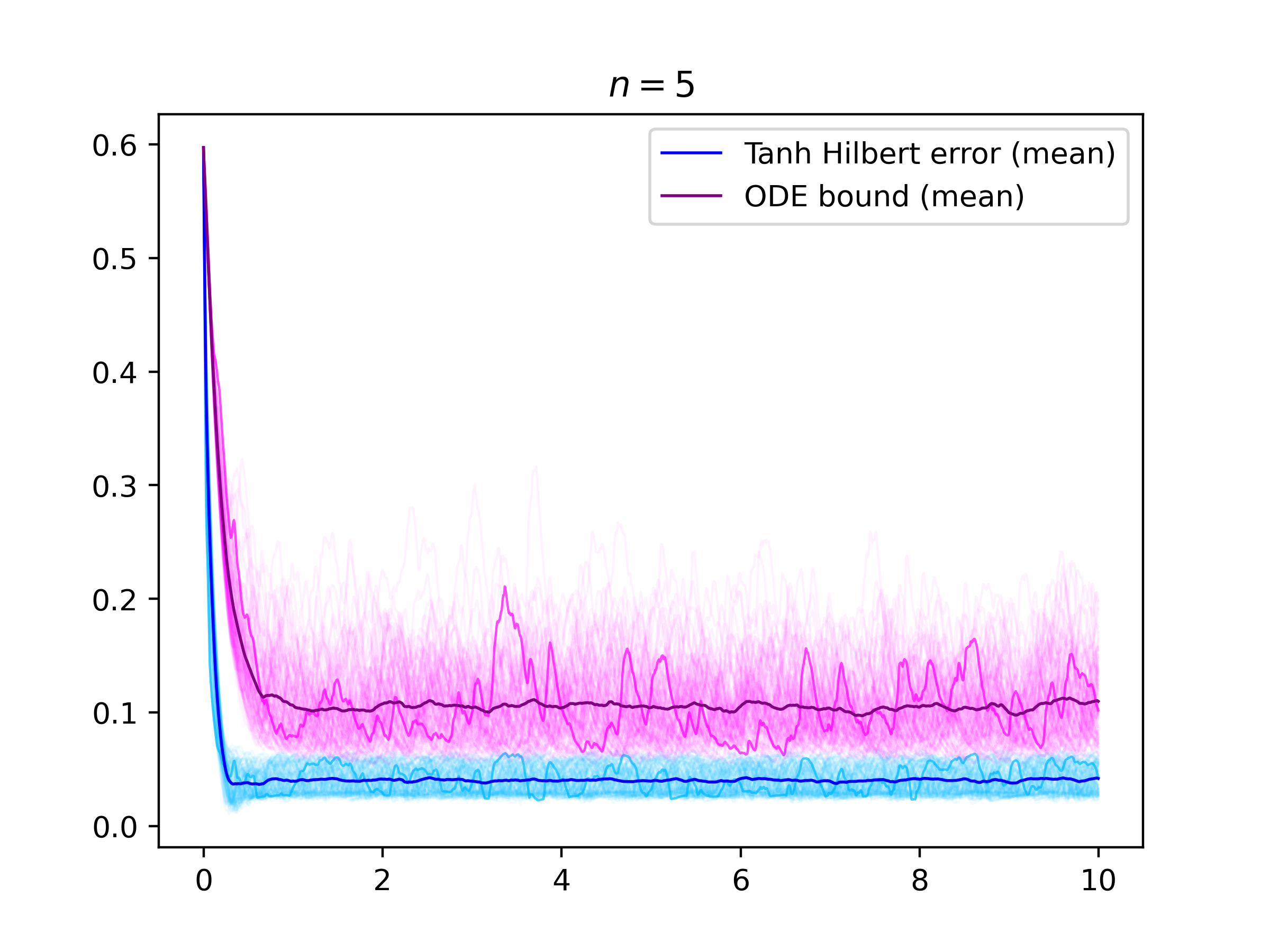} 
		\caption{For dimensions $n=2, 5$, we test our error bounds from Theorem~\ref{thm:pathwise_decay_approx_error} against the actual Hilbert error between the Wonham filter and an approximate filter. On the right we plot 100 realizations of the Hilbert projective error $\hilbert (\pi_t, \tilde \pi_t)$ (faded, light blue), of the ODE bound given by $4 \arctanh (u_t)$, where $u_t$ solves \eqref{eq:thm_approx_pathwise_ode} (faded, fuchsia), of the pathwise bound \eqref{eq:hilbert_pathwise_approx_bound} with pathwise decay rate $\tilde \lambda_t$ (faded, light green), and of the pathwise bound with deterministic decay rate $\lambda$ (faded, orange), for $t \in [0,10]$. We highlight one sample path of the Hilbert error at random, together with its three corresponding pathwise bounds. In blue, purple, green and red we plot the sample means of the errors and of the three bounds. On the right, for the same simulations, we plot 100 realizations of the quantity $\tanh(\hilbert(\pi_t, \tilde \pi_t)/4)$ together with the ODE bound $u_t$, and the sample means of both. Again, we highlight at random one realization of the $\tanh$ error and its corresponding ODE bound.}
		\label{fig:approx_filters}
	\end{figure}

	In the 3-state case, where the filter lives in $\mathcal{S}^2$, we can see that our estimates for the error are very close to its actual value (the ODE bounds given by the solution to \eqref{eq:thm_approx_pathwise_ode} in particular). In the 6-state case, with a 5-dimensional filter, our error bounds are less sharp. In fact, it is safe to assume that our error bounds get progressively worse as we increase the dimension of the state-space.
	
	Why is this the case? As we already mentioned in Section~\ref{sec:optimality_rate}, the main issue with our error bounds is the contraction rate. Our numerical experiments for the stability estimates (see Figure~\ref{fig:hilber_error_2d_100d}) show that the error contracts at a much faster rate than what we can prove. This makes sense, since by minimizing over all entries of $Q$, the decay rates in Theorem~\ref{thm:contraction} and Proposition~\ref{prop:pathwise_bound_stability_ode} give pathwise bounds for `worst case'-type of scenarios. A similar argument applies to our treatment of the approximation-error terms \eqref{eq:error_terms_approx}: to ensure our bounds hold, we need to maximize over all possible indices, and this implies that we are more and more likely to overestimate the errors as the dimension increases.
	
	There are a couple of directions that one could pursue at this point, to tighten our error estimates. The first would be to try to exploit some averaging over the indices, instead of simply minimizing/maximizing over them, to yield tighter decay rates/error terms. This could potentially be achieved if one looked for a bound in expectation instead of pathwise. This problem seems difficult, however, as it involves estimating the expectation of the $\argmin$ and $\argmax$ of the log differences between the ratios of $\pi_t$ and $\tilde \pi_t$. On the other hand, if, instead of proceeding analytically, one were able to estimate quantities numerically, it should be relatively easy to obtain good numerical estimates for the decay rate of the stability error in high dimensions, as we can see from the plots in Figure~\ref{fig:hilber_error_2d_100d}. The estimated rate can then be substituted into \eqref{eq:hilbert_pathwise_approx_bound} to yield tighter bounds for the error of approximate filters in high dimensions (which should hold with high probability). The issue of overestimating the error terms remains, but, since they are dominated by the negative exponentials, tightening the decay rate would yield a significant overall improvement.
	
	Numerical estimation of the decay rate opens up other possibilities as well. The arguments we developed in this paper work when the signal is given by any ergodic time-continuous Markov chain -- the strict positivity of the off-diagonal entries of $Q$ is only required to guarantee that the decay rates we derive are nonzero. In other words, the stability error of the Wonham filter decays as long as the signal is ergodic (as discussed by \cite{bax-chiga-lip04}) even when the $Q$-matrix is sparse. By discretization, the nonlinear filter on a compact state space, given by the solution to the Kushner--Stratonovich SPDE, is often approximated by a Wonham filter on a high number of states. However, the diffusion operator corresponds to a very sparse transition matrix. Given a numerical estimate for the decay rate of the Hilbert error of the discretized diffusion operator, one could then use our error estimates to understand the error of approximate filters in infinite dimensions.

	\paragraph{Acknowledgements} The research of EF was supported by the EPSRC under the award EP/L015811/1. SC acknowledges the support of the UKRI Prosperity Partnership Scheme (FAIR) under EPSRC Grant EP/V056883/1, the Alan Turing Institute and the Office for National Statistics (ONS), and the Oxford--Man Institute for Quantitative Finance.
	
	\newpage
	\appendix
	
	\section{The maximum process of a family of semimartingales}\label{app:smooth_max}
	
	In this appendix we use an appropriate smooth approximation to study the dynamics of the maximum of a family of continuous stochastic processes driven by a common Brownian motion.
	
	Recall the following smooth approximations of the maximum and the argmax.
	Let $\alpha \in (0, \infty)$ and let $\mathbf{x} = \{x_i\}_{i=0}^n$ be a sequence of real numbers.
	We define the \textit{LogSumExp} function $LSE_{\alpha}(\mathbf{x})$ as
	\begin{equation*}
	LSE_{\alpha}(\mathbf{x}) = \frac{1}{\alpha} \log \, \sum_k e^{\alpha x_k},
	\end{equation*}
	and the \textit{SmoothMax} function $S_{\alpha}(\mathbf{x})$ as
	\begin{equation*}
	S_{\alpha}(\mathbf{x}) = \frac{ \sum_{j} x_j e^{\alpha x_j}}{\sum_k e^{\alpha x_k} }.
	\end{equation*}
	
	Given a family $\mathbf{c} = \{c_i\}_{i=0}^n$ of real-valued coefficients, we also define the \textit{SoftArgMax} (or \textit{SoftMax}) function $S^{arg}_{\alpha}(\mathbf{x}, \mathbf{c})$ as
	\begin{equation*}
	S^{arg}_{\alpha}(\mathbf{x}, \mathbf{c}) = \frac{ \sum_{j} c_j e^{\alpha x_j}}{\sum_k e^{\alpha x_k} }.
	\end{equation*}
	
	We start by proving a few simple lemmata.
	
	\begin{notation}
		Let $\mathcal{I}$ be the argmax of $\mathbf{x}$, i.e. $\mathcal{I} : = \{ j \in \mathbf{N} \: : \:  x_j = \max_{i \in \mathbf{N}} x_i   \} \subset \mathbf{N}$.
	\end{notation}
	
	\begin{lemma}[Convergence to maximum]\label{lemma:convergence_to_max}
		\begin{equation*}
		\lim_{\alpha \rightarrow \infty} LSE_{\alpha} (\mathbf{x}) = \lim_{\alpha \rightarrow \infty} S_{\alpha}(\mathbf{x}) = \max_{i \in \{0, \dots, n\}} x_i
		\end{equation*}
	\end{lemma}
	\begin{proof}
		Let $M = \max_{i \in \{0, \dots, n\}} x_i$. We have that
		\begin{equation*}
		M=\frac{1}{\alpha} \log e^{\alpha M} \le LSE_{\alpha}(\mathbf{x}) \le \frac{1}{\alpha} \log \, \big( (n+1) e^{\alpha M} \big) = \frac{\log(n+1)}{\alpha} + M,
		\end{equation*}
		and taking the limit as $\alpha \rightarrow \infty$ yields the result. For the \textit{SmoothMax} function, consider $\mathcal{I}$, the argmax of $\mathbf{x}$, and let $|\mathcal{I}| = d \ge 1$ be its size. Then
		\begin{align*}
		S_{\alpha}(\mathbf{x})
		&= \sum_{j \in \mathcal{I}}\frac{ x_j}{d + \sum_{k \notin \mathcal{I}} e^{\alpha (x_k - x_j)} } + \sum_{j \notin \mathcal{I}}\frac{ x_j}{1 + \sum_{k \neq j} e^{\alpha (x_k - x_j)} } \\
		&= \frac{ dM}{d + \sum_{k \notin \mathcal{I}} e^{- \alpha (M - x_k)} }
		+ \sum_{j \notin \mathcal{I}}\frac{ x_j}{1 + \sum_{\subalign{k &\neq j \\ k &\in \mathcal{I}}} e^{\alpha (M - x_k)}  + \sum_{\subalign{k &\neq j \\ k &\notin \mathcal{I}}} e^{\alpha (x_k - x_j)}}
		\xrightarrow{\alpha \rightarrow \infty} M.
		\end{align*}
	\end{proof}
	\begin{lemma}\label{lemma:smooth_arg_max}
		Let $\mathcal{I}$ be the argmax of $\mathbf{x}$ and $|\mathcal{I}| = d \ge 1$ be its size. Then
		\begin{equation*}
		\lim_{\alpha \rightarrow \infty} S^{arg}_{\alpha}(\mathbf{x}, \mathbf{c}) = \frac{1}{d} \sum_{j \in \mathcal{I}} c_j.
		\end{equation*}
	\end{lemma}
	\begin{proof}
		Similar to Lemma \ref{lemma:convergence_to_max}.
	\end{proof}
	
	\begin{lemma}[Derivatives of $LSE_{\alpha}(\mathbf{x})$]
		\begin{align*}
		\frac{\partial}{\partial x_i} LSE_{\alpha}(\mathbf{x}) &= \frac{e^{\alpha x_i}}{\sum_k e^{\alpha x_k}}, \\
		\frac{\partial^2}{\partial x_i^2} LSE_{\alpha}(\mathbf{x}) &= \alpha \sum_{j \neq i}  \frac{ e^{\alpha(x_i + x_j)}}{\big( \sum_k e^{\alpha x_k}\big)^2}, \\
		\frac{\partial^2}{\partial x_i \partial x_j} LSE_{\alpha}(\mathbf{x}) &= - \alpha   \frac{ e^{\alpha(x_i + x_j)}}{\big( \sum_k e^{\alpha x_k}\big)^2}.
		\end{align*}
	\end{lemma}
	\begin{proof}
		Easy calculations.
	\end{proof}

	Now consider the function
	\begin{equation*}
	f_{\alpha}(\mathbf{x})
	= \sum_{i} \frac{\partial^2}{\partial x_i^2} LSE_{\alpha}(\mathbf{x})
	= - \sum_{i} \sum_{j\neq i} \frac{\partial^2}{\partial x_i \partial x_j} LSE_{\alpha}(\mathbf{x})
	= \alpha \sum_{i} \sum_{j \neq i}  \frac{ e^{\alpha(x_i + x_j)}}{\big( \sum_k e^{\alpha x_k}\big)^2}.
	\end{equation*}
	
	\begin{lemma}\label{lemma:unique_maximizer_f_alpha}
		If $\max_{i} x_i$ is unique, i.e. if $\exists ! \, j^{\star}$ such that $\max_{i} x_i = x_{j^{\star}}$, then
		\begin{equation*}
		\lim_{\alpha \rightarrow \infty} f_{\alpha}(\mathbf{x}) = 0.
		\end{equation*}
	\end{lemma}
	\begin{proof}
		Let $ x_{j^{\star}} := \max_{i} x_i$. Since $x_{j^{\star}}$ is the unique maximizer, there exists $\varepsilon_{j} > 0$ such that $x_j = x_{j^{\star}} - \varepsilon_{j}$ for all $j \neq j^{\star}$. Then we have
		\begin{align*}
		f_{\alpha}(\mathbf{x})
		&= \alpha \Bigg[ \sum_{j \neq j^{\star}}  \frac{ e^{\alpha(x_{j^{\star}} + x_j)}}{\big( \sum_k e^{\alpha x_k}\big)^2}
		+ \sum_{i \neq j^{\star}}  \frac{ e^{\alpha(x_i + x_j^{\star})}}{\big( \sum_k e^{\alpha x_k}\big)^2}
		+ \sum_{i \neq j^{\star}} \sum_{\subalign{j &\neq i \\j &\neq j^{\star}}}  \frac{ e^{\alpha(x_i + x_j)}}{\big( \sum_k e^{\alpha x_k}\big)^2} \Bigg] \\
		&= \frac{\alpha}{\big( e^{\alpha x_{j^{\star}}} + \sum_{k \neq j^{\star}} e^{\alpha(x_{j^{\star}} - \varepsilon_k)}\big)^2}
		\Bigg[ 2 \sum_{j \neq j^{\star}}  e^{\alpha(2 x_{j^{\star}} - \varepsilon_j)}
		+ \sum_{i \neq j^{\star}} \sum_{j \neq i}  e^{\alpha(2 x_{j^{\star}} - \varepsilon_i -\varepsilon_j)} \Bigg] \\
		&= \frac{\alpha}{ \big( 1 + \sum_{k \neq j^{\star}} e^{- \alpha  \varepsilon_k} \big)^2}
		\Bigg[ 2 \sum_{j \neq j^{\star}}  e^{-\alpha \varepsilon_j}
		+ \sum_{i \neq j^{\star}} \sum_{j \neq i}  e^{- \alpha( \varepsilon_i +\varepsilon_j)} \Bigg]
		,
		\end{align*}
		and since $\varepsilon_{j}$ is strictly positive for all $j \neq j^{\star}$, in the limit as $\alpha \rightarrow \infty$ the negative exponentials $e^{-\alpha \varepsilon_{j}}$ dominate $\alpha$, and $f_{\alpha} \rightarrow 0$.
	\end{proof}
	
	\begin{lemma}\label{lemma:g_converges_to_dirac}
		Consider the function
		\begin{equation*}
		g_{\alpha}(x) = \alpha \frac{e^{\alpha x}}{(1 + e^{\alpha x})^2}.
		\end{equation*}
		We have that $g_{\alpha}(x) \dx \rightarrow \delta_0$ as $\alpha \rightarrow \infty$ in the sense of weak convergence of measures, where $\delta_0$ denotes the Dirac mass at 0.
	\end{lemma}
	\begin{proof}
		First, note that for all $\alpha >0$
		\begin{equation*}
		\int_{\R} g_{\alpha}(x) \dx = 1.
		\end{equation*}
		Consider any continuous bounded function $\varphi(x) \in C_b(\R)$. For all $\varepsilon > 0$ there exists a $\delta>0$ such that
		\begin{align*}
		\Big| \int_{\R} \varphi(x) g_{\alpha}(x) \dx - \varphi(0) \Big|
		&\le \int_{\R} g_{\alpha}(x) \big| \varphi(x) -\varphi(0) \big|  \dx  \\
		&\le \varepsilon \int_{-\delta}^{\delta} g_{\alpha}(x) \dx
		+ \int_{-\infty}^{-\delta} \alpha e^{- \alpha \delta }  \big| \varphi(x) -\varphi(0) \big| \dx \\
		&\quad+ \int_{\delta}^{\infty} \frac{\alpha}{1 + e^{\alpha \delta }}  \big| \varphi(x) -\varphi(0) \big| \dx \\
		&\le \varepsilon + \int_{-\infty}^{-\delta} \alpha e^{- \alpha \delta }  \big| \varphi(x) -\varphi(0) \big| \dx 
		+ \int_{\delta}^{\infty} \frac{\alpha}{1 + e^{\alpha \delta }}  \big| \varphi(x) -\varphi(0) \big| \dx.
		\end{align*}
		Taking the limit as $\alpha \rightarrow \infty$, the last two integrals go to 0. Hence the limit of the left-hand side is less then $\varepsilon$ for any $\varepsilon >0$, so we are done.
	\end{proof}
	
	We now move on to studying the dynamics of the maximum of a family of continuous semimartingales driven by a common Brownian motion. Note that we specifically deal with semimartingales which have absolutely continuous finite variation part, which implies that their local times have a bicontinuous modification in $t \in \R^+$ and $a \in \R$ (see Definition~\ref{def:local_time}). This is the case for all stochastic processes which can be written as the solution of an It\^o SDE with integrable drift and stochastic term driven by a semimartingale with absolutely continuous finite variation.
	
	Consider a family of $\R$-valued continuous semimartingales $\mathbf{X}_t = \{ X^i_t\}_{i = 0}^n$ with dynamics
	\begin{equation}\label{eq:example_semimart}
	\di X^i_t = b_t^i \dt + \sigma_t^i \di B_t,
	\end{equation}
	where $b^i_t$ and $\sigma^i_t$ are (real, predictable, stochastically integrable) drift and diffusion coefficients for all $i = 0, \dots, n$, and $B_t$ is a standard Brownian motion.

	We apply It\^o's Lemma to derive the dynamics of $LSE_{\alpha}(\boldsymbol{X}_{\cdot})(t)$ as
	\begin{align}
	&\di LSE_{\alpha}(\boldsymbol{X}_{\cdot})(t)\\
	&= \sum_{i = 0}^n \frac{e^{\alpha X^i_t} }{\sum_{k} e^{\alpha X^k_t}} \di X^i_t
	+ \frac{1}{2} \sum_{i=0}^n \sum_{j=0}^n \alpha e^{\alpha X_t^i} \bigg( \frac{\delta_{ij}}{\sum_{k} e^{\alpha X^k_t}}
	- \frac{e^{\alpha X^j_t}}{(\sum_{k} e^{\alpha X^k_t})^2}   \bigg) \di \langle X^i_{\cdot}, X^j_{\cdot}  \rangle_t \nonumber \\
	&= \sum_{i = 0}^n \frac{e^{\alpha X^i_t} }{\sum_{k} e^{\alpha X^k_t}} b_t^i \dt
	+ \sum_{i = 0}^n \frac{e^{\alpha X^i_t} }{\sum_{k} e^{\alpha X^k_t}} \sigma_t^i \di B_t
	+ \frac{1}{2} f_{\alpha}(\mathbf{X}_{\cdot}, \boldsymbol{\sigma}_{\cdot})(t) \dt, \label{eq:LSE_X}
	\end{align}
	where we have written $\delta_{ij}$ for the Kronecker delta and defined the function
	\begin{equation}\label{eq:f_alpha_function}
	f_{\alpha}(\mathbf{X}_{\cdot}, \boldsymbol{\sigma}_{\cdot})(t) := \alpha \sum_{i} \sum_{j \neq i}  \frac{ e^{\alpha(X^i_t + X^j_t)}}{\big( \sum_k e^{\alpha X^k_t}\big)^2} \,\big( (\sigma_t^i)^2 - \sigma_t^i \sigma_t^j \big).
	\end{equation}
	
	We rewrite \eqref{eq:LSE_X} in integral form as follows, for all $s \le t$,
	\begin{align}
	LSE_{\alpha}(\boldsymbol{X}_{\cdot})(t)
	&= LSE_{\alpha}(\boldsymbol{X}_{\cdot})(s)
	+ \int_s^t S^{arg}_{\alpha}(\mathbf{X}_{\cdot}, \mathbf{b}_{\cdot})(r)\di r
	+ \int_s^t S^{arg}_{\alpha}(\mathbf{X}_{\cdot}, \boldsymbol{\sigma}_{\cdot})(r)\di B_r \nonumber \\
	&\quad+ \frac{1}{2} \int_s^t f_{\alpha}(\mathbf{X}_{\cdot}, \boldsymbol{\sigma}_{\cdot})(r) \di r \label{eq:LSE_X_integral}.
	\end{align}
	We are interested in the limit of the above when we send $\alpha$ to infinity. For each time $t$, define the argmax of $\mathbf{X}_t$ by $\mathcal{I}_t = \{ j \in \mathbf{N} \, : \, X^j_t \ge X^i_t \:\, \forall i \in \mathbf{N} \}$. Since $S^{arg}_{\alpha}(\mathbf{X}_{\cdot}, \mathbf{b}_{\cdot})(r) \le \max_{i} b_t^i$, and $b^i$ is integrable for all $i$ by assumption, we can apply dominated convergence to yield
	\begin{equation*}
	\lim_{\alpha \rightarrow \infty} \int_0^t S^{arg}_{\alpha}(\mathbf{X}_{\cdot}, \mathbf{b}_{\cdot})(s)\ds = \int_0^t \frac{1}{|\mathcal{I}_s|} \sum_{j \in \mathcal{I}_s} b^j_s \ds.
	\end{equation*}
	Similarly, $\max_{i} \sigma_t^i$ is integrable against $B_t$, so we can apply dominated convergence for stochastic integrals and get
	\begin{equation*}
	\lim_{\alpha \rightarrow \infty} \int_s^t S^{arg}_{\alpha}(\mathbf{X}_{\cdot}, \boldsymbol{\sigma}_{\cdot})(r)\di B_r = \int_s^t \frac{1}{|\mathcal{I}_r|} \sum_{j \in \mathcal{I}_r} \sigma^j_r\di B_r.
	\end{equation*}

	The last integral on the right hand side of \eqref{eq:LSE_X_integral} is trickier to deal with. 
	
	\begin{prop}\label{prop:local_time_bound}
		Consider a family of continuous semimartingales $\mathbf{X}_t = \{ X^i_t\}_{i = 0}^n$ with dynamics given by \eqref{eq:example_semimart}. Let $f_{\alpha}$ be defined as in \eqref{eq:f_alpha_function}. Then for all $s \le t$
		\begin{equation*}
		\lim_{\alpha \rightarrow \infty} \int_s^t f_{\alpha}(\mathbf{X}_{\cdot}, \boldsymbol{\sigma}_{\cdot})(r) \di r
		\le \sum_{i} \sum_{j > i} \Big(L^{0}_t(X^i_{\cdot} - X^j_{\cdot}) - L^{0}_s(X^i_{\cdot} - X^j_{\cdot}) \Big) \quad \text{a.s.}
		\end{equation*} 
	\end{prop}
	\begin{proof}
		Exploiting symmetry, we start by rewriting $f_{\alpha}(\mathbf{X}_{\cdot}, \boldsymbol{\sigma}_{\cdot})(t)$ as
		\begin{equation*}
		f_{\alpha}(\mathbf{X}_{\cdot}, \boldsymbol{\sigma}_{\cdot})(t)
		= \frac{1}{2} \alpha \sum_{i} \sum_{j \neq i}  \frac{ e^{\alpha(X^i_t + X^j_t)}}{\big( \sum_k e^{\alpha X^k_t}\big)^2} \,\big( \sigma_t^i - \sigma_t^j\big)^2,
		\end{equation*}
		and hence note that the last integral on the right-hand side of $\eqref{eq:LSE_X_integral}$ is always positive. Moreover, with $g_{\alpha}$ as in Lemma~\ref{lemma:g_converges_to_dirac},
		\begin{align*}
		\frac{ \alpha e^{\alpha(X^i_t + X^j_t)}}{\big( \sum_k e^{\alpha X^k_t}\big)^2}
		&= \frac{\alpha}{2 + e^{\alpha (X^i_t - X^j_t)} + e^{\alpha (X^j_t - X^i_t)} + \sum_{k \neq i,j} \sum_{l \neq i,j} e^{\alpha (X^k_t + X^l_t - X^i_t - X^j_t)}} \\
		&\le \frac{\alpha e^{\alpha(X^i_t - X^j_t)}}{(1 + e^{\alpha(X^i_t - X^j_t)})^2}
		= g_{\alpha}(X^i_{\cdot} - X^j_{\cdot})(t).
		\end{align*}
		The occupation time formula (see e.g. \cite[Chapter~6,~Corollary~1.6]{revuz_yor}) yields
		\begin{align*}
		\lim_{\alpha \rightarrow \infty} \int_s^t f_{\alpha}(\mathbf{X}_{\cdot}, \boldsymbol{\sigma}_{\cdot})(r) \di r
		&\le \lim_{\alpha \rightarrow \infty} \frac{1}{2} \int_s^t \sum_{i} \sum_{j \neq i} g_{\alpha} (X^i_{\cdot} - X^j_{\cdot})(r) \big( \sigma_r^i - \sigma_r^j\big)^2 \di r \\
		&= \sum_{i} \sum_{j \neq i} \lim_{\alpha \rightarrow \infty} \frac{1}{2} \int_s^t  g_{\alpha} (X^i_{\cdot} - X^j_{\cdot})(r) \di \langle X^i_{\cdot} - X^j_{\cdot} \rangle_r \\
		&= \sum_{i} \sum_{j \neq i} \lim_{\alpha \rightarrow \infty} \frac{1}{2} \int_{\R}  g_{\alpha} (z) \Big( L_t^z( X^i_{\cdot} - X^j_{\cdot}) - L_s^z( X^i_{\cdot} - X^j_{\cdot}) \Big)  \di z \\
		&= \frac{1}{2}\sum_{i} \sum_{j \neq i}  \Big( L^{0}_t(X^i_{\cdot} - X^j_{\cdot}) - L^{0}_s(X^i_{\cdot} - X^j_{\cdot})  \Big)
		\end{align*}
		almost surely, 	where the final equality relied on the weak convergence of $g_{\alpha}(z) \di z$ to a Dirac mass at 0 by Lemma \ref{lemma:g_converges_to_dirac}. Note that the $g_{\alpha}(z)$ are not compactly supported (compare with Definition~\ref{def:local_time}), but this is fine since the local time $ L_t^z( X^i_{\cdot} - X^j_{\cdot})$ is bounded in $z$ a.s.~(see Barlow and Yor \cite[Corollary~5.2.2]{barlow_yor_82}).
	\end{proof} 
	
	\section{Numerical experiments}\label{app:numerics}
	
	In this appendix we provide some details about the simulations for the plots in Figure~\ref{fig:hilber_error_2d_100d}. For the sake of comparison between the different dimensions, we give the rate matrix $Q$ a fixed structure, and keep the contraction coefficient constant across dimensions.
	
	For $n = 2, 20, 50, 100$, we take the signal process $X$ to be a Markov chain on $n+1$ states $\{0, \dots, n \}$ such that if $X$ is at state $i$ at some time $t$, it will be equally likely to jump to state $i+1$ or $i-1$, while it will only jump to state $j \neq i \pm 1$ with much lower probability. In other words, the chain switches quickly between a state and its two closest neighbours, but it only mixes slowly with the states further away. We let the jump rate from state $i$ grow with the dimension of the chain: for $n \ge 3$, we set the off-tridiagonal entries of $Q = (q_{ij})$ to be 1, the upper and lower diagonals to be $n+1$, and therefore the diagonal to be $-3n$, i.e.
	\begin{equation*}
	(q_{ij}) = \left\{
	\begin{array}{ll}
	n+1, & \mathrm{if} \: j \equiv i\pm 1 \pmod n, \\
	-3n, & \mathrm{if} \: j = i, \\
	1,   & \mathrm{otherwise.}
	\end{array}
	\right.
	\end{equation*}
	For $n = 2$, we simply take $Q$ to be the symmetric matrix with $-2$ on the diagonal and 1 in the other entries. By fixing $Q$ this way for all $n$, we have that the contraction rate from Theorem~\ref{thm:contraction} is $\lambda = 2$, and does not change across all dimensions.
	
	The chain $X$ has uniform stationary distribution, which we denote by $\mu = (\frac{1}{n}, \dots, \frac{1}{n})$. This is the point at the centre of the probability simplex $\mathcal{S}^n$. We take $\mathrm{law}(X_0) = \mu$. Finally, we set the sensor function $h \in \R^{n+1}$ to be a randomly generated vector such that, for each $i \in \mathbf{N}$, $h^i = z_i + x_i$, where $z_i$ is a random integer in $\{-10, \dots, 10\}$, and $x_i$ is a realization of a uniform random variable in $[0,1]$.
	
	The initial condition for the optimal filter $\pi_t$ is $ \pi_0 = \mu$. The `wrong' Wonham filter $\tilde \pi_t$ is initialized at $\nu \neq \mu$: to determine $\nu$, we perturb $\mu$ by adding/subtracting $\frac{1}{2}\min_i \mu_i$ from all the components of $\mu$ according to $n+1$ independent Bernoulli random variables, and renormalizing.
	
	Having fixed all these parameters, we generate 300 sample paths for the signal and the observation processes, and compute the optimal and `wrong' Wonham filters by solving the Zakai equation (see e.g. \cite[Remark~3.26]{bain09}) with a simple Euler scheme and renormalizing after each step. We plot the realizations of the Hilbert error $\hilbert (\pi_t, \tilde \pi_t)$, together with the bounds from Theorem~\ref{thm:contraction} and Proposition~\ref{prop:pathwise_bound_stability_ode}. Note that the bounds from Proposition~\ref{prop:pathwise_bound_stability_ode} are path-dependent (to compute them we need to observe $\tilde \pi_t$), so for each realization of the Hilbert error we have corresponding realizations of the bounds from Proposition~\ref{prop:pathwise_bound_stability_ode}. They are also expressed as bounds for $\tanh (\hilbert(\pi_t, \tilde \pi_t)/4)$ (as opposed to $\hilbert(\pi_t, \tilde \pi_t)$). Taking $\arctanh$ on both sides of \eqref{eq:tanh_less_u_stability}, and multiplying by 4, yields that $\hilbert(\pi_t, \tilde \pi_t) \le 4 \arctanh(u_t)$, where $u_t$ solves \eqref{eq:ode_tanh_stability};~given the potential for the dynamics of $u$ to have very large Lipschitz coefficients, we use a tamed Euler scheme (see e.g.~Hutzenthaler, Jentzen and Kloeden \cite{hutzenthaler2012}) to solve the ODE numerically. Concavity and monotonicity of $\tanh$ yield $\hilbert (\pi_t, \tilde \pi_t) \le \hilbert (\mu, \nu) e^{- \int_0^t \tilde \lambda_s \ds}$ from \eqref{eq:pathwise_ineq_tanh};~we compute $\tilde \lambda_t = 2 \min_{i \neq k} \big( q_{ik}q_{ki} + \sum_{j \neq i,k} \tilde \pi_t^j \min \{ q_{ji}q_{ik}/\tilde \pi_t^k, q_{jk}q_{ki}/\tilde \pi_t^i \} \big)^{1/2}$ at each timestep and perform numerical integration to plot the bound.


	\newpage
	\bibliographystyle{plain}
	\bibliography{biblio}
	
\end{document}